%% file: main.tex
\documentclass[10pt]{article}

\usepackage[margin=1in]{geometry}
\usepackage{amsmath,amssymb,amsthm,mathtools}
\usepackage{bm}
\usepackage{enumitem}
\usepackage{microtype}
\usepackage{xcolor}
\usepackage{authblk} 
\usepackage{natbib}
\usepackage{ying}
\usepackage{ulem}
\usepackage{mathtools}
\usepackage{placeins}

\mathtoolsset{showonlyrefs}

\usepackage[colorlinks,
            linkcolor=red,
            anchorcolor=blue,
            citecolor=blue
            ]{hyperref}
\usepackage[nameinlink,capitalise,noabbrev]{cleveref}
\allowdisplaybreaks
\setlist[itemize]{leftmargin=1.5em,itemsep=0.25em,topsep=0.4em}
\setlist[enumerate]{leftmargin=1.8em,itemsep=0.25em,topsep=0.4em}

\newcommand{\cgrid}{{c_{\textnormal{grid}}}}
\newcommand{\RI}{{\textnormal{RI}}}
\newcommand{\att}{{\textnormal{ATT}}}

\title{When Does Inexact Matching Ensure \\  Balance and Inference without Adjustment?}
\author{Ying Jin \\  \texttt{yjinstat@wharton.upenn.edu}}
\affil{Department of Statistics and Data Science,  University of Pennsylvania}
\date{}

\begin{document}
\maketitle

\begin{abstract}
    One-to-one matching without replacement is a classical approach to constructing comparable treated and control samples in the design of observational studies. It pairs each treated unit with a distinct control while minimizing a covariate distance objective. With continuous covariates, the matched pairs generally remain inexact, which contributes to bias in downstream analysis. Its key theoretical properties, such as the resulting   imbalance between matched pairs and when it is negligible to support valid inference, remain unclear.
    In this paper, we analyze one-to-one matching based on $d$-dimensional, continuous covariates with a quadratic covariate-distance objective. 
    First, we find that when $d\leq 3$, under standard conditions on the propensity score ensuring abundant control samples near each treated sample, the imbalance (difference between within-group averages) is root-$n$ negligible uniformly over the family of smooth functions with a common first- and second-order derivative bound. However, such balance is subject to a dimension restriction, as we construct examples in which the imbalance is root-$n$ non-negligible when $d=4$ and dominates root-$n$ rate when $d>4$. 
    Second, we show that when $d\leq 3$, the matched design allows valid Wald-type and bootstrap inference for the average treatment effect on the treated, distributional treatment effects, and quantile treatment effects. Thus, the same outcome-blind matched design supports various downstream inferences without having to tailor the design to the targets. Finally, paired randomization inference based on the matched design is asymptotically valid in the super-population sense for $d\leq 3$ but can fail when $d=4$. We corroborate the theoretical results with numerical experiments.
\end{abstract}

\input{01_intro}

\input{03_balance}

\input{04_wald}

\input{05_randomized}

\input{06_simulations}

\input{02_literature}

\section*{Discussion}

In this paper, we study when inexact one-to-one matching without replacement, on its
own, yields balance sufficient for valid inference. For the quadratic objective
with $d \le 3$ continuous covariates and abundant local control samples, the empirical 
imbalance between the treated and matched control samples is $o_P(n^{-1/2})$
uniformly over smooth functions, even though individual pairs may remain at distance
of order $n^{-1/d}$. 
As a consequence, a single outcome-blind matched design supports Wald-type and bootstrap inference for the ATT, distributional and quantile treatment effects, and paired randomization tests, without the need to tailor the design to the target or to fit outcome models. We also show that the dimension restriction is sharp: for $d = 4$ the imbalance of a fixed smooth function is non-negligible at the root-$n$ scale,
and for $d > 4$ it dominates the root-$n$ scale, even under the same favorable conditions. 
The practical implication of this restriction is that we may rely on matching alone when few continuous covariates are matched inexactly. 
When there are discrete features, the matching could be performed exactly in these discrete features while being inexact in the continuous ones. 
Otherwise, our analysis suggests that one may need to reduce the dimension of the
matching variables or combine matching with other adjustments such as outcome regression.
Finally, our analysis relies on the quadratic structure of the objective, and whether the same results hold for the ordinary Mahalanobis distance, as conjectured by \citet{GR2023}, remains open as future work.

\section*{Acknowledgments}

Generative AI tools were used for language editing, code assistance, and checking mathematical derivations and numerical implementations. All theoretical arguments, numerical results, and final manuscript content were reviewed and verified by the author, who takes full responsibility for the work.

\bibliographystyle{apalike}
\bibliography{references}

\newpage 
\appendix 
\input{10_appendix}

\end{document}

%% file: 01_intro.tex

\section{Introduction}

Matching is widely used in the observational studies to construct comparable treated and control groups~\citep{rubin1973matching,rosenbaum1985constructing,rosenbaum2020design}. 
One canonical form is one-to-one matching without replacement, which pairs each treated unit with a distinct control using their pre-treatment covariates~\citep{rosenbaum1989optimal} while minimizing certain objective for the distance between matched pairs. 
In this process, the pairs are formed before the outcomes are examined, thereby separating the design and analysis of observational studies and promoting transparency and objectivity. 

One criterion of a good match is the matched pairs be close enough---and ideally identical---in the pre-treatment covariates to help reduce bias from differences in observed covariates~\citep{rubin1973matching}. However, with continuous covariates, the matching generally remains inexact, and the differences in the covariates can leave bias in the subsequent analysis. 
A second criterion is balance between the treated and matched control groups \citep{rosenbaum1989optimal}. Even when individual pairs are inexact, their differences may cancel across pairs, making the overall bias relevant in downstream inference smaller than the individual discrepancies suggest. 
The key question is thus whether such bias is negligible relative to the sampling uncertainty so valid inference is possible. 
Establishing this connection is challenging: because assigning a control to a treated unit can change the choices available to the others, the matching assignments need to be analyzed jointly across all samples. Although matching is intuitive, transparent, and simple to describe, the implications of inexact matching for the balance of the selected groups and the validity of downstream inference remain unclear. 

Existing theory provides important results for matching estimators, including matching without replacement and inference after matching~\citep{AI2012,AS2022}. There, a central requirement for proving inferential guarantees is that the approximation error (pairwise distance) be sufficiently small. This requirement typically requires a large amount of control samples, yet does not necessarily follow from that alone: \cite{Savje2022} shows that the shortage of controls in a local region of the covariate space can result in inconsistency. 
Even when the controls are abundant locally in the entire sample space, the magnitude of the matching bias still matters. 
For optimal pair matching based on ordinary Mahalanobis distance, \cite{GR2023} establish the validity of an unadjusted paired randomization test with one continuous covariate and conjecture an extension to a dimension of three.  

In this paper, we study optimal one-to-one matching with a quadratic covariate-distance objective commonly used in the literature~\citep{rubin1980bias,rosenbaum1985constructing,bind2019bridging,zhang2023matching}. We work under an i.i.d.~super-population model where the covariates  are $d$-dimensional, bounded, continuously valued, and the dimension $d$ is viewed as fixed as $n\to \infty$. 
We assume a favorable setting for matching in which the covariates admit a density with respect to the Lebesgue measure on $[0,1]^d$ that is bounded above and away from zero. The propensity score is smooth, bounded away from zero and strictly below $1/2$ to ensure that control samples are abundant around every treated sample, thereby separating the key analysis from the complications pointed out by~\cite{Savje2022}. 
In this regime, the treated and control sample sizes grow at the same order, so we do not require the control reservoir to grow extremely large~\citep{AI2012,AS2022}. 

\subsection{Preview of results}

Our main result concerns the imbalance of the matched sample over smooth functions. Let $\{T_i\}_{i=1}^{n_1}$ and $\{C_j\}_{j=1}^{n_0}$ be the treated and control covariates, respectively, both generated under a super-population model (Section~\ref{subsec:setup}), and let the learned matching be represented by an injection $\hat\sigma_n \colon \{1,\dots,n_1\}\to \{1,\dots,n_0\}$ which matches treated sample $i$ to control sample $\hat\sigma_n(i)$, $i=1,\dots,n_1$. In Section~\ref{sec:balance}, we show that, when $d\leq 3$, the minimizer $\hat\sigma_n$ of the quadratic matching objective obeys 
\@\label{eq:intro_balance}
\sup_{g\in \cF_L} \bigg| \frac{1}{n_1}\sum_{i=1}^{n_1} \big\{g(T_i) - g(C_{\hat\sigma_n(i)}) \big\} \bigg| = o_P(n^{-1/2}),
\@
where $\cF_L$ is the class of twice-continuously-differentiable functions with uniformly bounded first- and second-order derivatives. Notably, the functions in this class are not fitted or included as constraints in the matching problem: the procedure simply minimizes the total quadratic distance. The result also allows the matrix defining the distance (such as that in the Mahalanobis distance) to be estimated in root-$n$ rate, and allows $\hat\sigma_n$ to be an approximate solution with a small  optimization gap. 

The uniform balance~\eqref{eq:intro_balance} for $d\leq 3$ appears to respect a sharp dimension restriction. We further construct an example in which the imbalance in one function $g$ is non-negligible at the root-$n$ scale for $d=4$, and dominates the root-$n$ scale when $d>4$, even though all the stated favorable conditions hold. 
This result partially echoes the conjecture in~\cite{GR2023} albeit using the quadratic objective. 

This balance result provides a common basis for various downstream inferences based on one matched design. 
We study a variety of population-level inference tasks in Section~\ref{sec:wald_inference}. 
Under suitable outcome conditions, for $d\leq 3$, we establish asymptotic normality and consistent variance estimation for the population average treatment effect on the treated (ATT), as well as valid pairwise bootstrap inference without recomputing the matching. To demonstrate the breadth of inferential guarantees, we show the matched design can analogously support pairwise bootstrap inference for  joint confidence bands for the counterfactual distribution functions,  distributional treatment effects, and  quantile treatment effects. 

In Section~\ref{sec:rand_inference}, we turn to the randomization-based inference commonly used after a matched design. 
We find that for $d\leq 3$, the paired difference-in-means randomization test is asymptotically valid under the sharp and weak null in the super-population sampling sense. 
A dimension boundary is present, as we show that for $d=4$, the randomization test for the weak null  can violate the type-I error at any nominal level $1-\alpha$. 

Together, these results justify the use of one outcome-blind matched design across several analyses for low-dimensional problems. Under the stated conditions, the matching and the inferential procedures studied in this paper do not require outcome regression or propensity-score estimation. While each inferential target requires distinct comparisons of the matched pairs (such as the paired averages in mean effect estimation and the empirical CDFs for distributional effect estimation), the uniform balance alone automatically controls the discrepancies in the variety of outcome-related functions relevant in each analysis, without knowing these functions in advance or involving them in constructing the matching. It remains in the design stage of the observational study but retains strong inferential guarantees when the feature dimension is small. 

Finally, we corroborate our theory with numerical demonstrations in Section~\ref{sec:simu}. In simulations across ATT, distributional and quantile treatment effects, and randomization tests, we observe robust coverage of confidence intervals/bands when $d\leq 3$, systematic undercoverage at $d=4$, and deteriorating coverage for $d>4$. We include a detailed discussion on the related literature in Section~\ref{sec:literature}.

\subsection{Problem setup and notations}
\label{subsec:setup}

We assume a super-population setting where the full sample $\{(X_i,Z_i,Y_i(1),Y_i(0))\}_{i=1}^n$ are i.i.d.~from a common distribution $\PP$, and we observe the i.i.d.~triplets $\{(X_i,Y_i,Z_i)\}_{i=1}^n$, with binary treatment $Z_i\in \{0,1\}$, pre-treatment covariates $X_i\in [0,1]^d$, and outcome $Y_i\in \RR$.  
Throughout the paper, we assume standard SUTVA and unconfoundedness, so that $\{Y(1),Y(0)\}\indep Z\given X$ under $\PP$, and $Y_i=Y_i(Z_i)$. 
The propensity score is denoted as $e(x)=\PP(Z_i=1\given X_i=x)$. 
In addition, we make the following assumptions on the covariate density and strictly upper bounded propensity scores to ensure a clean and favorable setting for matching.

\begin{assumption}\label{assump:obs-sample}
    Assume the i.i.d.~super-population model with SUTVA and unconfoundedness, so that $\{Y(1),Y(0)\}\indep Z\given X$ under $\PP$, and $Y_i=Y_i(Z_i)$. 
    In addition, suppose $X$ has marginal density $f$ with respect to the Lebesgue measure on $\Omega=[0,1]^d$. Suppose $f(\cdot)$ and $e(\cdot)$ have $C^1$ extensions to
a neighborhood of $\Omega$. For fixed 
constants $\underline f,\overline f,\eta_0,\eta_1$, we have 
$0<\underline f\le f(x)\le\overline f<\infty$ and  
$\eta_0\le e(x)\le\tfrac12-\eta_1$,
for all $x\in\Omega$. 
\end{assumption}

The condition $\eta_0\le e(x)\le\tfrac12-\eta_1$ is also assumed in \cite{GR2023}, which ensures sufficiently abundant control samples in every local region to rule out the complication mentioned in~\cite{Savje2022}, so that it is eventually possible to find a control for
each treated observation. 
Under this assumption, the event $\cI_n=\{2\le n_1\le n_0\}$ happens with probability tending to one. 

\begin{remark}
    The super-population model used in this paper follows that of~\cite{GR2023} but differs from some common alternatives in the literature, such as the design-based inference where $\{(X_i,Y_i(1),Y_i(0))\}_{i=1}^n$ are viewed as fixed and the sole source of randomness is the treatment assignment. In that regime, the randomness in the matching design arises from the treatment assignment only, and the fixed covariate values makes the analysis even more challenging. In contrast, the super-population model enables the analysis of large-sample behavior via concentration inequalities under transparent conditions. Our setting ensures $n_0$ and $n_1$ are of the same order, in contrast to the requirements in~\cite{AI2012} that $n_1^r/n_0=O(1)$ for some $r>d$ so the control reservoir grows much faster than the $n_1$. 
\end{remark}

For notational convenience, we re-index the treated and control covariates as $\{X_i\}_{Z_i=1}  =\{T_i\}_{i=1}^{n_1}$ and $\{X_i\}_{Z_i=0} = \{C_j\}_{j=1}^{n_0}$, where $n_1=\sum_{i=1}^n Z_i$ and $n_0=n-n_1$ are the treated and control group sizes.

We focus on the one-to-one matching problem with a quadratic objective. 
Formally, consider an injection $\sigma\colon \{1,\dots,n_1\}\to \{1,\dots,n_0\}$, and let $\cA_{n_1,n_0}$ be the collection of all such injections. For a fixed symmetric positive-definite matrix $A_0\in \RR^{d\times d}$, we define the matching objective 
\@\label{eq:def_matching_obj}
{Q}_n(\sigma;A_0) = \frac{1}{2n_1}\sum_{i=1}^{n_1} (T_i-C_{\sigma(i)})^\top A_0 (T_i-C_{\sigma(i)}),\quad \sigma\in \cA_{n_1,n_0},
\@
and the global optimal solution 
\$
\sigma_n^\star \in \argmin_{\sigma\in \cA_{n_1,n_0}}Q_n(\sigma;A_0).
\$

Later on, we shall also extend the results so $A_0$ can be replaced by random, data-dependent matrices such as the inverse covariance matrix. 
Write $Q_n^\star = \min_\sigma Q_n(\sigma;A_0)$ as the minimized matching objective. 
In practice one may settle at an approximate solution $\hat\sigma_n$ with an objective gap $\Delta_n = Q_n(\hat\sigma_n;A_0) - Q_n^\star \geq 0$.  

\begin{remark}
    When $A_0=I$, the objective~\eqref{eq:def_matching_obj} is the total squared Euclidean distance. When $A_0$ is the (empirical) inverse covariance matrix, the objective is the total squared Mahalanobis distance. This differs from the Mahalanobis norm objective in \cite{GR2023}, but enables principled theoretical analysis and is broadly used in practice~\citep{bind2019bridging,zhang2023matching}.
\end{remark}

\paragraph{Notations.} We close the section by stating the notations used throughout this paper. 
For a vector $v$, we denote it Euclidean norm by $\|v\|$. For a matrix
$M$, let
$\|M\|_{\mathrm{op}}=\sup_{\|v\|=1}\|Mv\|$
be its operator norm; when $M$ is symmetric,
$\lambda_{\min}(M)$ and $\lambda_{\max}(M)$ denote its smallest
and largest eigenvalues. We write $I_d$ for the $d\times d$
identity matrix. For a scalar function $g$ on a set $S$, let
$\|g\|_S=\sup_{x\in S}|g(x)|$, which is also written as $\|g\|_\infty$ when no confusion arises. The gradient and Hessian of $g$ are denoted by
$\nabla g$ and $\nabla^2 g$, respectively. For a square-integrable
real random variable $W$, we denote its $L_2$-norm as 
$\|W\|_2=\{\EE(W^2)\}^{1/2}$.
We use $\ind\{\cdot\}$ for an indicator and
$x_+=\max\{x,0\}$ for the positive part.
For positive deterministic sequences $a_n,b_n$, write
$a_n\lesssim b_n$, or equivalently $a_n=O(b_n)$, if
$a_n\le Cb_n$ for some constant $C<\infty$ and all sufficiently
large $n$. Write $a_n\asymp b_n$ if both $a_n\lesssim b_n$ and
$b_n\lesssim a_n$, and $a_n=o(b_n)$ if $a_n/b_n\to0$.
For a positive sequence $a_n$, the notation
$U_n=O_P(a_n)$ means that $\|U_n\|/a_n$ is bounded in probability,
whereas $U_n=o_P(a_n)$ means that $\|U_n\|/a_n\to 0$ in probability;
the norm is chosen according to whether $U_n$ is a scalar,
vector, matrix, or function. We use $\to_P$ for convergence
in probability and $\Rightarrow$ for weak
convergence. Unless stated otherwise, all limits are taken as
$n\to\infty$ with $d$ fixed. Constants denoted by $c$ or $C$
may change from line to line and do not depend on the sample sizes;
their dependence on other quantities is specified where needed.

%% file: 03_balance.tex

\section{Uniform balance from optimal quadratic matching}
\label{sec:balance}

We are interested in the balance between the matched groups. 
For a function $g\colon \cX\to \RR$, define 
\@\label{eq:def_imbalance_g}
B_{n,\sigma}(g) = \frac{1}{n_1}\sum_{i=1}^{n_1} \big\{ g(T_i) - g(C_{\sigma(i)})\big\}.
\@
For a fixed $L\in (0,\infty)$, we define $\cF_L$ as the set of functions with a $C^2$ extension on $\Omega$ and whose function values and all first-order and second-order partial derivatives are bounded in absolute value by $L$ on $\Omega$. 

\subsection{Uniform balance for \texorpdfstring{$d\leq 3$}{d <=3}}
\label{subsec:uniform_balance}

Our first result is the uniform balance ensured by the matched design in Theorem~\ref{thm:observational}, whose proof is in Appendix~\ref{app:sec_proof_uniform_balance}. 
For completeness, we define the corresponding optimization gap and balance measure to zero on $\cI_n^c$, which does not affect the high-probability result.  

\begin{theorem}[Uniform balance]
\label{thm:observational}
Suppose  $d\le3$, and Assumption~\ref{assump:obs-sample} holds. 
Let $\hat{A}$ be a data-dependent symmetric matrix such that $\norm{\hat{A}-A_0}_{\textnormal{op}} = O_P(1/\sqrt{n})$ for some fixed symmetric positive-definite matrix $A_0$. 
On the event $\cI_n$, let
$\hat\sigma_n\in\mathcal A_{n_1,n_0}$ be a measurable injection with optimization gap  
$\xi_n = Q_n(\hat\sigma_n;\hat{A})
-\min_{\sigma\in\mathcal A_{n_1,n_0}}Q_n(\sigma;\hat{A}) = o_P(1/n)$. 
Then for every
fixed $L\in(0,\infty)$, as $n\to \infty$, 
\@\label{eq:observational-balance}
\sqrt n\sup_{g\in\mathcal F_L}
|B_{n,\hat\sigma_n}(g)| = o_P(1).
\@
\end{theorem}

The convergence in probability stated in~\eqref{eq:observational-balance} concerns the random imbalance as a function of the covariates. Such randomness arises from the super-population model, where i.i.d.~sampling of the covariates provides the probabilistic basis and needed structure for controlling the imbalance with high probability. 

The condition $\|\hat{A}-A_0\|_{\textnormal{op}}$ allows the use of Mahalanobis distance where $\hat{A}= \hat\Sigma^{-1}$ for the empirical covariance matrix $\hat\Sigma$ which converges to the population covariance matrix in the required rate if tha latter has lower-bounded eigenvalues. 

Theorem~\ref{thm:observational} implies that the matched control sample approximates the treated group in the sense of empirical measure. Namely, on $\mathcal I_n=\{2\le n_1\le n_0\}$,  consider the empirical measures $\hat{P}_1 = n_1^{-1}\sum_{i=1}^{n_1} \delta_{T_i}$ and $\hat{Q}_0 = n_1^{-1} \sum_{i=1}^{n_1} \delta_{C_{\hat\sigma_n(i)}}$. It states that $\sqrt{n}\|\hat{P}_1-\hat{Q}_0\|_{\cF_L} = o_P(1)$, where we define $\|\mu\|_{\cF_L} = \sup_{g\in \cF_L} |\int g \ud \mu|$.

\begin{remark}[Fixed-function moment bounds]
The imbalance control in~\eqref{eq:observational-balance} is stated in probability. The proof of Theorem~\ref{thm:observational} (especially the intermediate result Theorem~\ref{thm:balance}) also yields stronger moment bounds for any fixed function $g\in \cF_L$ 
and exact optimizer under a fixed symmetric positive definite matrix $A_0$. 
For all sufficiently large $n$, recalling that $\sigma_n^\star$ is the exact optimizer, we have 
\[
\left|\EE [B_{n,\sigma_n^\star}(g)]\right|
\le C_g\Gamma_n,
\qquad
\EE\!\left[B_{n,\sigma_n^\star}(g)^2\right]
\le C_g\Big(\Gamma_n^2+\frac{\Gamma_n}{\sqrt n}\Big),
\qquad
\Gamma_n\asymp n^{-2/d}(\log n)^{1+2/d},
\]
where the constant $C_g$ does not depend on $n$. 
Thus, for $d\leq 3$, the imbalance $B_{n,\hat\sigma_n}(g)$ is root-$n$ negligible in the $L^2$-norm. See Proposition~\ref{prop:observational-moments} for a formal statement.  
\end{remark}

We now discuss our analysis techniques. A common way in the literature to control the approximation error left by matching is
to bound each pair's contribution in absolute value. For a smooth
function $g$,
\@\label{eq:first_order}
|B_{n,\sigma}(g)|
\le
\frac1{n_1}\sum_{i=1}^{n_1}
|g(T_i)-g(C_{\sigma(i)})|
\le
\bigg(\sup_{x\in\Omega}\|\nabla g(x)\|\bigg)
\frac1{n_1}\sum_{i=1}^{n_1}\|T_i-C_{\sigma(i)}\|.
\@
This bound gives root-$n_1$ negligible imbalance whenever the average
matching distance is $o_P(n_1^{-1/2})$, as required by the
approximation conditions in
\citet[Theorem~1]{AI2012} and \citet[Assumption~3]{AS2022}. However, this requirement is difficult to fulfill, which leads to requiring a large control pool that grows much faster than $n_1$. 

The key insight of our analysis is that the first-order expansion may neglect the cancellation between positive
and negative pairwise differences.
To see how this cancellation may happen, by Taylor expansion,
\[
B_{n,\sigma}(g)
=
\frac1{n_1}\sum_{i=1}^{n_1}
\nabla g(T_i)^\top\{T_i-C_{\sigma(i)}\}
+R_{n,\sigma}(g),
\qquad
|R_{n,\sigma}(g)|
\le
\frac{dL}{2n_1}\sum_{i=1}^{n_1}
\|T_i-C_{\sigma(i)}\|^2.
\]
The remainder is controlled by the quadratic matching objective.
The leading sum, however, keeps the signs of the differences, and bounding each term by the magnitude in~\eqref{eq:first_order} could be loose. The main step in our analysis
is to control this signed sum for the globally optimal assignment
under the stated sampling assumptions.
For a fixed function $g$ and a fixed matrix $A_0$, our proof (Appendix~\ref{app:sec_proof_uniform_balance}) studies
the optimal value of the tilted objective  
$Q_n(\sigma;A_0)+tB_{n,\sigma}(g)$ for $t\in \RR$. Its derivative at $t=0$ equals the imbalance of
the original optimizer $\sigma_n^\star$ almost surely. 
We use integration by parts, as well as the Efron--Stein inequality to obtain bounds on the optimal values and translate them to bounds on the derivative. The E--S inequality is applicable because we prove a stability property of the optimal value under one-observation replacement (Lemmas~\ref{lem:stab_opt_value} and~\ref{lem:value-bounds}). 
We then use a standard chaining technique to extend to uniform balance over $\cF_L$. 

This distinction also connects to \citet[Proposition~1]{GR2023}, who uses the conditional matching bias to characterize the validity of  paired randomization test.
Our result establishes sufficient conditions for uniformly negligible
smooth-function imbalance (bias) under the quadratic objective.

\subsection{The dimension boundary: counterexample for \texorpdfstring{$d\geq 4$}{d >= 4}}
\label{subsec:dim}

The uniform balance result holds for $d\leq 3$. A natural question is whether this is the sharp dimension boundary. 
We now show that with $d\geq 4$, even for the exact optimizer $\sigma_n^\star$ for a fixed matching objective with $\hat A=A_0=I_d$ and $\xi_n=0$ can fail to yield $\sqrt{n}$-negligible imbalance.  The construction of the example and the proof of Theorem~\ref{thm:dimension-obs} are in Appendix~\ref{app:sec_proof_boundary}. 

\begin{theorem}\label{thm:dimension-obs}
For every fixed integer $d\ge4$, there exist fixed $C^\infty$ functions $f,e$ satisfying Assumption~\ref{assump:obs-sample} and a fixed $C^\infty$ function $g$ on a neighborhood of $\Omega$, with $0\le g\le1$ and $\sup_\Omega\|\nabla g\|\le1$, such that the exact minimizer $\sigma_n^\star$ under $A_0=I_d$ has the following properties. If $d=4$, there are constants $b_*,p_*>0$ such that
\begin{equation}\label{sharp:eq:iid-prob}
 \liminf_{n\to\infty}\PP\{\sqrt{n_1}\,B_{n,\sigma_n^\star}(g)\ge b_*\}\ge p_*.
\end{equation}
If $d>4$, then $\sqrt n B_{n,\sigma_n^\star}(g)\to+\infty$ in probability. The imbalance is  set to be zero on $\mathcal I_n^c$. 
\end{theorem}

In the construction, the expected imbalance is of order $n^{-2/d}$, while its variance is of order at most $n^{-1}$. The former coincides with the sampling-error scale at $d=4$ and decays more slowly when $d>4$. For matching with replacement, \cite{AI2006} pointed out the bias scale of $n^{-2/d}$, which coincides with ours in the matching without replacement setting. 
Finally, the single-$g$ bound clarifies that the failure of uniform balance is due to inexact matching, instead of the intrinsic sampling uncertainty due to taking the uniform over the function class for $d=4$ (which happens to be $\sqrt{n}$).

%% file: 04_wald.tex

\section{Wald-type and bootstrap inference from one matched design}
\label{sec:wald_inference}

A useful implication of the uniform balance in Theorem~\ref{thm:observational} is that one matched design provides negligible bias in a wide range of functions, so that the impact of inexact matching on various downstream inference is negligible. 
In this section, we show how the same matched design allows convenient Wald-type inference and Bootstrap approximation for various targets, including the average treatment effect on the treated (ATT) (Section~\ref{subsec:wald_att}), distribution effects (Section~\ref{subsec:wald_cdf}), and quantile treatment effects (Section~\ref{subsec:wald_quantile}). 

Throughout this section, for notational convenience, we denote the treated outcomes as $\{Y_{1,i}\}_{i=1}^{n_1}$ and the control outcomes as $\{Y_{0,j}\}_{j=1}^{n_0}$. Then, a matched design $\hat\sigma_n$ matches each treated sample $i$ to the control sample $\hat\sigma_n(i)$. 
We also consider a bootstrap approximation to the Wald-type inference. 

\paragraph{Bootstrapping the matched pairs.} We consider the bootstrap inference when resampling the matched pairs. Given the observed data, the original matching $\hat\sigma_n$, and bootstrap size $B\in \NN^+$, for each $b=1,\dots,B$, we draw $I_1^{*(b)},\ldots,I_{n_1}^{*(b)}$ independently and uniformly from $\{1,\ldots,{n_1}\}$ with replacement. 
The $b$-th bootstrap matched pairs  are 
$(Y_{1,I_i^{*(b)}}, Y_{0, \hat\sigma_n(I_i^{*(b)})} )$, $i=1,\dots, n_1$.  
Note that we do not re-compute the matched design.  

\vspace{1em}

Throughout, we assume the conditions in Theorem~\ref{thm:observational}, which we state for easier reference. The only additional requirement is that $\hat{A}$ must also use the outcome-blind information as the matching. 

\begin{assumption}\label{assump:match}
    Suppose $d\leq 3$. 
    Let $\cG_n = \sigma((X_i,Z_i)_{i=1}^n)$ be the $\sigma$-field used in constructing the matching. 
    Let $\hat{A}$ be a data-dependent symmetric matrix that is $\cG_n$-measurable such that $\norm{\hat{A}-A_0}_{\textnormal{op}} = O_P(1/\sqrt{n})$ for some fixed symmetric positive-definite matrix $A_0$. 
On the event $\cI_n=\{2\leq n_1\leq n_0\}$, let
$\hat\sigma_n\in\mathcal A_{n_1,n_0}$ be a $\cG_n$-measurable injection with optimization gap  
$\xi_n = Q_n(\hat\sigma_n;\hat{A})
-\min_{\sigma\in\mathcal A_{n_1,n_0}}Q_n(\sigma;\hat{A}) = o_P(1/n)$.  
\end{assumption}

Throughout, the probability and asymptotics concern the super-population sampling uncertainty, so the covariates, $\hat\sigma_n$, and outcomes are random, and the extra bootstrap sampling randomness when applicable.

\subsection{Average treatment effect on the treated}
\label{subsec:wald_att}

The population ATT is defined as $\theta_{\att} := \EE[Y(1)-Y(0)\given Z=1]$. Given a matched design $\hat\sigma_n$, we estimate with the difference-in-means estimator, with the Wald-type variance estimator 
\@\label{eq:def_Ri_theta_att}
\hat\theta_{\att} = \frac{1}{n_1}\sum_{i=1}^{n_1} R_i,\quad \text{where}\quad R_i := Y_{1i} - Y_{0\hat\sigma_n(i)},\quad \hat{V}_{\att} := \frac{1}{n_1-1} \sum_{i=1}^{n_1} (R_i-\hat\theta_{\att})^2.
\@ 
For bootstrap inference, we consider the studentized t-statistics. For $b=1,\dots,B_n$ and some fixed $B_n\in \NN^+$,
\@\label{eq:def_bootstrap_att_tstat}
T^{*(b)} = \frac{\hat\theta_{\att}^{*(b)} - \hat\theta_\att}{\sqrt{\hat{V}_\att^{*(b)}/n_1}},\quad \hat\theta^{*(b)}_\att = \frac{1}{n_1}\sum_{i=1}^{n_1} R_{I_i^{*(b)}} ,\quad 
\hat{V}_\att^{*(b)} = \frac{1}{n_1-1}\sum_{i=1}^{n_1} (R_{I_i^{*(b)}} - \hat\theta_\att^{*(b)} )^2.
\@

We make standard moment and smoothness conditions on the outcomes. Let $\mu_z(x)=\EE[Y(z)\given X=x]$. 
We write the conditional treatment effect (CATE) function as $\tau(x) = \mu_1(x)-\mu_0(x)$, and the conditional variance function $\sigma_z^2(x):=\EE[(Y(z)-\mu_z(x))^2\given x=x]$. Let $\Var_z(\cdot)$ denote the variance conditional on $Z=z$. 

\begin{assumption}\label{assump:outcome}
    The function $\mu_0$ has a $C^2$ extension near $\Omega$, $\mu_1$ is bounded and measurable, and $\sigma_0^2$ is continuous. 
    In addition, $\EE[(Y(z)-\mu_z(x))^4\given X=x]\leq M$ for some a common  constant $M>0$ for all $x$ and $z\in \{0,1\}$. Define $v_\tau=\Var_1\{\tau(X)\}$, $v_\epsilon=\E_1\{\sigma_1^2(X)+\sigma_0^2(X)\}$,  
and suppose $V_\att:=v_\tau+v_\epsilon>0$.
\end{assumption}

We show that when $d\leq 3$, the matched-pair difference-in-means estimator is asymptotically normal with variance $V_\att$ defined above, and permits valid Wald-type and bootstrap inference. 

\begin{theorem}[Population ATT inference]\label{thm:ATT}\label{thm:bootstrap}
Under Assumptions~\ref{assump:obs-sample},~\ref{assump:match}, and~\ref{assump:outcome}, as $n\to \infty$,
\begin{equation}\label{eq:ATT-main}
 \sqrt {n_1}(\hat\theta_\att-\theta_\att)\Rightarrow N(0,V_\att),
 \qquad \hat V_{\att}\to V_{\att}\quad\text{in probability}.
\end{equation} 
Hence the Wald-CI $\hat\theta_\att\pm z_{1-\alpha/2}\sqrt{\hat V_\att/{n_1}}$ has limiting coverage of $1-\alpha$, where $z_{1-\alpha/2}$ is the $1-\alpha/2$-th quantile of a standard normal distribution. Furthermore, the matched-pair bootstrap satisfies
\begin{equation}\label{eq:bootstrap}
 \sup_{x\in \RR}\left|\PP^*\{ 
 T^{*(1)}\leq x \}-\Phi(x) 
                  \right|\to 0
 \quad\text{in probability}.
\end{equation}
For bootstrap, let $B=B_n$ be a
deterministic sequence of positive integers satisfying
$B_n\to\infty$ as $n\to\infty$.
For the t-statistics in~\eqref{eq:def_bootstrap_att_tstat}, define $c_{q}^*$ as the $q$-th empirical quantile of $\{T^{*(b)}\}_{b=1}^B$. Then the bootstrap confidence interval $\mathrm{CI}_{n,1-\alpha}^{\mathrm{boot}} = [\hat\theta_{\att} - c_{1-\alpha/2}^*  \cdot(\hat{V}_\att/n_1)^{1/2}, \hat\theta_{\att} - c_{\alpha/2}^* \cdot (\hat{V}_\att/n_1)^{1/2}]$ has limiting coverage of $1-\alpha$. 
\end{theorem}

The proof of Theorem~\ref{thm:ATT} is in Appendix~\ref{app:subsec_att_proof}. 
Here we outline the intuitions for the root-$n$ Wald-type inference. Let $\tau(x)=\EE[Y(1)\given X=x]-\EE[Y(0)\given X=x]$ be the conditional average treatment effect. Then 
\$
\sqrt{n_1}(\hat\theta_{\att}-\theta_\att) = \frac{1}{\sqrt{n_1}} \sum_{i=1}^{n_1} \{\tau(T_i) - \theta_\att\}
+ \sqrt{n_1} B_{n,\hat\sigma_n}(\mu_0) + \frac{1}{\sqrt{n_1}}\sum_{i=1}^{n_1} (\varepsilon_{1,i} - \varepsilon_{0,\hat\sigma_n(i)}),
\$
where $\epsilon_{1,i}=Y_{1,i}-\mu_1(T_i)$ and $\epsilon_{0,j} = Y_{0,j}-\mu_0(C_j)$ are the independent, mean-zero residuals. 
The first term is a standard centered sample mean. 
The second term is the imbalance in the (yet-unknown) outcome model $\mu_0$, but the uniform balance in Theorem~\ref{thm:observational} ensures it is root-$n$ negligible as long as the outcome model belongs to that $C^2$ smooth class. 
Finally, since the design is outcome-blind and the matching is an injection without re-using control samples in the matching pairs, the third term is a summation over independent and mean-zero variables conditional on the matched design. 
Finally, the (design-conditional) central limit theorem for the third term can be combined with the first two to obtain the desired result~\eqref{eq:ATT-main}. 

\subsection{Counterfactual distributions and distributional effects}
\label{subsec:wald_cdf}

Apart from mean effects, distributional effects are widely used to assess the variation of the outcomes. We show that if the underlying distribution functions belong to the smooth function class, then the matched design also seamlessly supports the inference on distributional effects. 

Let $\EE_z$ denote the expectation under the treatment-conditional distribution for $z\in \{0,1\}$. We define the conditional distribution function and the cumulative distribution function (in the treated)
\$
H_z(y\given x) = \PP(Y(z)\leq y\given X=x),\quad F_{z,T}(y) = \EE_1[H_z(y\given X)].
\$
The distributional effect (in the treated) is given by 
\$
D(y) := F_{1,T}(y) - F_{0,T}(y).
\$
Here both $F_{1,T}$ and $F_{0,T}$ are defined by the covariate distribution in the treated group. 

\paragraph{Point-wise Wald inference.} With the matched design $\hat\sigma_n$, the distributional effect can be estimated by 
\@\label{eq:def_emp_cdf}
\hat{D}(y):= \hat{F}_1(y) - \hat{F}_0(y),\quad 
\text{where}\quad 
\hat{F}_1(y) = \frac{1}{n_1}\sum_{i=1}^{n_1} \ind\{Y_{1,i}\leq y\},\quad 
\hat{F}_0(y) = \frac{1}{n_1}\sum_{i=1}^{n_1}\ind\{Y_{0,\hat\sigma_n(i)} \leq y\},
\@
whose variance can be estimated by 
\$
\hat{V}_D(y) = \frac{1}{n_1-1} \sum_{i=1}^{n_1} \big\{ R_i(y) - \hat{D}(y)\big\}^2,\quad 
R_i(y) = \ind\{Y_{1,i}\leq y\} - \ind\{Y_{0,\hat\sigma_n(i)} \leq y\}.
\$
We will show that $(\hat{F}_1(y),\hat{F}_0(y))$ converges to a Gaussian process, so pointwise Wald-type inference is convenient and valid. However, inference on confidence bands is arguable more convenient with bootstrap, which we present next. 

\paragraph{Bootstrap inference.} The bootstrap inference will use the same resampling scheme introduced at the beginning of Section~\ref{sec:wald_inference}. For $b=1,\dots,B_n$, given the resampled pair indices $I_i^{*(b)}$, we define  
\@\label{eq:def_cdf_boot}
\hat{D}^{*(b)}  = \hat{F}_1^{*(b)}  - \hat{F}_0^{*(b)},\quad 
\hat F_1^{*(b)}(y)
  =\frac1{n_1}\sum_{i=1}^{n_1}\ind\{Y_{1,I_i^{*(b)}}\le y\},~
 \hat F_0^{*(b)}(y)
 &=\frac1{n_1}\sum_{i=1}^{n_1}
       \ind\{Y_{0,\hat\sigma_n(I_i^{*(b)})}\le y\}.
\@
To perform inference on the bands, we define 
\$
T_F^{*(b)} = \sqrt{n_1} \max_{z\in \{0,1\}} \sup_{y\in \RR} \big|\hat{F}_z^{*(b)}(y) - \hat{F}_z(y)\big|,\quad 
T_D^{*(b)} = \sqrt{n_1} \sup_{y\in \RR} \big| \hat{D}^{*(b)}(y) - \hat{D}(y)\big|. 
\$
For a fixed $\alpha\in (0,1)$ and $H\in \{F,D\}$, we define $c_H^*$ as the empirical $(1-\alpha)$-quantile of $\{T_H^{*(b)}\}_{b=1}^{B_n}$.  
Define the joint CDF confidence bands $\text{CI}_{z,n,1-\alpha}^{\text{joint}}$, $z\in \{0,1\}$ and the confidence band for the contrast $\text{CI}_{n,1-\alpha}^{\text{boot}}(y)$ as 
\$
\text{CI}_{z,n,1-\alpha}^{\text{joint}} &= [L_z(y),U_z(y)] , \quad L_z(y) =  \hat{F}_z(y)-c_F^*/\sqrt{n_1} ,\quad U_z(y) =  \hat{F}_z(y)+c_F^*/\sqrt{n_1} , \\ 
\text{CI}_{n,1-\alpha}^{\text{contrast}}(y) &= [\hat{D}(y) - c_D^*/\sqrt{n_1}, \hat{D}(y) + c_D^*/\sqrt{n_1}].
\$
By default, $\text{CI}_{z,n,1-\alpha}^{\text{joint}}$ is truncated to $[0,1]$ and $\text{CI}_{n,1-\alpha}^{\text{contrast}}(y)$ is truncated to $[-1,1]$. 
To state our results, we require the conditional CDF of the \textit{control} outcome to be smooth. 

\begin{assumption}\label{assump:cdf}
    The functions $H_0,H_1$ are jointly measurable, and there exists a fixed $L<\infty$ such that 
 $\{H_0(y\given\cdot):y\in\RR\}\subseteq\mathcal F_L$.
\end{assumption}

For a random variable $W\in \RR$ with CDF $G_W(\cdot)$, we say it has a \emph{regular $p$-quantile} if there exists some $c\in \RR$ such that $G_W(\cdot)$ is continuous at $c$, $G_W(c)=p$, and $G_W(c-\varepsilon)<p<G_W(c+\varepsilon)$ for every $\varepsilon>0$. 

We present Wald inference for $D(y)$ in a pointwise fashion since the corresponding asymptotic variance can be explicitly estimated. 
For confidence bands, we state the inferential results for bootstrap inference. 
To describe the joint limiting distribution, we let $(\BB_1,\BB_0)$ be centered Gaussian process with covariance kernel 
\@\label{dist:eq:kernel}
 K_{zz}(s,t)&=F_{z,T}(s\wedge t)-F_{z,T}(s)F_{z,T}(t),\notag\\
 K_{10}(s,t)&=\operatorname{Cov}_1\{H_1(s\given X),H_0(t\given X)\},
 \qquad K_{01}(s,t)=K_{10}(t,s).
\@
Theorem~\ref{thm:wald_cdf} shows the convergence of the point estimates $(\hat{F}_1,\hat{F}_0)$, as well as the validity of Wald CI and bootstrap confidence bands. The proof is in Appendix~\ref{app:subsec_thm_wald_cdf}. The key idea is that the uniform balance in Theorem~\ref{thm:observational} applied to the conditional CDFs makes the matching bias negligible in the CDF inference. 

\begin{theorem}
\label{thm:wald_cdf}
Under Assumptions~\ref{assump:obs-sample},~\ref{assump:match}, and~\ref{assump:cdf}, as $n\to \infty$, 
\begin{equation}\label{dist:eq:CDFCLT}
 \sqrt{n_1}(\hat F_1-F_{1,T},\hat F_0-F_{0,T})
 \Rightarrow(\mathbb B_1,\mathbb B_0)
 \quad\text{in }\ell^\infty(\mathbb R)^2.
\end{equation}
For each $y\in \RR$, the Wald interval $\textnormal{CI}_{n,1-\alpha}^{\textnormal{contrast,Wald}}(y) = \hat{D}(y) \pm z_{1-\alpha/2} \sqrt{\hat{V}_D(y)/n_1}$ has a limiting coverage of $1-\alpha$ for $D(y)$. 
For bootstrap, suppose $B_n\to \infty$. Define
 $W_F=\max_z\sup_y|\mathbb B_z(y)|$ and 
 $W_D=\sup_y|\mathbb B_1(y)-\mathbb B_0(y)|$. 
If $W_F$ has a regular $(1-\alpha)$-quantile, then
\begin{equation}\label{dq:eq:joint-cdf-coverage}
 \PP\big\{ F_{z,T}(y) \in \textnormal{CI}_{z,n,1-\alpha}^{\textnormal{joint}}
       \text{ for every }y\in\RR,\ z\in\{0,1\}\big\}
 \to 1-\alpha.
\end{equation}
If $W_D$ has a regular $(1-\alpha)$-quantile, then
\begin{equation}\label{dq:eq:direct-cdf-coverage}
 \mathbb P\{D(y)\in\mathrm{CI}^{\textnormal{contrast}}_{n,1-\alpha}(y)
                         \text{ for every }y\in\mathbb R\}
 \to1-\alpha.
\end{equation}
The coverage probabilities include both data and bootstrap sampling randomness. 
\end{theorem}

\subsection{Quantile treatment effects}
\label{subsec:wald_quantile}

The same matched design also support inference on quantile treatment effects at regular quantiles. 
For $u\in (0,1)$, the quantile treatment effect is defined as 
\$
\Delta_Q(u) =Q_1(u)-Q_0(u),\quad \text{where}\quad Q_z(u) = \inf\{y\colon F_{z,T}(y)\geq u\}. 
\$
An estimator is $\hat{\Delta}_Q(u) = \hat{Q}_1(u)-\hat{Q}_0(u)$, where $\hat{Q}_z(u)$ is the $u$-th quantile for the empirical CDFs $\hat{F}_1(\cdot)$ and $\hat{F}_0(\cdot)$ given in~\eqref{eq:def_emp_cdf}. 
As the asymptotic variance in Wald-type inference requires involves the density function and requires additional estimation, we shall primarily focus on the convenient bootstrap inference. 

Following the approach of~\cite{chernozhukov2020generic}, one can invert the CDF inference in Section~\ref{subsec:wald_cdf} for the quantile treatment effect. Namely, one can invert the upper and lower confidence bands for $F_{z,T}(y)$ obeying~\eqref{dq:eq:joint-cdf-coverage} to obtain a valid confidence bands for $\{\Delta_Q(u)\}_{u\in (0,1)}$. However, this approach can be conservative and may produce unbounded confidence bands near the tails. 

Here, we present a  direct bootstrap inference approach for regular quantiles to show the wide range of inferential targets the matched design can support. We use the same resampling scheme introduced earlier. For $b=1,\dots,B_n$, given the resampled pair indices $I_i^{*(b)}$, we define  
\$
\hat{Q}_z^{*(b)}(u) = \inf\{y\colon \hat{F}_z^{*(b)}(y) \geq u\},\quad \hat\Delta_Q^{*(b)} = \hat{Q}_1^{*(b)}(u) - \hat{Q}_0^{*(b)}(u),
\quad S_Q^{*(b)}(u) = \sqrt{n_1} \{\hat\Delta_Q^{*(b)}  - \hat\Delta_Q(u)\}.
\$
where the bootstrap empirical CDFs $\hat{F}_z^{*(b)}(\cdot)$ are given in~\eqref{eq:def_cdf_boot}. 

The inference will be limited to regular quantiles. Consider a range $\cU = [u_-,u_+]$ with $0<u_-<u_+<1$.  
\begin{assumption}\label{assump:quantile}
    For each $z\in \{0,1\}$, there is a compact interval $J_z = [a_z,b_z]$ such that $F_{z,T}(a_z)<u_-$ and $F_{z,T}(b_z)>u_+$. On a neigborhood of $J_z$, $F_{z,T}$ is continuously differentiable and its derivative $f_{z,T}$ is bounded away from zero on $J_z$.   
\end{assumption}
Let $c_Q^*$ be the empirical $(1-\alpha)$-quantile of $T_Q^{*(b)} = \sup_{u\in \cU}|S_Q^{*(b)}(u)|$ for $b=1,\dots,B_n$. We define the simultaneous confidence band 
\$
\textnormal{CI}_{n,1-\alpha}^Q(u) = \big[\hat\Delta_Q(u) - c_Q^*/\sqrt{n_1}, \hat\Delta_Q(u) + c_Q^*/\sqrt{n_1}\big].
\$
For a fixed $u\in (0,1)$, let $c_{Q}^*(u)$ be the empirical $(1-\alpha)$-th quantile of $\{|S_Q^{*(b)}(u)|\}_{b=1}^{B_n}$. 
We then define the pointwise bootstrap interval is 
\$
\textnormal{CI}_{n,1-\alpha}^{Q,\textnormal{pt}}(u) = \big[\hat\Delta_Q(u) - c_{Q}^*(u)/\sqrt{n_1}, \hat\Delta_Q(u) + c_{Q}^*(u)/\sqrt{n_1}\big].
\$

The proof of Theorem~\ref{thm:quantile} is in Appendix~\ref{app:subsec_proof_quantile}. 

\begin{theorem}\label{thm:quantile} 
Under Assumptions~\ref{assump:obs-sample},~\ref{assump:match},~\ref{assump:cdf} and~\ref{assump:quantile}, as $n\to \infty$, 
\begin{equation}\label{dist:eq:Qlimit}
 \sqrt{n_1}(\hat\Delta_Q-\Delta_Q)
 \Rightarrow\mathcal Z_Q
 \quad\text{in }\ell^\infty(\mathcal U),\qquad
 \mathcal Z_Q(u)=
 -\frac{\mathbb B_1\{Q_1(u)\}}{f_{1,T}\{Q_1(u)\}}
 +\frac{\mathbb B_0\{Q_0(u)\}}{f_{0,T}\{Q_0(u)\}},
\end{equation}
where the joint process $(\BB_1,\BB_0)$ is defined in Section~\ref{subsec:wald_cdf}. 
Suppose $B_n\to\infty$. If
$W_Q=\sup_{u\in\mathcal U}|\mathcal Z_Q(u)|$
has a regular $(1-\alpha)$-quantile, then
\begin{equation}\label{dq:eq:quantile-coverage}
 \PP\big\{\Delta_Q(u)\in\mathrm{CI}^{Q}_{n,1-\alpha}(u) \text{ for every }u\in\cU\big\} \to 1-\alpha.
\end{equation}
For every fixed $u\in\mathcal U$ with
$\operatorname{Var}\{\mathcal Z_Q(u)\}>0$, 
\begin{equation}\label{dq:eq:quantile-pointwise-coverage}
 \PP\big\{\Delta_Q(u)\in\mathrm{CI}^{Q,\mathrm{pt}}_{n,1-\alpha}(u)\big\} \to 1-\alpha.
\end{equation}
The above coverage probabilities include both the data and bootstrap sampling randomness. 
\end{theorem}

%% file: 05_randomized.tex

\section{Asymptotic validity of randomization inference}
\label{sec:rand_inference}

Paired randomization tests are commonly used after matching in observational studies. They compare the observed statistic with a reference distribution obtained by independently exchanging the treated and control roles within each completed pair. Ideally, if the matching finds pairs with exactly the same covariate values, i.e., $T_i=C_{\hat\sigma_n(i)}$, then conditional on the covariates, the treatment is as-if randomized in the matched pair. However, with  continuously-valued covariates, the matching is often inexact and leaves bias, in the sense that the treatment conditional on the matched pair is no long as-if uniform within the pair. 

Recognizing that matching is inexact, \cite{pimentel2024covariate} proposed to modify the permutation probabilities using estimated
within-pair propensity-score discrepancies.  
Here, our focus is instead when the randomization inference is asymptotically valid in the super-population sense, similar to~\cite{GR2023},  without
outcome adjustment or propensity-score modification.

\subsection{Paired randomization test of the weak and sharp null}

We consider two common hypotheses 
\$
H_{0,\text{sharp}}\colon Y(1)=Y(0)~\text{almost surely},\qquad 
H_{0,\text{weak}} \colon \theta_\att = \EE[Y(1)-Y(0)\given Z=1]=0.
\$
Throughout, we shall use the difference-in-mean test statistic. 
Recall the paired difference $R_i$ and the estimator $\hat\theta_\att$ defined in~\eqref{eq:def_Ri_theta_att}. 
Fix some $B_n\in \NN^+$. For $b=1,\dots,B_n$, given the observed data, we generate independent signs $\omega_i^{(b)}\sim\text{Unif}(\{-1,1\})$, independently of the data.
We then define the randomization p-values 
\begin{equation}\label{eq:rimean:statistics}
p_{\RI,n}
=\frac{1+\sum_{b=1}^{B_n}
\ind\{|T_{\RI,n}^{\omega^{(b)}}| \ge |T_{\RI,n}|\}}{B_n+1}, 
\qquad 
T_{\RI,n}= \hat\theta_{\mathrm{ATT}},
\qquad
T_{\RI,n}^{\omega^{(b)}}
=\frac1{n_1}\sum_{i=1}^{n_1}\omega_i^{(b)}R_i, 
\end{equation}
On the event $\cI_n^c$, we use the convention $p_{\RI,n}=1$. 
Theorem~\ref{thm:rand_inf} shows the validity of the difference-in-mean-based randomization inference, whose proof is in Appendix~\ref{app:subsec_proof_rand_inf}. 

\begin{theorem}\label{thm:rand_inf}
Suppose Assumptions~\ref{assump:obs-sample},~\ref{assump:match},~\ref{assump:outcome} hold, and $B_n\to \infty$. 
Under $H_{0,\mathrm{weak}}$, for every fixed
$\alpha\in(0,1)$, 
\[
\mathbb P\{p_{\RI,n}\le\alpha \}\to\alpha.
\]
The same conclusion holds under $H_{0,\mathrm{sharp}}$.  
\end{theorem}
 
The probability in Theorem~\ref{thm:rand_inf} includes the randomness in both data sampling and randomization. 
The key idea of the randomization inference validity is that, due to the negligible imbalance proved in Section~\ref{subsec:uniform_balance}, the bias in the difference-in-mean statistic caused by inexact matching is asymptotically negligible. Note that while randomization inference for sharp null, assuming exact matching, can in principle use any test statistic, our results are limited to the difference-in-mean statistic so that the uniform balance works.

\subsection{A dimension boundary at \texorpdfstring{$d=4$}{d = 4}}

The preceding conclusion can fail when matching bias is not negligible. 
We continue to use the counterexample from Theorem~\ref{thm:dimension-obs}, keeping the covariate and treatment distribution and the function $g(x)$ and setting
\begin{equation}\label{eq:rimean:binary-model}
Y(1)=Y(0)=\ind\{U\le 1/4+g(X)/2\},
\qquad U\sim\operatorname{Unif}(0,1),\quad U\perp(X,Z).
\end{equation}
We consider the exact optimizer $\hat\sigma_n = \sigma_n^\star$ for $A_0=I_4$. 
The next proposition shows that at $d=4$, the randomization p-value can fail to be asymptotically valid in this example. 

\begin{prop} 
\label{prop:rand_boundary}
Assume $B_n\to \infty$. 
For the model \eqref{eq:rimean:binary-model}, for every fixed
$\alpha\in(0,1)$, there is a constant $\epsilon_\alpha>0$ such that
\[
\liminf_{n\to\infty}\mathbb P\{p_{\RI,n}\le\alpha\}
\ge\alpha+\epsilon_\alpha.
\] 
\end{prop}

The proof of Proposition~\ref{prop:rand_boundary} is in Appendix~\ref{app:subsec_rand_dimension}. 
The key idea is as follows: the observed statistic is asymptotically equal to a centered outcome-noise
term plus the imbalance term  
$\tfrac12\sqrt{n_1}B_{n,\sigma_n^\star}(g)$. 
The randomization statistics approximate the centered noise distribution, but the imbalance term is non-negligible by Theorem~\ref{thm:dimension-obs} and translates to the violation of the type-I error control.

%% file: 06_simulations.tex

\section{Numerical demonstrations}
\label{sec:simu}

We corroborate the theory via simulations, where we vary the dimension of the covariates and evaluate the coverage of Wald and Bootstrap confidence intervals. 
Below, we introduce the matching algorithms evaluated, whereas the data-generating processes and evaluation metrics are introduced separately in each subsection.

\vspace{-0.5em}
\paragraph{Matching algorithm.} Throughout the experiments, we evaluate both global matching and greedy matching. For each sample of size $n$, let $T_i$ and $C_j$ denote the covariates of
the $n_1$ treated and $n_0$ control observations, respectively. We estimate $\hat A$ as the inverse covariance of all $n$
covariate vectors, using denominator $n-1$, and form the 
cost matrix $c_{ij}=\tfrac12(T_i-C_j)^\top\hat A(T_i-C_j)$. Global matching casts the optimization as a linear assignment problem, and uses SciPy's floating-point assignment solver to minimize
$Q_n(\sigma;\hat A)=n_1^{-1}\sum_i c_{i,\sigma(i)}$ over injections, which is guaranteed to find the globally optimal solution. 
The Greedy matching approach visits treated observations in an independently drawn random
order and assigns each the remaining unmatched control with the lowest matching cost; this approach does not necessarily find the optimal solution and is shown as a robustness check. 

\subsection{Population ATT inference}
\label{subsec:simu_att}

We first evaluate population ATT inference established in Section~\ref{subsec:wald_att}. 
%
We generate $n$ independent observations with
$X_i\sim\operatorname{Unif}([0,1]^d)$ and
$Z_i\mid X_i\sim\operatorname{Bernoulli}\{e(X_i)\}$, where
$e(x)=0.1+0.3x_1$. Potential outcomes are generated by
$Y_i(z)=\mu_z(X_i)+(0.5+0.5X_{i1})\varepsilon_{iz}$ for  $z\in\{0,1\}$, 
where the errors $\{\varepsilon_{iz}\}$ are independent standard normal,
independent of the covariates and treatment indicators. We set nonlinear outcomes  
$\mu_0(x) =2x_1+\frac{2}{\sqrt d}\sum_{j=1}^d(x_j^2-1/3)+g_d(x)$, 
$\mu_1(x)=\mu_0(x)+a(x_1-3/5)$, where $a\in \{0,2\}$ controls the heterogeneity of the conditional average treatment effect (CATE) function, with $g_d(x)=0$ when $d=1$ and $g_d(x) = 2(x_1-1/2)(x_2-1/2)$  when $d\geq 2$. 
Since the matching design is outcome-blind, the estimation and inference require no fitting or adjustment based on outcome models. It also does not require estimating the propensity score. 
We vary the total sample size $n\in\{500,2000,5000\}$ and covariate dimension $d\in\{1,2,3,4,6\}$.
The population ATT
is $\theta=0$ under both heterogeneity settings.

Given the matching $\hat\sigma_n$ by the two matching algorithms introduced at the beginning of this section, we compute the paired differences
$R_i=Y_{1i}-Y_{0,\hat\sigma_n(i)}$, ATT estimator, Wald-type and bootstrap (where we set $B_n=500$) confidence intervals at the $1-\alpha=95\%$ nominal level as introduced in Section~\ref{subsec:wald_att}. 
For each matching algorithm and each inference method (Wald versus bootstrap), we evaluate the empirical coverage and CI length over $R=2000$ independent runs (including data generation, matching, and inference). 

We summarize the results for the heterogeneous CATE case with $a=2$ in Table~\ref{tab:att-a2}, while the results for the constant CATE model with $a=0$ are deferred to Table~\ref{tab:att-a0} in Appendix~\ref{app:sec_simu} with qualitatively similar messages. 
Consistent with our theory, with both global and greedy matching, both Wald and bootstrap CI achieve close to $95\%$ empirical coverage for $d\leq 3$, with the gap diminishing as the sample size grows to $n=5000$. This shows the particular benefit of the matching design: even with nonlinear outcome models, as long as $d\leq 3$, matching without outcome information still delivers valid inference. 

In contrast, when $d\in \{4,6\}$, the coverage gap does not vanish as we increase the sample size: it appears to converge to some probability below $0.95$ for $d=4$ and keeps deteriorating for $d=6$. This suggests that $d=4$ can indeed be the dimension boundary beyond which the bias of inexact matching becomes non-negligible. 

\begin{table}[p]
\centering
\caption{Population ATT inference results with $a=2$ (heterogeneous CATE). Coverage and its Monte Carlo SE (parentheses) are in percentage. Results are averaged over $R=2000$ independent runs.}
\label{tab:att-a2}
\small\setlength{\tabcolsep}{6pt}\renewcommand{\arraystretch}{1.08}
\begin{tabular}{rrlrrrr}
\toprule
&&&\multicolumn{2}{c}{Coverage (\%)}&\multicolumn{2}{c}{Mean CI width}\\
\cmidrule(lr){4-5}\cmidrule(lr){6-7}
$n$ & $d$ & Matching & Wald & Bootstrap-t & Wald & Bootstrap-t \\
\midrule
500 & 1 & Global & 94.65 (0.50) & 95.00 (0.49) & 0.4452 (0.0007) & 0.4493 (0.0009)\\
 &  & Greedy & 94.80 (0.50) & 94.60 (0.51) & 0.4452 (0.0007) & 0.4490 (0.0009)\\
\addlinespace[2pt]
 & 2 & Global & 94.70 (0.50) & 95.25 (0.48) & 0.4483 (0.0007) & 0.4532 (0.0009)\\
 &  & Greedy & 94.30 (0.52) & 94.10 (0.53) & 0.4489 (0.0008) & 0.4539 (0.0009)\\
\addlinespace[2pt]
 & 3 & Global & 93.90 (0.54) & 93.95 (0.53) & 0.4566 (0.0008) & 0.4611 (0.0009)\\
 &  & Greedy & 93.75 (0.54) & 93.55 (0.55) & 0.4575 (0.0008) & 0.4622 (0.0009)\\
\addlinespace[2pt]
 & 4 & Global & 91.20 (0.63) & 91.10 (0.64) & 0.4661 (0.0008) & 0.4711 (0.0009)\\
 &  & Greedy & 91.55 (0.62) & 90.80 (0.65) & 0.4679 (0.0008) & 0.4731 (0.0009)\\
\addlinespace[2pt]
 & 6 & Global & 85.25 (0.79) & 84.95 (0.80) & 0.4890 (0.0008) & 0.4934 (0.0010)\\
 &  & Greedy & 85.70 (0.78) & 85.50 (0.79) & 0.4914 (0.0008) & 0.4962 (0.0010)\\
\midrule
2000 & 1 & Global & 93.65 (0.55) & 93.35 (0.56) & 0.2224 (0.0002) & 0.2228 (0.0003)\\
 &  & Greedy & 93.50 (0.55) & 93.45 (0.55) & 0.2224 (0.0002) & 0.2227 (0.0003)\\
\addlinespace[2pt]
 & 2 & Global & 94.65 (0.50) & 94.65 (0.50) & 0.2229 (0.0002) & 0.2234 (0.0003)\\
 &  & Greedy & 94.20 (0.52) & 94.00 (0.53) & 0.2231 (0.0002) & 0.2235 (0.0003)\\
\addlinespace[2pt]
 & 3 & Global & 94.50 (0.51) & 94.25 (0.52) & 0.2247 (0.0002) & 0.2246 (0.0003)\\
 &  & Greedy & 94.30 (0.52) & 94.25 (0.52) & 0.2250 (0.0002) & 0.2251 (0.0003)\\
\addlinespace[2pt]
 & 4 & Global & 91.10 (0.64) & 90.80 (0.65) & 0.2280 (0.0002) & 0.2280 (0.0003)\\
 &  & Greedy & 91.30 (0.63) & 90.75 (0.65) & 0.2285 (0.0002) & 0.2286 (0.0003)\\
\addlinespace[2pt]
 & 6 & Global & 78.40 (0.92) & 77.75 (0.93) & 0.2364 (0.0002) & 0.2367 (0.0003)\\
 &  & Greedy & 79.20 (0.91) & 79.00 (0.91) & 0.2371 (0.0002) & 0.2374 (0.0003)\\
\midrule
5000 & 1 & Global & 95.00 (0.49) & 94.85 (0.49) & 0.1407 (0.0001) & 0.1407 (0.0002)\\
 &  & Greedy & 94.95 (0.49) & 94.80 (0.50) & 0.1407 (0.0001) & 0.1407 (0.0002)\\
\addlinespace[2pt]
 & 2 & Global & 94.70 (0.50) & 94.55 (0.51) & 0.1408 (0.0001) & 0.1407 (0.0002)\\
 &  & Greedy & 94.75 (0.50) & 94.60 (0.51) & 0.1408 (0.0001) & 0.1407 (0.0002)\\
\addlinespace[2pt]
 & 3 & Global & 94.05 (0.53) & 93.95 (0.53) & 0.1415 (0.0001) & 0.1415 (0.0002)\\
 &  & Greedy & 93.65 (0.55) & 93.35 (0.56) & 0.1416 (0.0001) & 0.1417 (0.0002)\\
\addlinespace[2pt]
 & 4 & Global & 91.70 (0.62) & 90.90 (0.64) & 0.1430 (0.0001) & 0.1430 (0.0002)\\
 &  & Greedy & 91.70 (0.62) & 91.65 (0.62) & 0.1431 (0.0001) & 0.1432 (0.0002)\\
\addlinespace[2pt]
 & 6 & Global & 71.50 (1.01) & 70.80 (1.02) & 0.1470 (0.0001) & 0.1469 (0.0002)\\
 &  & Greedy & 72.05 (1.00) & 71.60 (1.01) & 0.1474 (0.0001) & 0.1472 (0.0002)\\
\bottomrule\end{tabular}
\par\vspace{7pt}
\end{table}

\subsection{Distributional treatment effects}

We then evaluate the distributional inference based on a matched design in Section~\ref{subsec:wald_cdf}. 
%
We generate $n$ independent observations with
$X_i\sim\operatorname{Unif}([0,1]^d)$ and
$Z_i\mid X_i\sim\operatorname{Bernoulli}\{e(X_i)\}$, where the propensity score is $e(x)=0.10+\frac{0.30}{1+\exp\{-6h_d(x)\}}$ for 
 $h_d(x)=\frac{1}{\sqrt d}\sum_{j=1}^d(x_j^2-1/3)$. 
The potential outcomes are generated as  $Y_i(z)=4h_d(X_i)+\varepsilon_{iz}$ for
$z\in\{0,1\}$, where the errors are independent standard normal and
independent of everything else. Thus
$H_1(y\given x)=H_0(y\given x)=\Phi\{y-4h_d(x)\}$ and the population
contrast is $D(y)=F_{1,T}(y)-F_{0,T}(y)=0$ for every $y$.
We vary the total sample size $n\in\{500,2000,5000\}$ and covariate dimension $d\in\{2,3,4,6\}$.
%
Given the matching $\hat\sigma_n$ by the two algorithms described, we form the bootstrap (with $B_n=500$) joint confidence bands and confidence band for the contrast $D(y)$ at the $1-\alpha=95\%$ nominal level as described in Section~\ref{subsec:wald_cdf}. 
For each matching algorithm and each type of confidence bands, we evaluate the empirical coverage and mean band width over $R=2000$ independent runs.  

Table~\ref{tab:cdf-bands} summarizes the results. As expected, the coverage is close to nominal for $d\leq 3$. For $d=4$, as we increase the sample size $n$, the coverage tends to a constant level below $0.95$. The coverage for $d=6$ keeps deteriorating as $n$ increases. This confirms our theory, and the uniform balance violation at $d\geq 4$ appears to play a role in this experiment.

\begin{table}[!htbp]
\centering
\caption{Coverage and band width in CDF inference. Joint coverage requires both $F_{0,T}$ and $F_{1,T}$ to be covered
at every threshold; contrast coverage requires $D=F_{1,T}-F_{0,T}$ to
be covered at every threshold. Coverage and its Monte Carlo SE (parentheses) are in percentage. The confidence bands are truncated to $[0,1]$ before computing the mean width. Results are averaged over $R=2000$ independent runs.}
\label{tab:cdf-bands}
\small\setlength{\tabcolsep}{6pt}\renewcommand{\arraystretch}{1.08}
\begin{tabular}{rrlrrrr}
\toprule
& & & \multicolumn{2}{c}{Coverage (\%)} & \multicolumn{2}{c}{Mean band width} \\
\cmidrule(lr){4-5}\cmidrule(lr){6-7}
$n$ & $d$ & Matching & Joint CDF & Contrast & Joint CDF & Contrast \\
\midrule
 500 & 1 & Global & 94.00 (0.53) & 96.60 (0.41) & 0.2566 (0.00028) & 0.2740 (0.00036) \\
     &   & Greedy & 93.65 (0.55) & 96.70 (0.40) & 0.2568 (0.00028) & 0.2742 (0.00036) \\
\addlinespace[2pt]
     & 2 & Global & 94.00 (0.53) & 95.35 (0.47) & 0.2546 (0.00028) & 0.2788 (0.00037) \\
     &   & Greedy & 93.80 (0.54) & 95.40 (0.47) & 0.2546 (0.00028) & 0.2788 (0.00037) \\
\addlinespace[2pt]
     & 3 & Global & 92.35 (0.59) & 94.75 (0.50) & 0.2543 (0.00028) & 0.2816 (0.00038) \\
     &   & Greedy & 92.20 (0.60) & 94.50 (0.51) & 0.2544 (0.00028) & 0.2820 (0.00037) \\
\addlinespace[2pt]
     & 4 & Global & 90.10 (0.67) & 92.35 (0.59) & 0.2550 (0.00028) & 0.2857 (0.00039) \\
     &   & Greedy & 90.25 (0.66) & 92.30 (0.60) & 0.2547 (0.00028) & 0.2855 (0.00038) \\
\addlinespace[2pt]
     & 6 & Global & 77.95 (0.93) & 83.95 (0.82) & 0.2542 (0.00027) & 0.2913 (0.00036) \\
     &   & Greedy & 78.85 (0.91) & 83.95 (0.82) & 0.2543 (0.00026) & 0.2913 (0.00037) \\
\addlinespace[2pt]
2000 & 1 & Global & 94.90 (0.49) & 95.55 (0.46) & 0.1310 (0.00009) & 0.1382 (0.00011) \\
     &   & Greedy & 95.20 (0.48) & 96.00 (0.44) & 0.1311 (0.00009) & 0.1381 (0.00011) \\
\addlinespace[2pt]
     & 2 & Global & 94.00 (0.53) & 96.00 (0.44) & 0.1303 (0.00009) & 0.1411 (0.00011) \\
     &   & Greedy & 93.80 (0.54) & 95.80 (0.45) & 0.1303 (0.00009) & 0.1411 (0.00011) \\
\addlinespace[2pt]
     & 3 & Global & 93.65 (0.55) & 94.25 (0.52) & 0.1299 (0.00009) & 0.1420 (0.00012) \\
     &   & Greedy & 93.60 (0.55) & 94.85 (0.49) & 0.1297 (0.00009) & 0.1420 (0.00012) \\
\addlinespace[2pt]
     & 4 & Global & 91.30 (0.63) & 91.95 (0.61) & 0.1299 (0.00009) & 0.1432 (0.00012) \\
     &   & Greedy & 91.30 (0.63) & 91.80 (0.61) & 0.1299 (0.00009) & 0.1433 (0.00012) \\
\addlinespace[2pt]
     & 6 & Global & 68.20 (1.04) & 67.50 (1.05) & 0.1298 (0.00009) & 0.1457 (0.00012) \\
     &   & Greedy & 69.10 (1.03) & 68.55 (1.04) & 0.1298 (0.00009) & 0.1456 (0.00012) \\
\addlinespace[2pt]
5000 & 1 & Global & 94.50 (0.51) & 95.35 (0.47) & 0.0835 (0.00005) & 0.0876 (0.00005) \\
     &   & Greedy & 94.55 (0.51) & 95.50 (0.46) & 0.0834 (0.00005) & 0.0877 (0.00006) \\
\addlinespace[2pt]
     & 2 & Global & 94.30 (0.52) & 94.40 (0.51) & 0.0829 (0.00005) & 0.0896 (0.00006) \\
     &   & Greedy & 94.55 (0.51) & 94.55 (0.51) & 0.0830 (0.00005) & 0.0896 (0.00006) \\
\addlinespace[2pt]
     & 3 & Global & 94.10 (0.53) & 93.15 (0.56) & 0.0828 (0.00005) & 0.0901 (0.00006) \\
     &   & Greedy & 93.85 (0.54) & 93.40 (0.56) & 0.0828 (0.00005) & 0.0901 (0.00006) \\
\addlinespace[2pt]
     & 4 & Global & 90.65 (0.65) & 89.40 (0.69) & 0.0828 (0.00005) & 0.0906 (0.00006) \\
     &   & Greedy & 90.70 (0.65) & 88.80 (0.71) & 0.0828 (0.00005) & 0.0906 (0.00006) \\
\addlinespace[2pt]
     & 6 & Global & 57.15 (1.11) & 56.30 (1.11) & 0.0827 (0.00005) & 0.0920 (0.00006) \\
     &   & Greedy & 57.65 (1.10) & 57.40 (1.11) & 0.0827 (0.00005) & 0.0920 (0.00006) \\
\bottomrule
\end{tabular}
\end{table}

\subsection{Quantile treatment effects}

We further demonstrate the inference for quantile treatment effects in Section~\ref{subsec:wald_quantile}. 
We use the same covariate distribution $X\sim\operatorname{Unif}([0,1]^d)$ and treatment assignment
mechanism $e(x)$ as in the CDF experiment. 
The potential outcomes are generated as 
$Y_i(0)=4h(X_i)+\varepsilon_{i0},$ and
$Y_i(1)=0.5+4h(X_i)+\varepsilon_{i1}$, where the errors are independent standard normal and independent of everything else. The target is the quantile treatment effect 
$\Delta_Q(u)=F_{1,T}^{-1}(u)-F_{0,T}^{-1}(u)$ for
 $u\in U=[0.1,0.9]$. 
We consider the total sample size $n\in\{500,2000,5000\}$ and covariate dimension $d\in\{1,2,3,4,6\}$. 
Given the matching $\hat\sigma_n$ from the described global and greedy procedures, we construct the bootstrap simultaneous confidence bands for $\textnormal{CI}_{n,1-\alpha}^{Q}(u)$ as stated in Section~\ref{subsec:wald_quantile} over $B_n=500$ boostrap draws. 
We then report coverage over the entire curve on $U$ and the mean width of the confidence band, averaged over $R=2000$ independent runs as before.  

The simultaneous coverage is reported in Table~\ref{tab:quantile-risk-aligned} column named \texttt{Raw cov.}, where we observe valid coverage for $d\leq 3$ and deteriorating coverage for $d=6$ when $n$ increases. 
However, the coverage appears conservative: it remains above $97\%$ for $d\leq 4$ when $n$ is as large as $5000$. 

To diagnose this conservativeness, we construct an oracle version (column name \texttt{Oracle cov.}) by removing the matching-induced bias in the quantile treatment effects, and observe the over-coverage for all dimensions; this confirms that the inference without matching bias is itself conservative at this sample size.  
In addition, we measure the matching-induced bias by averaging $D_Q=\sup_{u\in U}|\delta_n(u)|$ and $\sqrt{n_1}D_Q$, where   $Q_{T,n}$ and $Q_{M,n}$
are the quantile functions of the unit-variance Gaussian mixtures
with means $4h(X_i)$ over treated and matched-control covariates,
respectively.  These quantities are reported in the column \texttt{Mean design bias}. 
The results confirm our theory: the scaled bias $\sqrt{n_1} D_Q$ appears to diminish for $d\leq 3$ (although a bit slow for $d=3$), converge for $d=4$, while keeps increasing for $d=6$. 
Together, these results demonstrate the dimension boundary for the imbalance, yet in this example, the effect on simultaneous coverage is diluted especially for $d=4$ due to the intrinsic conservativeness of the inference procedure. 

Finally, we report the coverage and width of point-wise confidence intervals for the $u$-th quantile treatment effects introduced in Section~\ref{subsec:wald_quantile}, for $u\in \{0.25, 0.5, 0.75\}$ in Table~\ref{tab:pointwise-qte} in Appendix~\ref{app:sec_simu}. The point-wise confidence intervals converge to the nominal level faster than simultaneous ones for $d\leq 3$, and illustrates the systematic undercoverage for $d=4$ and deteriorating coverage for $d=6$.

\begin{table}[!htbp]
\centering
\small\setlength{\tabcolsep}{4pt}\renewcommand{\arraystretch}{1.04}
\caption{Simultaneous coverage for quantile treatment effects. Coverage is for the entire $U=[0.1,0.9]$, in percent, with Monte Carlo SEs in parentheses.
Oracle coverage is a diagnostic which removes the matching-induced bias from the construction of confidence bands.}
\label{tab:quantile-risk-aligned}
\begin{tabular}{rrlccccc}
\toprule
&&&\multicolumn{3}{c}{Direct simultaneous band}&\multicolumn{2}{c}{Mean design bias}\\
\cmidrule(lr){4-6}\cmidrule(lr){7-8}
$n$&$d$&Matching&Raw cov. (\%) & Oracle cov. (\%) &Width&$D_Q$&$\sqrt{n_1}D_Q$\\
\midrule
500 & 1 & Global & 99.90 (0.07) & 99.85 (0.09) & 1.8362 (0.0040) & 0.00862 (0.00033) & 0.0956 (0.0037) \\
 &  & Greedy & 99.75 (0.11) & 99.70 (0.12) & 1.8412 (0.0041) & 0.00964 (0.00033) & 0.1069 (0.0037) \\
\addlinespace[2pt]
 & 2 & Global & 99.30 (0.19) & 99.60 (0.14) & 1.8309 (0.0045) & 0.04471 (0.00085) & 0.4979 (0.0095) \\
 &  & Greedy & 99.30 (0.19) & 99.35 (0.18) & 1.8333 (0.0045) & 0.04540 (0.00083) & 0.5054 (0.0093) \\
\addlinespace[2pt]
 & 3 & Global & 99.55 (0.15) & 99.70 (0.12) & 1.8313 (0.0046) & 0.09649 (0.00120) & 1.0747 (0.0136) \\
 &  & Greedy & 99.45 (0.17) & 99.75 (0.11) & 1.8349 (0.0046) & 0.09558 (0.00120) & 1.0648 (0.0135) \\
\addlinespace[2pt]
 & 4 & Global & 99.15 (0.21) & 99.40 (0.17) & 1.8423 (0.0045) & 0.15228 (0.00150) & 1.6937 (0.0169) \\
 &  & Greedy & 99.15 (0.21) & 99.65 (0.13) & 1.8435 (0.0046) & 0.14997 (0.00149) & 1.6681 (0.0168) \\
\addlinespace[2pt]
 & 6 & Global & 97.65 (0.34) & 99.95 (0.05) & 1.8600 (0.0043) & 0.24923 (0.00174) & 2.7746 (0.0196) \\
 &  & Greedy & 97.35 (0.36) & 99.80 (0.10) & 1.8610 (0.0043) & 0.24762 (0.00174) & 2.7565 (0.0196) \\
\midrule
2000 & 1 & Global & 99.10 (0.21) & 99.10 (0.21) & 0.8899 (0.0012) & 0.00096 (0.00004) & 0.0212 (0.0009) \\
 &  & Greedy & 99.35 (0.18) & 99.35 (0.18) & 0.8900 (0.0012) & 0.00119 (0.00004) & 0.0262 (0.0009) \\
\addlinespace[2pt]
 & 2 & Global & 99.15 (0.21) & 99.20 (0.20) & 0.8877 (0.0014) & 0.01150 (0.00019) & 0.2552 (0.0042) \\
 &  & Greedy & 99.00 (0.22) & 98.90 (0.23) & 0.8871 (0.0014) & 0.01149 (0.00020) & 0.2550 (0.0044) \\
\addlinespace[2pt]
 & 3 & Global & 98.85 (0.24) & 98.90 (0.23) & 0.8799 (0.0014) & 0.03833 (0.00039) & 0.8526 (0.0086) \\
 &  & Greedy & 98.50 (0.27) & 99.00 (0.22) & 0.8804 (0.0014) & 0.03718 (0.00039) & 0.8271 (0.0087) \\
\addlinespace[2pt]
 & 4 & Global & 97.80 (0.33) & 99.30 (0.19) & 0.8833 (0.0014) & 0.07535 (0.00055) & 1.6772 (0.0122) \\
 &  & Greedy & 98.25 (0.29) & 99.20 (0.20) & 0.8841 (0.0014) & 0.07304 (0.00055) & 1.6259 (0.0122) \\
\addlinespace[2pt]
 & 6 & Global & 92.45 (0.59) & 99.40 (0.17) & 0.8862 (0.0014) & 0.15845 (0.00078) & 3.5290 (0.0175) \\
 &  & Greedy & 91.80 (0.61) & 99.45 (0.17) & 0.8879 (0.0014) & 0.15515 (0.00077) & 3.4557 (0.0174) \\
\midrule
5000 & 1 & Global & 99.00 (0.22) & 99.00 (0.22) & 0.5485 (0.0006) & 0.00023 (0.00001) & 0.0079 (0.0003) \\
 &  & Greedy & 98.80 (0.24) & 98.85 (0.24) & 0.5486 (0.0006) & 0.00029 (0.00001) & 0.0101 (0.0003) \\
\addlinespace[2pt]
 & 2 & Global & 98.60 (0.26) & 98.55 (0.27) & 0.5480 (0.0006) & 0.00474 (0.00007) & 0.1662 (0.0026) \\
 &  & Greedy & 98.85 (0.24) & 98.80 (0.23) & 0.5481 (0.0006) & 0.00465 (0.00007) & 0.1632 (0.0025) \\
\addlinespace[2pt]
 & 3 & Global & 98.30 (0.29) & 98.65 (0.26) & 0.5450 (0.0006) & 0.02051 (0.00019) & 0.7208 (0.0065) \\
 &  & Greedy & 98.55 (0.27) & 98.85 (0.24) & 0.5453 (0.0006) & 0.01976 (0.00019) & 0.6942 (0.0066) \\
\addlinespace[2pt]
 & 4 & Global & 97.25 (0.37) & 98.75 (0.25) & 0.5443 (0.0006) & 0.04776 (0.00029) & 1.6792 (0.0101) \\
 &  & Greedy & 97.80 (0.33) & 98.95 (0.23) & 0.5446 (0.0006) & 0.04590 (0.00029) & 1.6140 (0.0101) \\
\addlinespace[2pt]
 & 6 & Global & 83.95 (0.82) & 98.85 (0.24) & 0.5471 (0.0006) & 0.11701 (0.00044) & 4.1194 (0.0155) \\
 &  & Greedy & 85.10 (0.80) & 98.90 (0.23) & 0.5477 (0.0006) & 0.11387 (0.00043) & 4.0086 (0.0153) \\
\bottomrule
\end{tabular}
\end{table}

\subsection{Randomization inference}

Finally, we demonstrate the super-population validity of randomization inference for $d\leq 3$ and the dimension boundary for $d\geq 4$. 
We use the same covariate distribution, propensity score, outcome regression, and matching procedures as in the ATT experiment in Section~\ref{subsec:simu_att}.
We consider two null settings. Under the weak null, we keep the heterogeneous CATE with $a=2$, so that $\theta_\att=0$. 
Under the sharp null, we instead set 
$Y_i(1)=Y_i(0)=\mu_0(X_i)+s(X_i)\eta_i$, 
where the $\eta_i$'s are independent standard normal variables and independent of everything else, leading to the same observed distribution as the $a=0$ case. We vary the total sample size $n\in \{500, 2000, 5000\}$ and the covariate dimension $d\in \{1,2,3,4,6\}$.  
Given the matching $\hat\sigma_n$ by the two algorithms, we form the randomization p-value in Section~\ref{sec:rand_inference} with $B_n=4999$. We then evaluate the type-I error at nominal levels $\alpha\in \{0.025, 0.05, 0.1\}$. 

The results for the weak null are in Table~\ref{tab:ri-weak-heterogeneous-alphas}, while we defer those for the sharp null to Table~\ref{tab:ri-sharp-shared-alphas} in Appendix~\ref{app:sec_simu}. 
Across both null settings and various nominal levels, both matching methods yield close-to-nominal rejection rates for $d\leq 3$. The rejection rate  for $d=4$ appears to converge to some value above the nominal level, and that for $d=6$ keeps increasing as the total sample size $n$ increases. This is consistent with our theory on asymptotic, super-population validity of randomization inference based on the difference-in-mean statistic. 

\begin{table} 
\centering
\small
\setlength{\tabcolsep}{4.5pt}
\renewcommand{\arraystretch}{1.08}
\caption{Rejection rates in randomization inference under the weak null. The rejection rates and its Monte Carlo SE (parentheses) are in percentage, averaged over $R=2000$ independent runs.}
\label{tab:ri-weak-heterogeneous-alphas}
\begin{tabular}{@{}rr*{6}{r}@{}}
\toprule
 & & \multicolumn{3}{c}{Global} & \multicolumn{3}{c}{Greedy} \\
\cmidrule(lr){3-5}\cmidrule(lr){6-8}
$n$ & $d$ & $\alpha=0.025$ & $\alpha=0.05$ & $\alpha=0.10$ & $\alpha=0.025$ & $\alpha=0.05$ & $\alpha=0.10$ \\
\midrule
500 & 1 & 2.65 (0.36) & 5.10 (0.49) & 9.00 (0.64) & 2.70 (0.36) & 4.90 (0.48) & 9.30 (0.65) \\
500 & 2 & 2.60 (0.36) & 5.10 (0.49) & 10.70 (0.69) & 2.65 (0.36) & 5.70 (0.52) & 10.20 (0.68) \\
500 & 3 & 3.15 (0.39) & 5.90 (0.53) & 10.95 (0.70) & 3.10 (0.39) & 5.90 (0.53) & 10.75 (0.69) \\
500 & 4 & 5.00 (0.49) & 8.65 (0.63) & 15.05 (0.80) & 4.60 (0.47) & 8.25 (0.62) & 14.70 (0.79) \\
500 & 6 & 9.15 (0.64) & 14.45 (0.79) & 23.05 (0.94) & 9.00 (0.64) & 13.95 (0.77) & 22.35 (0.93) \\
\addlinespace
2000 & 1 & 3.60 (0.42) & 6.30 (0.54) & 11.40 (0.71) & 3.75 (0.42) & 6.80 (0.56) & 11.80 (0.72) \\
2000 & 2 & 3.05 (0.38) & 5.25 (0.50) & 10.80 (0.69) & 3.30 (0.40) & 5.65 (0.52) & 11.10 (0.70) \\
2000 & 3 & 2.50 (0.35) & 5.50 (0.51) & 10.55 (0.69) & 2.80 (0.37) & 5.60 (0.51) & 10.45 (0.68) \\
2000 & 4 & 4.60 (0.47) & 8.80 (0.63) & 15.25 (0.80) & 3.80 (0.43) & 8.35 (0.62) & 14.85 (0.80) \\
2000 & 6 & 14.15 (0.78) & 21.50 (0.92) & 32.65 (1.05) & 13.00 (0.75) & 20.60 (0.90) & 31.40 (1.04) \\
\addlinespace
5000 & 1 & 2.75 (0.37) & 4.75 (0.48) & 9.95 (0.67) & 2.50 (0.35) & 5.15 (0.49) & 9.85 (0.67) \\
5000 & 2 & 2.70 (0.36) & 5.10 (0.49) & 10.50 (0.69) & 2.60 (0.36) & 5.45 (0.51) & 10.85 (0.70) \\
5000 & 3 & 3.15 (0.39) & 5.85 (0.52) & 11.55 (0.71) & 3.45 (0.41) & 6.40 (0.55) & 11.75 (0.72) \\
5000 & 4 & 4.60 (0.47) & 8.35 (0.62) & 15.15 (0.80) & 5.00 (0.49) & 8.05 (0.61) & 15.25 (0.80) \\
5000 & 6 & 20.25 (0.90) & 28.65 (1.01) & 40.75 (1.10) & 19.60 (0.89) & 27.75 (1.00) & 39.35 (1.09) \\
\bottomrule
\end{tabular}
\end{table}

%% file: 02_literature.tex

\section{Related literature}\label{sec:literature}

This work adds to the large-sample theory of matching, especially one-to-one matching without replacement. We discuss related literature in several strands. 
 
\vspace{-0.5em}
\paragraph{Matching with and without replacement.}
Much of the large-sample theory for matching concerns nearest-neighbor (NN) matching with replacement. \citet{AI2006} establish asymptotic and bias results for estimators with a fixed number of matches, and \citet{AI2011} develop regression-based bias correction. In these procedures, control samples can be reused for different treated units. More recent work has revisited NN matching from
the perspective of increasing numbers of matches and semiparametric efficiency~\citep{lin2023estimation,he2024propensity,cattaneo2025rosenbaum,lin2026consistency}. In particular, \citet{lin2023estimation} shows that the matching nearest neighbors essentially estimates a density ratio, and with regression adjustment, NN matching can be doubly robust and semiparametrically efficient.  
In contrast, we study one-to-one matching without replacement, a classical design
\citep{rosenbaum1989optimal} with applied precedents for the squared Mahalanobis
objective~\citep{bind2019bridging,zhang2023matching}. 
Assigning a control to one treated observation changes the controls available to the others. Consequently, the local nearest-neighbor calculations do not directly apply in our setting, which necessitates distinct analyses.

\vspace{-0.5em}
\paragraph{Error cancellation, bias, and balance in matching without replacement}

This work expands upon the classical bias-reduction theory for matching. Under ellipsoidal symmetry and related mixture assumptions, \citet{RT1992} and \citet{RS2006} establish equal-percent bias-reduction properties for affinely invariant matching. These results describe how matching improves covariate comparability but do not study the rate for empirical imbalance. Similarly, the balancing-score identity of \citet{RR1983} provides a population-level perspective without controlling the finite-sample discrepancy. 

For matching without replacement, even consistency depends on the availability of controls in the regions occupied by treated observations. \citet{Savje2022} shows that optimal matching on the propensity score can be inconsistent when some regions have too few controls. We rule out this complication by assuming the propensity score is strictly bounded below $1/2$, providing a clean basis for analyzing the matching bias in large samples.
\citet{GR2023} study the finite-sample matching bias using ordinary Mahalanobis distance, similarly assuming the propensity score is strictly bounded below $1/2$. They establish asymptotic validity of an unadjusted paired randomization test with one continuous covariate, conjecture validity with up to three, and give a failure example for $d=4$. Our results establish a quadratic counterpart, yielding dimensionality results echoing their conjecture.  
Second-order bias and the relevance of dimension four also have nearest-neighbor precedents. The recent work of \citet{VTY2025} derive higher-order bias bounds and expansions under geometric and smoothness conditions in the NN matching setting, yet matching without replacement requires distinct analysis techniques. Our contribution is to show that
the global use-once constraint does not create a first-order bias, and to control
the resulting signed imbalance uniformly.

\vspace{-0.5em}
\paragraph{Inference from matched samples.}

Inference after matching without replacement already has a substantial foundation. \citet{AI2012} derive a martingale representation, establish an asymptotic normal distribution for a without-replacement matching estimator, and give a consistent paired-difference variance estimator. \citet{AS2022} develop post-matching inference that allows regression misspecification and justify resampling complete matched sets (similar to our bootstrap procedures). These works isolate the approximation error from inexact matching, often assuming sufficiently small absolute distance, defined as  
 $A_n=\frac1m\sum_{i=1}^m\|T_i-C_{\sigma(i)}\|$. 
A sufficient condition for the results in \citet[Theorem~1]{AI2012} and \citet[Assumption~3]{AS2022} is $\sqrt m\,A_n = o_P(1)$. The primitive sufficient conditions in \citet[Proposition~1]{AI2012} include $m^r/N=O(1)$ for some $r>d$, requiring a control pool that grows much faster than the treated sample. By contrast, when $m\asymp N\asymp n$ and the covariates have bounded full-dimensional densities, their argument gives an $N^{-1/d}$ lower bound for average absolute distance with probability tending to one. Hence the absolute-distance condition cannot hold in dimensions two and above.
By contrast, our theory controls the imbalance $B_{n,\sigma_n^\star}(g)$ directly instead of establishing the distance bound, assuming quadratic  objective and smoothness conditions.  
 
\vspace{-0.5em}
\paragraph{Related balancing and distributional methods.}
Explicit balancing offers another route to small aggregate error. \citet{Kallus2020} formulate matching in terms of function-class imbalance and variance, while \citet{WZ2023} establish root-$n$ and efficiency results for matching subject to aggregate balance constraints. Their asymptotic analysis allows replacement and directly imposes aggregate balance constraints. Our procedure concerns minimizing the matching objective only, without any other constraints. The uniform balance is a consequence of minimizing geometric cost.

Our inference procedures also connect to a literature for counterfactual distribution and quantile effects, including regression-based~\citep{CFM2013}, inverse-propensity weighting-based~\citep{DH2014}, and matching-with-replacement based~\citep{YZ2023} procedures. We provide a justification for paired bootstrap after inexact matching without replacement.

%% file: 10_appendix.tex

\input{131_unif_balance}
\input{132_dim_balance}

\input{141_wald_att}
\input{142_wald_cdf}

\input{143_wald_quantile}

\input{150_rand_inference}

\FloatBarrier
\section{Auxiliary lemmas and proofs}

Below we quote well-known concentration inequalities (without proof) and simple facts that will be used in the technical proofs. 
The two concentration bounds below can be found in \citet[Chapter 2]{concentration2013book} and \citet[Chapter 3]{concentration2013book}, respectively. 

\begin{lemma}[Chernoff bounds, quoted]\label{lem:chernoff}
If $X$ is binomial with mean $\mu$, then $\PP\{X\le(1-\epsilon)\mu\}\le\exp(-\epsilon^2\mu/2)$ for $0<\epsilon<1$, and $\PP\{X\ge(1+\epsilon)\mu\}\le\exp\{-\epsilon^2\mu/(2+\epsilon)\}$ for $\epsilon>0$.
\end{lemma} 

\begin{lemma}[Efron--Stein inequality]\label{lem:ES}
For independent coordinates $Z=(Z_1,\ldots,Z_s)$ and square-integrable $H(Z)$, let $Z^{(j)}$ replace $Z_j$ by an independent copy. Then $\Var (H)\le\frac12\sum_j\EE[\{H(Z)-H(Z^{(j)})\}^2]$.
\end{lemma} 

The next is a bound which follows simply from Taylor expansion.

\begin{lemma}[Taylor expansion]\label{lem:Taylor}
For $x,y\in\Omega$, $g(x)-g(y)=\nabla g(x)^\top(x-y)+R_g(x,y)$, where $|R_g(x,y)|\le 1/2\cdot\norm{\nabla^2 g}_\infty\norm{x-y}^2$.
\end{lemma} 

\begin{lemma}[Finite lower envelopes]\label{lem:envelope}
Let $\{a_j\}_{j=1}^J$ and $\{b_j\}_{j=1}^J$ be fixed real numbers for some fixed integer $J<\infty$. 
Define $V(t)=\min_{j\in J}(a_j+t b_j)$. 
Then $V$ is continuous, concave, piecewise affine, and absolutely continuous on bounded intervals. If $j_0$ is a minimizer at $t=0$, i.e., $a_{j_0}=V(0)$, then for every $h>0$,
\begin{equation}\label{eq:secants}
 \frac{V(h)-V(0)}h\le b_{j_0}\le\frac{V(0)-V(-h)}h.
\end{equation}
Finally, for every $t_0\in \RR$ at which $V$ is differentiable, we have $b_j = V'(t_0)$ for every $j\in \argmin_{k \in J} \{a_k+t_0b_k\}$. 
\end{lemma}

\begin{proof}[Proof of Lemma~\ref{lem:envelope}]
Since each $t\mapsto a_j+tb_j$ function is affine, the minimum of finitely many such affine functions is concave and continuous. Also, as each of them is a straight line, since $J<\infty$, there are only finitely many possible intersection points between distinct nonparallel lines (identical lines cause no difficulty). The function $V$ is therefore piecewise affine. 
Equation~\eqref{eq:secants} follows from the fact that $V(h)\le a_{j_0}+hb_{j_0}$ and $V(-h)\le a_{j_0}-hb_{j_0}$ due to the definition of $V$. 
To prove the last claim, fix any $t_0\in \RR$ at which $V$ is differentiable, and take any $j\in \argmin_{k\in J}\{a_k+t_0b_k\}$. Then by the definition of $V$, for every $h\in \RR$, 
\$
V(t_0+h)\leq a_j + (t_0+h) b_j  = V(t_0)+hb_j.
\$
For $h>0$, dividing both sides by $h$ and take the limit $h\to 0$, by differentiability we know $V'(t_0)\leq b_j$. 
For $h<0$, dividing both sides by $h$ changes the direction of the inequality; so taking the limit $h\to 0$, by differentiability we know $V'(t_0)\geq b_j$. Combining the two directions, we obtain $V'(t_0)=b_j$. 
\end{proof}

Lemma~\ref{lem:unif:bd} is the almost-sure version of the McDiarmid's inequality; see \citet[Theorem~6.2]{concentration2013book}. 
Its proof is well-known, and we only briefly justify the use of almost-sure bounds. 

\begin{lemma} 
\label{lem:unif:bd}
Let $X_1,\ldots,X_n$ be independent random elements and let
$W=w(X_1,\ldots,X_n)$ for a bounded measurable real function $w$.
For each $j$, let $X^{(j)}$ replace $X_j$ by an independent copy from
its own distribution. Suppose there are deterministic $c_j\ge0$ such that
$|w(X)-w(X^{(j)})|\le c_j$ almost surely for each $j$.
If $v_c=\sum_{j=1}^n c_j^2>0$, then for every $u>0$,
\begin{equation}\label{eq:unif:bd}
 \mathbb P(|W-\EE[W]|>u)
 \le2\exp(-2u^2/v_c).
\end{equation}
If $v_c=0$, then $W=\EE[W]$ almost surely.
\end{lemma}
\begin{proof}[Proof of Lemma~\ref{lem:unif:bd}] 
To justify this almost-sure formulation, write $X_{-j}=(X_i)_{i\ne j}$
and $P_j=\mathcal L(X_j)$.
By Fubini, for almost every $x_{-j}$, the replacement
assumption gives
\[
|w(x_{-j},x_j)-w(x_{-j},x'_j)|\le c_j
\quad\text{for }(P_j\otimes P_j)\text{-almost every }
(x_j,x'_j).
\]
Here $w(x_{-j},x_j)$ means that $x_j$ is inserted in
coordinate $j$. Consequently, the conditional essential
range of $W$ given $X_{-j}$ has length at most $c_j$:
otherwise, two sets of values separated by more than $c_j$
would each have positive conditional probability, contrary
to the displayed bound and conditional independence of
the two draws.
The proof of the cited theorem applies with these essential
ranges and gives the asserted tail bound.
If $v_c=0$, the Efron--Stein inequality gives
$\operatorname{Var}(W)=0$.
\end{proof}

Lemma~\ref{lem:unif:entropy} gives a covering number bound for $\cF_L$ which is also standard ($\cF_L$ is contained in a fixed H\"older ball
of order $2$ in the convention of
\citet[Section~2.7.1]{VW1996}; see their Theorem 2.7.1 for a covering number bound).

\begin{lemma}
\label{lem:unif:entropy}
For the class $\mathcal F_L$ on $[0,1]^d$, there is a constant $C$
depending only on $d,L$ such that, for every $0<\varepsilon\le1$, a
finite set $\mathcal N_\varepsilon\subset\mathcal F_L$ satisfies
\begin{equation}\label{eq:unif:entropy}
 \sup_{g\in\mathcal F_L}\min_{f\in\mathcal N_\varepsilon}
          \|g-f\|_\infty\le\varepsilon,
 \qquad
 \log|\mathcal N_\varepsilon|
 \le C\varepsilon^{-d/2}\log(C/\varepsilon).
\end{equation}
In particular, $\mathcal F_L$ has a countable dense subset in the sup norm.
\end{lemma}
 
Lemma~\ref{lem:unif:chain} is also a standard result providing uniform covergence via covering. We include a brief proof quoting some textbook results. 
\begin{lemma}[Uniform convergence under sub-Gaussian increments]
\label{lem:unif:chain}
Let $1\le d<4$. For every $n$, let $\{Z_n(g):g\in\mathcal F_L\}$ be a real
process with measurable coordinates and almost surely continuous paths
under $\|\cdot\|_\infty$.
Suppose a fixed constant $0<K<\infty$ satisfies
\[
 \mathbb P(|Z_n(g)-Z_n(f)|>u)
 \le2\exp\{-u^2/(K\|g-f\|_\infty^2)\}
\]
for every $n$, distinct $g,f\in\mathcal F_L$, and $u>0$.
Then whenever $Z_n(g)=o_P(1)$ for every fixed $g$, we have  
 $\sup_{g\in\mathcal F_L}|Z_n(g)|\to 0$ in probability.  
\end{lemma}
\begin{proof}[Proof of Lemma~\ref{lem:unif:chain}]
For $\delta>0$, define
\[
\omega_n(\delta)
=
\sup_{ g,f\in\mathcal F_L, \|g-f\|_\infty<\delta}
|Z_n(g)-Z_n(f)|.
\]
The countable dense subset given in 
Lemma~\ref{lem:unif:entropy}, together with sample-path
continuity, gives measurable versions of this supremum
and of $\sup_g|Z_n(g)|$.

The sub-Gaussian increment bound and the entropy-integral
maximal inequality of
\citet[Theorem~2.2.4]{VW1996} imply
\[
\sup_n\EE\omega_n(\delta)
\le
C\int_0^\delta
\sqrt{1+\log N(u,\mathcal F_L,\|\cdot\|_\infty)}\,du
\to 0 
\]
as $\delta\downarrow 0$. 
where the last limit follows from
Lemma~\ref{lem:unif:entropy} and $d<4$.

Fix $\delta>0$ and choose a deterministic finite
$\delta/2$-net $\mathcal N_\delta\subset\mathcal F_L$.
For every sample path,
\[
\sup_{g\in\mathcal F_L}|Z_n(g)|
\le
\max_{f\in\mathcal N_\delta}|Z_n(f)|+\omega_n(\delta).
\]
Consequently, for every $\varepsilon>0$,
\[
\mathbb P\!\left\{\sup_g|Z_n(g)|>\varepsilon\right\}
\le
\sum_{f\in\mathcal N_\delta}
\mathbb P\{|Z_n(f)|>\varepsilon/2\}
+
\frac{2}{\varepsilon}\EE\omega_n(\delta).
\]
For any fixed $\delta$, the finite sum tends to zero by
pointwise convergence. Taking $\limsup_{n\to\infty}$
and then letting $\delta\downarrow0$ proves the claim.
\end{proof}

Lemma~\ref{lem:expect_var_bd} provides a one-sided variance bound. It can be proved by applying the Cantelli's inequality, and we include a short proof here. 

\begin{lemma}\label{lem:expect_var_bd}
    If a real random variable $Z$ has mean $\mu>0$ and variance at most $v$, then $\PP(Z\geq \mu/2) \geq \mu^2/(4v+\mu^2)$.  
\end{lemma}

\begin{proof}[Proof of Lemma~\ref{lem:expect_var_bd}]
    If $v=0$ the conclusion is immediate. Otherwise, for any $t>0$ and $b\geq 0$, on the event $\{Z-\mu\leq -t\}$ we have $(Z-\mu-b)^2 \geq (t+b)^2$. Thus by Markov's inequality, we have $\PP(Z-\mu\leq -t) \leq (v+b^2)/(t+b)^2$. Taking $b=v/t$  gives $\PP(Z-\mu\leq -t)\leq (v+v^2/t^2)/(t+v/t)^2$, hence $\PP(Z\geq \mu/2)\geq 1- (v+v^2/t^2)/(t+v/t)^2 = 1-v/(v+t^2) = t^2/(v+t^2)$. Then, taking $t=\mu/2$ gives the desired bound.
\end{proof}

The next lemma is the standard continuity theorem; see, e.g., the textbook~\cite{billingsley2012probability}. 
\begin{lemma}\label{lem:Levy}
If characteristic functions of probability laws on $\RR^k$ converge pointwise to a characteristic function continuous at zero, the corresponding laws converge weakly to that law. Conversely, weak convergence implies pointwise convergence of characteristic functions.  
\end{lemma} 

We also present a conditional central limit theorem; its proof is standard and we include one below. 

\begin{lemma}[Conditional CLT]\label{lem:conditional-CLT}
Let $n_1=n_1(n)\in \NN^+$ be $\cG_n$-measurable  with ${n_1}\to\infty$ in probability. Conditional on $\cG_n$, suppose the row $\eta_{n1},\ldots,\eta_{n,n_1}$ is independent and centered, with $\EE(|\eta_{ni}|^3\mid\cG_n)\le K$ for a fixed constant $K$, $1\leq i\leq n_1$. If
$v_n={n_1}^{-1}\sum_i\E(\eta_{ni}^2\mid\cG_n)\to v\ge0$ in probability, then for each fixed $u\in \RR$,
\begin{equation}\label{eq:cond-cf}
 \E\left[ e^{iu {n_1}^{-1/2}\sum_i\eta_{ni}}\given |\cG_n\right]
 \to \exp(-u^2v/2)
 \quad\text{in probability and }L^1.
\end{equation}
If $A_n$ is $\cG_n$-measurable and $A_n\Rightarrow N(0,w)$, then $A_n+{n_1}^{-1/2}\sum_i\eta_{ni}\Rightarrow N(0,w+v)$.
\end{lemma}

\begin{proof}[Proof of Lemma~\ref{lem:conditional-CLT}]
Write 
$S_n=n_1^{-1/2}\sum_{i=1}^{n_1}\eta_{ni}$, 
$\varphi_n(u)=\EE(e^{iuS_n}\mid\mathcal G_n)$. 
For every $\varepsilon>0$, the third-moment bound gives
\[
\frac1{n_1}\sum_{i=1}^{n_1}
\EE\!\left[
\eta_{ni}^2
\ind\{|\eta_{ni}|>\varepsilon\sqrt{n_1}\}
\,\middle|\,\mathcal G_n
\right]
\le \frac{K}{\varepsilon\sqrt{n_1}}.
\]

Suppose first that $v>0$. From any subsequence, select a
further subsequence along which $n_1\to\infty$ and $v_n\to v$
almost surely. For almost every realization, the conditional
laws along this further subsequence satisfy the conditions
of the Lindeberg--Feller theorem
\citep[Theorem~3.4.10]{Durrett2019}.
Applying it to these independent rows gives
$\varphi_n(u)\to e^{-u^2v/2}$ almost surely along the further
subsequence, for every fixed $u\in\mathbb R$.
Thus the convergence holds in probability along the original
sequence.

If $v=0$, conditional Cauchy--Schwarz gives directly
$|\varphi_n(u)-1|
\le |u|\,\EE(|S_n|\mid\mathcal G_n)
\le |u|\sqrt{v_n} = o_P(1)$.
In either case, $|\varphi_n(u)-e^{-u^2v/2}|\le2$,
so the convergence also holds in $L^1$.

Finally, since $A_n$ is $\mathcal G_n$-measurable,
we know 
$\EE [e^{iu(A_n+S_n)}]=\EE\{e^{iuA_n}\varphi_n(u)\} =e^{-u^2v/2}\EE [e^{iuA_n}]+o(1) \to e^{-u^2(v+w)/2}$.  
The error in the second equality is bounded in absolute
value by $\EE|\varphi_n(u)-e^{-u^2v/2}|$.
The last assertion follows from Lemma~\ref{lem:Levy}.
\end{proof}

\subsection{Proofs of auxiliary lemmas in Appendix~\ref{app:subsec_thm_wald_cdf}}
\label{cdf32:app:comparisons}

In this section we prove the technical lemmas in Appendix~\ref{app:subsec_thm_wald_cdf}. 

We use the bounded-variable Bernstein inequality
\citep[Section~2.8]{concentration2013book}. If independent centered variables $X_i$ satisfy
$|X_i|\le b$ and $\sum_i\EE [X_i]^2\le v$, then
\begin{equation}\label{cdf32:eq:bernstein}
\mathbb P\bigg\{\Big|\sum_i X_i\Big|>
             \sqrt{2vt}+ {2bt}/{3}\bigg\}\le2e^{-t},
\qquad t>0.
\end{equation} 
 
\begin{proof}[Proof of Lemma~\ref{cdf32:lem:crossing}]
For each fixed threshold $y$, the inverse-CDF identity gives
\[
\ind\{A_i\le y\}=\ind\{V_i\le F_i(y)\},\qquad
\ind\{B_i\le y\}=\ind\{V_i\le G_i(y)\}.
\]
Their difference is nonzero if and only if $V_i$ lies between the two CDF
values, apart from endpoints of probability zero. Therefore
\[
\EE[e_i(y)]=F_i(y)-G_i(y),\qquad
\EE[e_i(y)^2]=|F_i(y)-G_i(y)|.
\]
The centered $e_i(y)$ are bounded by two and their variance sum is at most
$m\delta$. The $e_i(y)^2$ are indicators, with both mean sum and variance
sum at most $m\delta$.

\medskip
\noindent\textit{First, suppose the CDFs are continuous.}
Let $J(y)=(2m)^{-1}\sum_i\{F_i(y)+G_i(y)\}$ and put $K=m^2$.
Choose increasing points $y_j$ with $J(y_j)=j/K$, $1\le j<K$, and add
$y_0=-\infty$, $y_K=+\infty$. 
The existence of such points follows from continuity. Each interval $(y_j,y_{j+1}]$ has $J$-mass $1/K$.
Let
\[
N_j=\sum_{i=1}^m
 \{\ind\{A_i\in(y_j,y_{j+1}]\}
  +\ind\{B_i\in(y_j,y_{j+1}]\}\}.
\]
Then $\EE[N_j]=2m/K$. The contributions for different $i$ are independent
and lie in $[0,2]$. Since $x^2\le2x$ on $[0,2]$, their variance sum is at
most $4m/K$. Their centered versions are bounded by two.

Apply \eqref{cdf32:eq:bernstein} at every grid point to the centered sums
of $e_i(y_j)$ and of $e_i(y_j)^2$, and to each $N_j-2m/K$.
There are at most $3(K+1)$ quantities. With $t=(r+4)\log m$, a union
bound gives failure probability at most
$6(K+1)m^{-(r+4)}\le C_r m^{-r}$. On its complement,
\begin{align*}
\max_j\left|\sum_i\{e_i(y_j)-\EE[e_i(y_j)]\}\right|
&\le C_r\{\sqrt{m\delta\log m}+\log m\},\\
\max_j\sum_i e_i(y_j)^2
&\le m\delta+C_r\{\sqrt{m\delta\log m}+\log m\},\\
\max_{j<K}N_j
&\le\frac{2m}{K}
  +C_r\left\{\sqrt{\frac{m\log m}{K}}+\log m\right\}.
\end{align*}
For $y_j\le y\le y_{j+1}$, an indicator can change only if the corresponding
observation belongs to $(y_j,y_{j+1}]$. Consequently,
\[
\sum_i|e_i(y)-e_i(y_j)|\le N_j,
\qquad
\sum_i|\EE[e_i(y)]-\EE[e_i(y_j)]|\le 2m/K.
\]
Thus the centered sum at any threshold differs from its value at a grid
point by at most $N_j+2m/K$. Also, because
$|a^2-b^2|\le2|a-b|$ for $a,b\in\{-1,0,1\}$,
\[
\sum_i e_i(y)^2\le\sum_i e_i(y_j)^2+2N_j.
\]
Divide these bounds by $\sqrt m$ and $m$, respectively. With $K=m^2$, the
additional terms are bounded by $C_r\log m/\sqrt m$ and
$C_r\log m/m$. Enlarging $C_r$ proves
\eqref{cdf32:eq:crossing-centered}--\eqref{cdf32:eq:crossing-square}, including
the two tail intervals. The finite grid points have no outcome mass in
this continuous case.

\medskip
\noindent\textit{Now allow jumps in the CDFs.}
Fix $m$ and the CDFs. For $\epsilon>0$ define
\[
F_{i,\epsilon}(y)=\int_0^1 F_i(y+\epsilon v)\,dv,
\qquad
G_{i,\epsilon}(y)=\int_0^1 G_i(y+\epsilon v)\,dv.
\]
For any proper CDF $F$, the corresponding $F_\epsilon$ is the law of
$X-\epsilon V$ for independent $X\sim F$ and $V\sim\operatorname{Unif}(0,1)$.
Conditionally on $X$, this law has a density bounded by $1/\epsilon$, so
$F_\epsilon$ is continuous. Monotonicity and right continuity give
\begin{equation}\label{cdf32:eq:smoothing}
F(y)\le F_\epsilon(y)\le F(y+\epsilon),\qquad
F_\epsilon(y)\to F(y)\quad\text{for every }y.
\end{equation}
If $F\le G$, the set where $G\ge v$ contains the set where $F\ge v$;
taking their infima gives $G^{\leftarrow}(v)\le F^{\leftarrow}(v)$.
Apply this to the two inequalities in \eqref{cdf32:eq:smoothing} to obtain
\begin{equation}\label{cdf32:eq:inverse-sandwich}
F^{\leftarrow}(v)-\epsilon
\le F_\epsilon^{\leftarrow}(v)\le F^{\leftarrow}(v),
\qquad 0<v<1.
\end{equation}
Moreover,
$\|F_{i,\epsilon}-G_{i,\epsilon}\|_\infty
 \le\|F_i-G_i\|_\infty$ by the integral triangle inequality.

Using the original uniforms, define
$A_i^\epsilon=F_{i,\epsilon}^{\leftarrow}(V_i)$ and
$B_i^\epsilon=G_{i,\epsilon}^{\leftarrow}(V_i)$, and let
$e_i^\epsilon(y)$ be their indicator differences. The continuous-case bounds
hold with the original $\delta$ and the same constants for every $\epsilon$.
Equation \eqref{cdf32:eq:inverse-sandwich} implies
$A_i-\epsilon\le A_i^\epsilon\le A_i$ and similarly for $B_i$.
For every fixed $y$, it follows that 
$\ind\{A_i^\epsilon\le y\}\to \ind\{A_i\le y\}$, and 
$\ind\{B_i^\epsilon\le y\}\to \ind\{B_i\le y\}$. 
The expectations of $e_i^\epsilon(y)$
also converge to those of $e_i(y)$ by \eqref{cdf32:eq:smoothing}.

For these fixed $m$ and CDFs, let $a$ and $b$ denote the deterministic
right sides of \eqref{cdf32:eq:crossing-centered} and
\eqref{cdf32:eq:crossing-square}. All the paths involved are right-continuous,
so their suprema equal the suprema over rational thresholds. If either
original supremum exceeds its bound, a rational threshold witnesses a strict
violation. Along any deterministic $\epsilon_l\downarrow0$, that same strict
violation persists for all sufficiently large $l$. Writing $E$ and $E_l$
for the unions of the two failure events in the original and approximating
samples, respectively, gives
\[
\ind_E\le\liminf_{l\to\infty}\ind_{E_l},
\qquad
\mathbb P(E)\le\liminf_{l\to\infty}\mathbb P(E_l)
\le C_r m^{-r},
\]
where the middle inequality is Fatou's lemma. This proves both inequalities
for arbitrary proper CDFs. The approximation takes place at each fixed $m$
and is used only in this proof; the estimators and bootstrap do not modify
any observed outcomes.
\end{proof}
  
\begin{proof}[Proof of Lemma~\ref{cdf32:lem:resample}]
Condition on the fixed pairs. For any $y$, the variables
$e_{I_i^*}(y)-\bar e_m(y)$ are independent, centered, bounded by two, and have
conditional variance
\[
\frac1m\sum_{j=1}^m\{e_j(y)-\bar e_m(y)\}^2
\le\frac1m\sum_{j=1}^m e_j(y)^2\le\kappa_m.
\]
As $y$ varies, the vector $(e_1(y),\ldots,e_m(y))$ has at most $2m+1$
distinct values: it changes only at an observed $A_i$ or $B_i$.
At one representative threshold for each vector, apply
\eqref{cdf32:eq:bernstein} with variance bound $m\kappa_m$, bound two on
the centered summands, and $t=(r+3)\log m$. A union bound has failure
probability at most $2(2m+1)m^{-(r+3)}\le C_r m^{-r}$.
On its complement, division by $\sqrt m$ gives
\[
\sup_y\left|\frac1{\sqrt m}\sum_i
                       \{e_{I_i^*}(y)-\bar e_m(y)\}\right|
\le\sqrt{2\kappa_m(r+3)\log m}
      +\frac{4(r+3)\log m}{3\sqrt m}.
\]
This is \eqref{cdf32:eq:resampling-bound} with an enlarged $C_r$.
Ties only decrease the number of distinct vectors.
\end{proof}

\input{161_additional_simulation}

%% file: 131_unif_balance.tex

\section{Proof of Theorem~\ref{thm:observational}}
\label{app:sec_proof_uniform_balance}

Under the super-population model, conditional on the treatment assignments $\{Z_i\}_{i=1}^n$, the treated and control group sizes are fixed, and the covariates in each group are i.i.d.~from the distribution $\PP_{X\given Z=1}$ and $\PP_{X\given Z=0}$, respectively. 
Our theoretical results are more easily stated when holding the $\{Z_i\}_{i=1}^n$'s fixed, and will be extended to random treatment assignments by concentration arguments. 
In the former case, for easier distinctions between treated and control samples, we denote covariates in the treated group as $\{T_i\}_{i=1}^{n_1}$ and those in the control group as $\{C_j\}_{j=1}^{n_0}$. 
We assume the following conditions on the densities:

\begin{assumption}\label{ass:cov}
$\{T_i\}_{i=1}^{n_1}$ and $\{C_j\}_{j=1}^{n_0}$ are i.i.d.~from distributions with densities $p_1$ and $p_0$, respectively. The densities $p_0,p_1$ have $C^1$ extensions to a neighborhood of $\Omega$, satisfy $0<\underline p\le p_z(x)\le\overline p<\infty$ on $\Omega$, and have bounded first derivatives. For some fixed $\delta>0$ and all sufficiently large $n$,
\begin{equation}\label{eq:surplus}
 n_0 p_0(x)\ge(1+\delta)n_1 p_1(x),\qquad x\in\Omega.
\end{equation} 
\end{assumption}

We first prove the result conditional on the group sample sizes, viewing $n_1,n_0$ as fixed. 

\begin{theorem}\label{thm:uniform_balance}
    Under Assumption~\ref{ass:cov},   suppose $d\leq 3$, and the (approximate) solution $\hat\sigma_n$ satisfies $\Delta_n := Q_n(\hat\sigma_n ;A_0)-Q_n^\star = o_P(1/n)$. Then, for every fixed $0<L<\infty$, we have 
    \@\label{eq:uniform_balance}
    \sqrt{n}\sup_{g\in \cF_L} \big| B_{n,\hat\sigma_n}(g)\big| \to 0
    \@
    in probability. In particular, this conclusion holds for the exact optimizer $\sigma_n^\star$. 
\end{theorem}

The proof of Theorem~\ref{thm:uniform_balance} is given in Appendix~\ref{app:subsec_imbalance_fixed_theorem}-\ref{app:subsec_covering}. The road map is:
\begin{itemize}
    \item Appendix~\ref{app:subsec_imbalance_fixed_theorem} controls the imbalance for a fixed function $g$ for the exact optimizer when the objective involves a fixed matrix $A$. 
    \item Appendix~\ref{app:subsec_opt_hatA} extends this to inexact optimization and data-dependent matrix $\hat{A}$. 
    \item Appendix~\ref{app:subsec_covering} turns the per-$g$ bound into a uniform balance result using chaining and covering. 
\end{itemize}
Then, we complete the proof of Theorem~\ref{thm:observational} in Appendix~\ref{app:subsec_proof_obs} with random treated/control sample sizes. 

\subsection{Controlling the imbalance of exact optimizer for a fixed function}
\label{app:subsec_imbalance_fixed_theorem}

Here we control $B_{n,\sigma_n^\star}(g)$ for a fixed function $g\in \cF_L$ and a fixed symmetric positive-definite matrix $A$. 
Write 
\$
a_A = \lambda_{\min}(A)>0,\quad b_A=\lambda_{\max}(A)<\infty.
\$
We maintain the following fixed-ratio condition: 
for some fixed $a_{\mathrm s}>0$, $a_{\mathrm s}n\le n_1,n_0\le(1-a_{\mathrm s})n$.  

For clarity, in this part, we revise a few notations to make the dependence on $A$ and dataset $D=(T_1,\dots,T_{n_1},C_1,\dots,C_{n_0})$ explicit. 
We use  $Q_D(\sigma,A)$ to denote the objective $Q_n(\sigma;A)$ defined in~\eqref{eq:def_matching_obj}, and 
denote the minimizer as $\sigma_{D,A}^\star$. 
We use $B_D(\sigma,g)$ to denote the imbalance measure $B_{n,\sigma}(g)$ defined in~\eqref{eq:def_imbalance_g}. 
We also denote the minimizer-incurred imbalance as $B_{D,A}(g) = B_{n,\sigma_{D,A}^\star}(g)$. 
Analogously, we write $L_D(\sigma)=L_n(\sigma)$, where 
\@\label{eq:def_Ln}
L_n(\sigma) = \frac{1}{n_1}\sum_{i=1}^{n_1} \|T_i-C_{\sigma(i)}\|^2.
\@

\begin{theorem}[Smooth signed balance]\label{thm:balance}
Under Assumption~\ref{ass:cov}, for all sufficiently large $n$, we have 
\begin{equation}\label{eq:balance-moments}
 \big| \E B_{D,A}(g) \big| \le C\Gamma_n,
 \qquad
 \E B_{D,A}(g)^2\le C\{\Gamma_n^2+\Gamma_n/\sqrt n\}.
\end{equation}
Consequently,
\begin{equation}\label{eq:balance-rate}
 B_{D,A}(g)=O_P \big(  
 n^{-2/d}(\log n)^{1+2/d}
 +n^{-1/4-1/d}(\log n)^{1/2+1/d} \big).
\end{equation}
For $d\le3$, $\sqrt n B_{D,A}(g)\to0$ in $L^2$ and in probability as $n\to \infty$. 
\end{theorem}

The remaining of this subsection is devoted to the proof of Theorem~\ref{thm:balance}. 
The key strategy is to define a tilted loss function that combines the matching objective and imbalance. We shall control the mean, variance, and expected derivative of the tilted loss, which leads to bounds on the imbalance for a fixed function $g$.

Fixing a constant $\cgrid$, we define $q_n = \lfloor \{n/(\cgrid\log n)\}^{1/d}\rfloor$, and partition the cube $\Omega=[0,1]^d$ into $q_n^d$ equal cells of side length $l_n = q_n^{-1}$. We use half-open intervals except on the outer faces of $\Omega$, so every point in $\Omega$ belongs to a unique cell. 
We then set the cell diameter 
\@\label{eq:def_rn_Gamma_n}
r_n = \sqrt{d}\cdot  l_n  \asymp \Big(\frac{\log n}{n}\Big)^{1/d},\quad 
\Gamma_n = r_n^2 \log n, \quad 
H_n(t) = (r_n^2+t^2)\log n  \asymp \Gamma_n + t^2\log n.
\@
Also, $n\ell_n^d$ lies between fixed multiples of $\cgrid \log n$. 
We shall specify suitable choices of $t$ in later stages of the proof. 
Throughout, we use $C$ to denote a constant that depends only on fixed regularity quantities and the $C^2$-norm of $g$, but not on $n_1,n_0, n$ and the specific choice of $g$. The value of $C$ may change from line to line.

We call a \uline{\textit{one-coordinate replacement}} of a dataset $D=(T_1,\dots,T_{n_1},C_1,\dots,C_{n_0}) \in \Omega^{n}$ as replacing one element by another vector in $\Omega$ while holding others and the memberships to the treated/control groups fixed.


For any $t\in \RR$, define the tilted loss 
\$
\ell_t(x,y) = \frac{1}{2}(x-y)^\top A (x-y) + t\{g(x)-g(y)\},
\$
and the tilted matching objective 
\@\label{eq:def_Vt}
V_D(t) = \min_{\sigma\in \cA_{n_1,n_0}} \frac{1}{n_1}\sum_{i=1}^{n_1} \ell_t(T_i,C_{\sigma(i)}) =
 \min_{\sigma\in \cA_{n_1,n_0}}  \Big\{ Q_D(\sigma,A) + t B_D(\sigma,g) \Big\}.
\@
Then $V_D(0)$ coincides with the matching objective. 
We suppress the dependence on $A,g$ in $\ell_t$ and $V_D(t)$ only. 
We write the minimizer of $V_D(t)$ as 
\@\label{eq:def_Vd_minimizer}
\sigma_t = \argmin_{\sigma\in \cA_{n_1,n_0}} \Big\{Q_D(\sigma,A) +t B_D(\sigma,g)\Big\}.
\@
When this minimizer is non-unique, we choose $\sigma_t$ as the lexicographically first minimizer. This optimizer is well-defined. 
The proof of the next lemma is in Appendix~\ref{app:subsec_lem_unique_tilt}. 

\begin{lemma} \label{lem:unique}
For every fixed $t\in\RR$, the minimizer in~\eqref{eq:def_Vt} is unique almost surely.
\end{lemma}

Let $L_g = \sup_{x\in \Omega}\|\nabla g(x)\|$ and write $r=\|x-y\|$. Then as $\Omega$ is convex we know $|g(x)-g(y)|\leq L_gr$. By the smallest eigenvalue of $A$, we know $\ell_t(x,y)\geq (a_A/2)r^2 - L_g |t| r$. The Cauchy-Schwarz inequality gives 
\@\label{eq:A2:edge-lower}
\ell_t(x,y) \geq \frac{a_A}{2}r^2 - \frac{a_A}{4}r^2 - \frac{L_g^2}{a_A}t^2 = \frac{a_A}{4} \|x-y\|^2 - \frac{L_g^2}{a_A}t^2. 
\@
Thus the cost contribution of an edge between any $(x,y)$ to $V_D(t)$ is at least $-Ct^2$ for some constant $C>0$. 
If $(x,y)$ is connected by a local edge, then $r\leq r_n$ in~\eqref{eq:def_rn_Gamma_n}, and the Cauchy-Schwarz inequality gives 
\@\label{eq:A2:edge-upper}
\ell_t(x,y) \leq \frac{b_A}{2}r_n^2 +L_g|t|r_n 
\leq \frac{b_A+1}{2} r_n^2 + \frac{L_g^2}{2}t^2 \leq C(r_n^2+t^2).
\@
These inequalities hold for every $t\in \RR$. 

The lemma below states that one-coordinate replacement in data does not change the tilted objective too much; see Appendix~\ref{app:subsec_proof_lem_opt_stab} for the proof. Its proof combines a graph perspective of matching, as well as the abundance of control samples in each cell $Q$ split as before. 

\begin{lemma}\label{lem:stab_opt_value}
    If $D\in \cE_n$ and $D'$ is a one-coordinate replacement of $D$ by any point in $\Omega$, then $|V_D(t)-V_{D'}(t)|\leq C H_n(t)/n$ uniformly for every $|t|\leq 1$. 
\end{lemma}

Such stability allows us to apply the Efron--Stein inequality to control the tilted loss. 
To this end, for $x\in\Omega$, let $D^{i\to x}$ be the dataset obtained by replacing only $T_i$ by $x$ in $D$. Define the deterministic function
\[
 \mu_t(x)=\EE V_{D^{1\to x}}(t)
 =\EE_{T_2,\ldots,T_{n_1},C_1,\ldots,C_{n_0}}
    V_{(x,T_2,\ldots,T_{n_1},C_1,\ldots,C_{n_0})}(t),
\]
where the expectation is taken over $(T_2,\dots,T_{n_1},C_1,\dots,C_{n_0})$. 
By independence among data we know $\EE V_D(t)=\int_\Omega\mu_t(x)p_1(x)\ud x$. 
The proof of the next lemma is in Appendix~\ref{app:subsec_lem_value_bounds}.

\begin{lemma}\label{lem:value-bounds}
Choose $\cgrid$ sufficiently large. Uniformly over deterministic $|t|\le1$,
\begin{equation}\label{eq:value-bounds}
 \Var \big( V_D(t)\big) \le\frac{CH_n(t)^2}{n},\qquad
 \sup_{x\in\Omega}|\mu_t(x)-\EE [V_D(t)]|\le\frac{CH_n(t)}n.
\end{equation} 
Furthermore, for every deterministic \(t\in[-1,1]\)  and every measurable optimizer of $V_D(t)$, i.e., any measurable map $s_t:\Omega^n\to \mathcal A_{n_1,n_0}$ satisfying 
$s_t(D)\in \argmin_{\sigma} 
\left\{Q_D(\sigma,A)+tB_D(\sigma,g)\right\}$ almost surely, we have 
$\EE[L_D(s_t(D))]\leq C(r_n^2+t^2)$. 
\end{lemma}

After controlling the variance and Euclidean distances for the tilted-optimal matching, we proceed to control the expectation of the slope of $V_D(t)$. Lemma~\ref{lem:slope} establishes this slope equals the imbalance for the tilted-optimal $\sigma_t$, which will be used in combination of the fact that $\sigma_0 = \sigma_{D,A}^\star$ to control the imbalance.

\begin{lemma}\label{lem:slope}
For every deterministic $|t|\le1$, $V_D'(t)$ exists almost surely, $V_D'(t)=B_D(\sigma_t,g)$ almost surely, and $|\EE V_D'(t)|\le C\{H_n(t)+|t|\}$, where $\sigma_t$ is defined in~\eqref{eq:def_Vd_minimizer}.
\end{lemma}

The proof of Lemma~\ref{lem:slope} uses   integration-by-parts to bound $\EE[V_D'(t)]$ by the summation of (i) several integrals that involve $\mu_t(x)-\EE[V_D(t)]$, and (ii) multiples of $\EE[L_D(\sigma_t)]$. Both can be controlled based on components in Lemma~\ref{lem:value-bounds}. The detailed proof is in Appendix~\ref{app:subsec_lem_slope}.


We now combine Lemmas~\ref{lem:value-bounds} and~\ref{lem:slope} to complete the proof of Theorem~\ref{thm:balance}.
For a deterministic $0<h\leq 1/\log n$, we define the left and right difference quotients
\@\label{eq:def_Rh_Lh}
R_h = \frac{V_D(h)-V_D(0)}{h},\quad L_h = \frac{V_D(0)-V_D(-h)}{h}.
\@
Lemma~\ref{lem:envelope} applied to $a_\sigma = Q_D(\sigma,A)$ and $b_\sigma = B_{D}(\sigma, g)$ for finitely many $\sigma$ (replacing the $j$-indices) gives 
\@\label{eq:Bd_sandwich}
R_h \leq B_{D}(\sigma_0,  g) \leq L_h,
\@
where $\sigma_0$ is the minimizer of $V_D(t)$ at $t=0$. By definition, $\sigma_0 = \sigma_{D,A}^\star$ is the global minimizer of the matching objective only without the tilt. By Lemma~\ref{lem:slope}, $V_D'(s)$ exists almost surely and equals the imbalance measure at $\sigma_s$. 
Taking $s=0$ then gives $B_ D (\sigma_{D,A}^\star, g) = V_D'(0)$ and thus by Lemma~\ref{lem:slope} we know 
\$
\big| \EE[B_ D (\sigma_{D,A}^\star, g)] \big| \leq C H_n(0) = C\Gamma_n,
\$
which proves the first inequality in~\eqref{eq:balance-moments} (using the updated notation here). 
In the following, we aim to bound $R_h$ and $L_h$ for fixed values of $h$ and then choose $h$ appropriately to prove the second inequality in~\eqref{eq:balance-moments}.

For each dataset $D$, by Lemma~\ref{lem:envelope}, the mapping $s\mapsto V_D(s)$ is continuous, picewise affine, and absolutely continuous on $[-h,h]$. By Lemma~\ref{lem:slope}, $V_D'(s)$ exists almost surely and equals the imbalance measure at $\sigma_s$, hence by definition $|V_D'(s)|\leq 2\|g\|_\infty$ is uniformly bounded. Applying  Fubini's theorem and Lemma~\ref{lem:slope},
\@\label{eq:bd_Rh_expect}
\big|\EE[R_h]\big|  = \bigg|\frac{1}{h}\int_0^h \EE[V_D'(s)] \ud s \bigg| 
\leq \frac{C}{h}\int_0^h \{H_n(s)+|s|\}\ud s = C(\Gamma_n + h^2 \log n/3 + h/2) 
\leq C(\Gamma_n + h)
\@
since $h \log n\leq 1$. 
Similarly the same arguments for $L_h$ gives 
\$
\big|\EE[L_h] \big|= \bigg|\frac{1}{h}\int_{-h}^0 \EE[V_D'(s)] \ud s\bigg| 
\leq C(\Gamma_n + h ). 
\$
Furthermore, writing $\|W\|_2= (\EE[W^2])^{1/2}$, we can bound the standard deviation of $R_h$ by 
\$
\text{sd}(R_h) &= \frac{1}{h}\big\|\{V_D(h)-\EE[V_D(h)] \} - \{V_D(0)-\EE[V_D(0)]\}\big\|_2 \\ 
&\leq \frac{\text{sd}(V_D(h)) + \text{sd}(V_D(0))}{h} \leq \frac{C\{H_n(h)+H_n(0)\}}{\sqrt{n} h} 
\leq C  \Big( \frac{\Gamma_n}{\sqrt{n}h} + \frac{h\log n}{\sqrt{n}}\Big).
\$
where the first inequality uses triangle inequality, the second uses Lemma~\ref{lem:value-bounds}, and the last uses the definition of $H_n(t)$. 
The same idea yields exactly the same bound for $\text{sd}(L_h)$. Since $\log n \leq \sqrt{n}$ for all sufficiently large $n$, $\text{sd}(R_h)\vee \text{sd}(L_h) \leq C\{\Gamma_n/(\sqrt{n}h) +h\}$. 
Combining the bounds on the expectation and variance, we know 
\@\label{eq:Rh_Lh_L2_bd}
\EE[R_h^2]+\EE[L_h^2]\leq C\Big\{ (\Gamma_n +  h)^2 + [\Gamma_n/(\sqrt{n}h) +h]^2 \Big\}
\leq C\Big\{ (\Gamma_n +  h)^2 + \frac{\Gamma_n^2}{nh^2}  \Big\}
\@
where the value of the absolute constant $C$ can change across equations, and we used the inequality $(a+b)^2\leq 2a^2+2b^2$, absorbing the $h^2$ term into $(\Gamma_n+h)^2$. Going back to~\eqref{eq:Bd_sandwich}, we know 
\$
\EE[B_D(\sigma_{D,A}^\star, g)^2] 
\leq \EE\big[ (|R_h|+|L_h|)^2 \big] 
\leq 2\EE[R_h^2]+2\EE[L_h^2] \leq C\Big\{ (\Gamma_n +  h)^2 + \frac{\Gamma_n^2}{nh^2}  \Big\}.
\$
We now choose the deterministic $h_n=(\Gamma_n/\sqrt{n})^{1/2}$, which gives $h_n^2 = \Gamma_n/\sqrt{n}$ and $\Gamma_n^2/(nh^2) = \Gamma_n/\sqrt{n}$, and
\$
\EE[B_D(\sigma_{D,A}^\star, g)^2] \leq C \big\{\Gamma_n^2 +  \Gamma_n / \sqrt{n} \big\},
\$
which proves the second inequality in~\eqref{eq:balance-moments}. 

Plugging in the definition $\Gamma_n = r_n^2 \log n = (\log n)^{1+2/d}/n^{2/d}$, and by Markov's inequality, we know 
\$
B_D(\sigma_{D,A}^\star, g) = O_P \big(\Gamma_n + (\Gamma_n/\sqrt{n})^{1/2}\big)
= O_P \big((\log n)^{1+2/d}/n^{2/d} + 
(\log n)^{1/2+1/d}/n^{1/d+1/4} \big),
\$
which is~\eqref{eq:balance-rate}. Finally, we note that 
\$
\EE\big[ (\sqrt{n}B_D(\sigma_{D,A}^\star, g))^2 \big] 
\leq C\{ n\Gamma_n^2 + \sqrt{n} \Gamma_n\}
= C \big\{ (\log n)^{2+4/d} n^{1-4/d} + (\log n)^{1+2/d} n^{1/2-2/d} \big\} \to 0
\$
when $d\leq 3$, which proves the last claim of Theorem~\ref{thm:balance}.

\subsection{Allowing optimization error and data-dependent matrix}
\label{app:subsec_opt_hatA}

We now extend Theorem~\ref{thm:balance} in two ways. First, we allow the chosen matching to be subject to optimization error of order $O(1/n)$. Second, we allow the matrix used in the metric to be data-dependent. 

Recall the definition of the matching objective $Q_D(\sigma,A) = Q_n(\sigma;A)$ defined in~\eqref{eq:def_matching_obj} where we made the dependence on the data $D$ and matrix $A$ explicit. The exact global minimizer is denoted as $\sigma_{D,A}^\star$. The imbalance measure is denoted as $B_D(\sigma,g) = B_{n,\sigma}(g)$ defined in~\eqref{eq:def_imbalance_g} with explicit notational dependence.  

Let $\hat\sigma\in \cA_{n_1,n_0}$ be any injection with the optimization gap 
\$
\Delta_D(\hat\sigma,A):= Q_{D}(\hat\sigma,A) - Q_D(\sigma_{D,A}^\star,A) \geq 0.
\$

\begin{theorem}\label{thm:obj_gap}
    Under Assumption~\ref{ass:cov}, for any deterministic $h_n \in (0,  1/\log n]$, it holds that 
    \@\label{eq:obj_gap}
    |B_D( \hat\sigma, g) |\leq Z_{n,h_n} + \Delta_n(\hat\sigma)/h_n,\quad Z_{n,h_n} =O_P\Big(\Gamma_n +h_n + \frac{\Gamma_n}{\sqrt{n} h_n} \Big).
    \@
    Here we define $Z_{n,h}=|R_h|+|L_h|$. In particular, if $\Delta_D(\hat\sigma,A) = O_P(\epsilon_n)$ for a deterministic sequence of $\epsilon_n \geq 0$ obeying $\epsilon_n = o((\log n)^{-2})$, then 
    \@\label{eq:obj_gap_2}
|B_D( \hat\sigma, g) | = O_P\Big( \Gamma_n + (\Gamma_n/\sqrt{n} + \epsilon_n)^{1/2} \Big),
    \@
    and $L_D(\hat\sigma) = O_P(r_n^2+\epsilon_n)$. 
    If $d\leq3$ and $\Delta_D(\hat\sigma,A)=o_P(1/n)$, then 
    $B_D( \hat\sigma, g) = o_P(n^{-1/2})$ and $L_D(\hat\sigma) = o_P(1)$.  
\end{theorem}

\begin{proof}[Proof of Theorem~\ref{thm:obj_gap}]
By definition of $V_D(t)$ and $\Delta_D(\hat\sigma,A)$, we have 
\$
&V_D(h) \leq Q_D(\hat\sigma,A) + h B_D( \hat\sigma, g) 
\leq V_D(0) + \Delta_{D}(\hat\sigma,A) +  h B_D( \hat\sigma, g),\\ 
&\text{and}\quad V_D(-h) \leq V_D(0) + \Delta_{D}(\hat\sigma,A) - h B_D( \hat\sigma, g).
\$
Then we have 
\$
R_h - \Delta_{D}(\hat\sigma,A)/h \leq B_D( \hat\sigma, g) \leq L_h + \Delta_{D}(\hat\sigma,A)/h,
\$
where we recall the definition of $R_h$ and $L_h$ in~\eqref{eq:def_Rh_Lh}. 
Thus~\eqref{eq:obj_gap} follows from the bounds on $R_h$ and $L_h$ we obtained in the proof of Theorem~\ref{thm:balance}, e.g., equation~\eqref{eq:Rh_Lh_L2_bd}. 

Now choose $h_n = \sqrt{\Gamma_n/\sqrt{n}+\epsilon_n}$ which is smaller than $1/\log n$ for all sufficiently large $n$, and $\Gamma_n/(\sqrt{n}h_n)$ and $\epsilon_n/h_n$ are both upper bounded by $h_n$. 
Furthermore, we consider the event $\cE_n$ on which there are sufficiently many controls in each cell defined in Lemma~\ref{lem:count_stable} in Appendix~\ref{app:subsec_local_cells}. On the event $\cE_n$, each treated vertex can be matched to a distinct control with a Euclidean distance below $r_n$, which gives $Q_{D}(\sigma_{D,A}^\star,A) \leq C r_n^2$. On $\cE_n^c$, we have $Q_{D}(\sigma_{D,A}^\star,A) \leq b_A d$, and Lemma~\ref{lem:count_stable} gives for every fixed $b>0$, one can find  a sufficiently large constant $\cgrid$ so $\PP(\cE_n^c)\leq n^{-b}$ when $n$ is large enough. Combining the bounds and choosing $b$ appropriately gives $Q_{D}(\sigma_{D,A}^\star,A) = O_P(r_n^2)$. The eigenvalue condition on $A$ yields 
\$
(a_A/2) L_D(\hat\sigma) \leq Q_D(\hat\sigma,A) \leq Q_{D}(\sigma_{D,A}^\star,A) + \Delta_D(\hat\sigma,A) = O_P(r_n^2 + \epsilon_n),
\$
so $L_D(\hat\sigma)=O_P(r_n^2+\epsilon_n)$. 

Finally, since $\Delta_D(\hat\sigma,A)=o_P(1/n)$, there exists a deterministic sequence $b_n\to 0$ such that $\PP(n\Delta_D(\hat\sigma,A)>b_n)\to 0$. To see this, for each integer $k$ one can choose $n_k$ so $\PP(n\Delta_D(\hat\sigma,A)>1/k)<1/k$ for all $n\geq n_k$, and then choose $b_n = 1/k$ between successive $n_k$. 
Now set $\epsilon_n=b_n/n$, then $\Delta_D(\hat\sigma,A) = O_P(\epsilon_n)$ and $n\epsilon_n = o(1)$. Putting this back to~\eqref{eq:obj_gap_2}, when $d\leq 3$, we know $\sqrt{n}\Gamma_n \to 0$ and $n(\Gamma_n/\sqrt{n}+\epsilon_n) = o(1)+n\epsilon_n = o(1)$. This leads to $B_D( \hat\sigma, g) = o_P(n^{-1/2})$. Finally, $L_D(\hat\sigma)=O_P(r_n^2+\epsilon_n) = o_P(1)$ by definition. 
\end{proof}

Next, we establish imbalance bounds when the matrix $A$ is replaced by some data-dependent $\hat{A}$ that converges. 
Let $A_0$ be a fixed positive definite matrix, and let $\hat{A}$ be any data-dependent, random matrix that is symmetric and positive definite on an event with probability tending to one. 
Let $\hat\sigma\in \cA_{n_1,n_0}$ be any injection, and define its optimization gap for the matching objective using $\hat{A}$ as
\$
\xi_n := Q_D(\hat\sigma,\hat{A}) - Q_D(\sigma_{D,\hat{A}}^\star, \hat{A}). 
\$

\begin{theorem}\label{thm:hat_A_imbalance}
    Suppose $\norm{\hat A-A_0}_{\mathrm{op}}=O_P(a_n)$ for a  deterministic sequence $a_n\to0$, and $\xi_n=O_P(\eta_n)$ for a deterministic sequence $\eta_n\to0$. Then the optimization gap under the oracle balancing objective satisfies 
\begin{equation}\label{eq:metric-gap}
 Q_{D}(\hat\sigma,A_0)-\min_\sigma Q_{D}(\sigma,A_0)
 =O_P(a_n r_n^2+\eta_n).
\end{equation}
Whenever $(a_n r_n^2+\eta_n)(\log n)^2\to0$, equation~\eqref{eq:obj_gap_2} applies with $\epsilon_n=a_n r_n^2+\eta_n$. 
Moreover, if $\norm{\hat A-A_0}_{\mathrm{op}}=O_P(n^{-1/2})$, then~\eqref{eq:balance-rate} holds under $\hat{A}$ for the corresponding exact minimizer, i.e., 
\begin{equation}\label{eq:balance-rate-hat-A}
 B_{D}(\hat\sigma, g)=O_P \big(  
 n^{-2/d}(\log n)^{1+2/d}
 +n^{-1/4-1/d}(\log n)^{1/2+1/d} \big),
\end{equation}
which obeys $\sqrt{n} B_{D}(\hat\sigma, g) =o_P(1)$ for $d\leq 3$. 
Finally, if $\xi_n = O_P(1/n)$, then $B_D(\hat\sigma, g)=o_P(n^{-1/2})$ when  $d\leq 3$.
\end{theorem}

\begin{proof}[Proof of Theorem~\ref{thm:hat_A_imbalance}]
Work on the event that $\hat A$ has eigenvalues between $a_{A_0}/2$ and $2b_{A_0}$ and that $\cE_n$ defined in Lemma~\ref{lem:count_stable} holds. Its probability tends to one. 
Since on $\cE_n$ one can match all  treated  units to distinct control units in the same cell, the oracle optimizer $\sigma_{D,A_0}^\star$ has $L_D(\sigma_{D,A_0}^\star)\le Cr_n^2$, while $L_D(\hat\sigma)\le C(r_n^2+\xi_n)$ applying the eigenvalue condition on $\hat{A}$ on this event. 

For any injection $\sigma$, we know $|Q_{D}(\sigma,A_0)-Q_{D}(\sigma,\hat{A})|\le\frac12\norm{\hat A-A_0}_{\mathrm{op}}\cdot L_D(\sigma)$. Thus 
\$
0 &\leq Q_D(\hat\sigma,A_0) - \min_{\sigma} Q_D(\sigma, A_0) \\
&= \big\{ Q_D(\hat\sigma,A_0) - Q_D(\hat\sigma,\hat{A}) \big\}
+ \big\{ Q_D(\hat\sigma,\hat{A}) - Q_D(\sigma_{D,\hat{A}}^\star, \hat{A}) \big\}
+  \big\{ Q_D(\sigma_{D,A_0}^\star, \hat{A}) - Q_D(\sigma_{D,A_0}^\star,A_0) \big\}.
\$
Applying the above bounds for the first and third terms and the optimization gap bound for the second yield
\$
0 \leq Q_D(\hat\sigma,A_0) - \min_{\sigma} Q_D(\sigma, A_0) &\leq \frac12\norm{\hat A-A_0}_{\mathrm{op}}\cdot  \big\{ L_D(\hat\sigma) + L_D(\sigma_{D,A_0}^\star)\big\} + \xi_n \\ 
&\leq C \norm{\hat A-A_0}_{\mathrm{op}}\cdot r_n^2 + (1+ C\norm{\hat A-A_0}_{\mathrm{op}}) \xi_n
\$
which proves \eqref{eq:metric-gap}. If $a_n=n^{-1/2}$ and $\eta_n=0$, 
then optimization gap becomes
\$
0 \leq Q_D(\hat\sigma,A_0) - \min_{\sigma} Q_D(\sigma, A_0)  \leq O_P(r_n^2/\sqrt{n}) ,
\$ 
where 
$r_n^2/\sqrt n = o(\Gamma_n/\sqrt{n})$, so plugging into~\eqref{eq:obj_gap_2} with $\epsilon_n = \Gamma_n/\sqrt{n} = o((\log n)^{-2})$ yields 
\$
|B_{D}(\hat\sigma, g)| = O_P(\Gamma_n + (\Gamma_n/\sqrt{n} +\epsilon_n)^{1/2})
= O_P(\Gamma_n + (\Gamma_n/\sqrt{n})^{1/2}),
\$
and therefore~\eqref{eq:balance-rate-hat-A}  and the root-$n$ bound follow from the definition of $\Gamma_n$. 
Finally, if $\xi_n=o_P(1/n)$, then the optimization gap becomes 
\$
0 \leq Q_D(\hat\sigma,A_0) - \min_{\sigma} Q_D(\sigma, A_0)  \leq O_P(r_n^2/\sqrt{n}) + o_P(1/n) = o_P(1/n),
\$ 
since for $d\leq 3$  we know $r_n^2/\sqrt{n} = (\log n)^{2/d} n^{-1/2-2/d} = o_P(1/n)$. Thus by the last claim in Theorem~\ref{thm:obj_gap} we know 
$B_{D}(\hat\sigma,g)= o_P(1/\sqrt{n})$. This concludes the proof. 
\end{proof}

By Theorems~\ref{thm:obj_gap} and~\ref{thm:hat_A_imbalance}, we obtain the $o_P(1/\sqrt{n})$ rate for the imbalance even with (i) an approximate solution with $o_P(1/n)$ optimization gap, and (ii) a data-dependent matrix which converges in $O_P(1/\sqrt{n})$ rate to a fixed matrix. In practice, this corresponds to using  inverse sample covariance matrices such as the Mahalanobis distance.

\subsection{Uniform balance by covering}
\label{app:subsec_covering}

In this section, we expand the imbalance bound for a fixed function $g$ to the function class $\cF_L$ using the standard chaining trick. 
We inherit the definitions of $r_n$, $\Gamma_n$, and $H_n(t)$ in the last subsection. 

To clearly state the balance regarding any function $g$, we update the notations so throughout this subsection, the tilted objective and the difference quotients are denote by 
\$
V_{D,g}(t) = \min_{\sigma\in \cA_{n_1,n_0}} \Big\{ Q_D(\sigma, A_0) + tB_D(\sigma,g) \Big\}, \quad 
S_{D,t}(g) = \frac{V_{D,g}(t)-V_{D,g}(0)}{t},\quad t\neq 0.
\$
We denote the minimizer that achieves $V_{D,g}(t)$ by $\sigma_{D,g}(t)$. 
Since the expectation and variance bounds in Appendix~\ref{app:subsec_imbalance_fixed_theorem} are with respect to fixed quantities, those results especially~\eqref{eq:bd_Rh_expect} and~\eqref{eq:Rh_Lh_L2_bd} imply that, for all sufficiently  large $n$ and all deterministic $|t|\in (0, 1/\log n]$, 
\@\label{eq:bd_SDt}
\sup_{g\in \cF_L} \big|\EE[S_{D,t}(g)] \big| \leq C(\Gamma_n+t), \qquad 
\sup_{g\in \cF_L} \EE\big[ S_{D,t}(g)^2\big] 
\leq C \Big\{ (\Gamma_n + |t|)^2 + \frac{\Gamma_n^2}{nt^2}\Big\}.
\@
These constants $C$ are universal over $g\in \cF_L$, since the earlier bounds use $g$
only through $\|g\|_\infty$, $\sup_\Omega\|\nabla g\|$, and
$\sup_\Omega\|\nabla^2g\|_{\mathrm{op}}$, which are at most $L$,
$\sqrt d\,L$, and $dL$, respectively. 
The event $\cE_n$ defined in Lemma~\ref{lem:count_stable} only uses the distributions of the covariates and does not involve $g$. 

We note a continuity result whose proof is in Appendix~\ref{app:subsec_lem_cont_SDt}. 
\begin{lemma} \label{lem:cont_SDt}
Following the definitions above, one has 
    \@
|S_{D,t}(g) - S_{D,t}(f) |\leq 2\|g-f\|_{\infty},\quad g,f\in \cF_L, ~t\neq 0.
\@
\end{lemma}

We write the entire dataset as $D=(W_1,\dots,W_n)=(T_1,\dots,T_{n_1},C_1,\dots,C_{n_0})$. 
For each $k\in \{1,\dots,n\}$, let $D^{(k)} = (W_1,\dots,W_{k-1},\tilde{W}_k,W_{k+1},\dots,W_n)$ be the data obtained by replacing the $k$-th data vector by an independent copy from the corresponding treated/control group, with other data vectors unchanged. 

Lemma~\ref{lem:unif:replacement} shows the quotient is stable with respect to the one-coordinate replacement. Its proof is in Appendix~\ref{app:unif:replacement-proof}, which uses similar ideas and the graphical results used in the proof of Lemma~\ref{lem:stab_opt_value}.

\begin{lemma}
\label{lem:unif:replacement}
For every fixed $g,f\in\mathcal F_L$, deterministic $t\ne0$, and
$k\in\{1,\ldots,n\}$,
\begin{equation}\label{eq:unif:replacement}
 \left|\{S_{D,t}(g)-S_{D,t}(f)\}
       -\{S_{D^{(k)},t}(g)-S_{D^{(k)},t}(f)\}\right|
 \le\frac4{n_1}\|g-f\|_\infty
 \quad\text{almost surely}.
\end{equation}
\end{lemma}

We now define the centered process
\@\label{eq:unif:centered-process}
 \mathbb Z_{n,t}(g)=\sqrt n\{S_{D,t}(g)-\mathbb E S_{D,t}(g)\}.
\@
For fixed $g,f\in \cF_L$, we apply Lemma~\ref{lem:unif:bd} to the bounded random variable $\sqrt{n}\{S_{D,t}(g)-S_{D,t}(f)\}$. 
Then by Lemma~\ref{lem:unif:replacement} we can take $c_k = 4\sqrt{n}\|g-f\|_{\infty}/n_1$, and hence $\sum_{k=1}^n c_k^2 = 16n^2\|g-f\|_\infty^2 /n_1^2$ which is upper bounded by a constant $16\|g-f\|_\infty^2/a_s^2$ based on our sample size assumption. Thus, for $u>0$ and $\|g-f\|_\infty>0$, we have 
\@\label{eq:bd_subgaussian}
\PP\big( | \mathbb Z_{n,t}(g) - \mathbb Z_{n,t}(f) |>u \big) \leq 2\exp\Big\{-\frac{u^2}{K\|g-f\|_\infty^2}\Big\},\quad K = 8/a_s^2. 
\@
Moreover, Lemma~\ref{lem:cont_SDt} gives $| \mathbb Z_{n,t}(g) - \mathbb Z_{n,t}(f) | \leq 4\sqrt{n} \|g-f\|_\infty$ for each dataset. 

The next lemma provides the uniform convergence leveraging the standard covering and chaining techniques; see Lemmas~\ref{lem:unif:entropy} and~\ref{lem:unif:chain}. We isolate this result to make the ideas clear. 

\begin{lemma} 
\label{lem:unif:quotients}
Suppose $d\le3$. Let $h_n\in(0,1/\log n]$ be a deterministic sequence with  $nh_n^2+\frac{\Gamma_n^2}{h_n^2}\to 0$. 
Then
\begin{equation}\label{eq:unif:secants-conclusion}
 \sqrt n\sup_{g\in\mathcal F_L}|S_{D,h_n}(g)|
 +\sqrt n\sup_{g\in\mathcal F_L}|S_{D,-h_n}(g)|
 \to 0\quad\text{in probability}.
\end{equation}
\end{lemma}

\begin{proof}[Proof of Lemma~\ref{lem:unif:quotients}]
We prove the result for $t_n:=h_n$ first. 
Since $d\le3$, plugging in the definition gives 
$n\Gamma_n^2\to0$. By \eqref{eq:bd_SDt}, for every fixed $g$, 
\[
 \mathbb E\{\mathbb Z_{n,t_n}(g)^2\}
 =n\operatorname{Var}\{S_{D,t_n}(g)\}
 \le C\left\{n(\Gamma_n+h_n)^2+
                       \frac{\Gamma_n^2}{h_n^2}\right\} \to 0. 
\]
Here $n(\Gamma_n+h_n)^2\le2n\Gamma_n^2+2nh_n^2$. 
Thus $\mathbb Z_{n,t_n}(g) = o_P(1)$ for any fixed $g$. 
By~\eqref{eq:bd_subgaussian}, Lemma~\ref{lem:unif:chain}
therefore gives
$\sup_{g\in\mathcal F_L}|\mathbb Z_{n,t_n}(g)|=o_P(1)$. On the other hand, we have  
$\sqrt n\sup_g|\mathbb E[ S_{D,t_n}(g)]|
\le C\sqrt n(\Gamma_n+h_n)\to0$ by the same conditions by~\eqref{eq:bd_subgaussian}. Thus by the triangle inequality and by the definition of $\ZZ_{n,t}(g)$, 
\$
\sqrt n\sup_{g\in \cF_L}|S_{D,t_n}(g)| \leq \sup_{g\in \cF_L} \mathbb{Z}_{n,t_n}(g) +\sqrt n\sup_{g\in \cF_L}|\mathbb E[ S_{D,t_n}(g)]| = o_P(1). 
\$
The result for $t_n:=-h_n$ follows in a similar way, which can be combined using a union bound. 
\end{proof}

\paragraph{Competing the proof of Theorem~\ref{thm:uniform_balance}.} We are now ready to complete the entire proof. 

Let $\hat\sigma_n$ be a measurable random injection and write
$\Delta_n=\Delta_D(\hat\sigma_n,A_0)=o_{\mathbb P}(n^{-1})$.
Choose a deterministic sequence $\epsilon_n>0$ such that
$n\epsilon_n\to0$ and $\Delta_n=O_{\mathbb P}(\epsilon_n)$. This is feasible as we argued in the proof of Theorem~\ref{thm:obj_gap}.  
Set $h_n= (\frac{\Gamma_n}{\sqrt n}+\epsilon_n )^{1/2}$. 
Since $\sqrt n\Gamma_n\to0$ for $d\le3$, we know that 
\[
 nh_n^2=\sqrt n\Gamma_n+n\epsilon_n\to 0,
 \qquad
 \Gamma_n^2 / h_n^2\le\sqrt n\Gamma_n\to 0.
\]
Also $h_n\log n=(\sqrt n h_n)(\log n/\sqrt n)\to0$, so condition in 
Lemma~\ref{lem:unif:quotients} is fulfilled for all sufficiently large $n$. 

Similar to the proofs in Section~\ref{app:subsec_opt_hatA}, for every dataset, 
every $h>0$, and every $g\in\mathcal F_L$, we have
\begin{equation}\label{eq:unif:gap-sandwich}
 S_{D,h}(g)-\frac{\Delta_n}{h}
 \le B_D(\hat\sigma_n,g)
 \le S_{D,-h}(g)+\frac{\Delta_n}{h}.
\end{equation} Taking the supremum and
using $\epsilon_n\le h_n^2$ gives
\$ 
 \sqrt n\sup_{g\in\mathcal F_L}|B_D(\hat\sigma_n,g)|
 &\le\sqrt n\sup_g|S_{D,h_n}(g)|
     +\sqrt n\sup_g|S_{D,-h_n}(g)|
     +\frac{\sqrt n\Delta_n}{h_n} o_P(1)
   +O_P\!\left(\frac{\sqrt n\epsilon_n}{h_n}\right)
  =o_P(1),
\$
because $\sqrt n\epsilon_n/h_n\le\sqrt{n\epsilon_n}\to0$.
For a fixed realized injection $\hat\sigma_n$, the mapping $g\mapsto B_D(\hat\sigma_n,g)$ is
$2$-Lipschitz in the sup norm, so its supremum is also measurable. This proves Theorem~\ref{thm:uniform_balance}.



\subsection{Proof of Theorem~\ref{thm:observational}}
\label{app:subsec_proof_obs}

\begin{proof}[Proof of Theorem~\ref{thm:observational}]
Fix $L\in(0,\infty)$ and write $\pi=\mathbb P(Z=1)$.
We set all the optimization gaps and imbalance measures in this proof to zero on the event 
$\mathcal I_n^c=\{n_1<2\}\cup\{n_1>n_0\}$, as stated earlier.

The propensity bounds imply
$\eta_0\le\pi\le1/2-\eta_1$.
Conditional on the treatment label vector, the
grouped covariates have distribution
$P_1^{\otimes n_1}\otimes P_0^{\otimes n_0}$, with densities
\[
p_1(x)=\frac{f(x)e(x)}{\pi},
\qquad
p_0(x)=\frac{f(x)\{1-e(x)\}}{1-\pi},
\]
where $n_1=\sum_{i=1}^n Z_i$ and $n_0=n-n_1$, and both are random counts. 
Let $\rho=(1/2+\eta_1)/(1/2-\eta_1)>1$. 
The density and derivative bounds required by Assumption~\ref{ass:cov} therefore follow from Assumption~\ref{assump:obs-sample}.

For $q=n_1/n$, we know 
\[
\inf_{x\in\Omega}\frac{(n-k)p_0(x)}{kp_1(x)}
\ge \frac{1-q}{q}\frac{\pi}{1-\pi}\rho.
\]
The right side is continuous in $q\in(0,1)$ and equals $\rho$
at $q=\pi$. Choose fixed $\delta>0$ and $\eta>0$ so that
$[\pi-\eta,\pi+\eta]\subseteq(0,1/2)$, $\eta<\pi$, and the last
display is at least $1+\delta$ whenever $|q-\pi|\le\eta$.
Define
\[
\mathcal K_n
=\{k\in\{0,\ldots,n\}:|k/n-\pi|\le\eta\},
\qquad
\mathcal J_n=\{n_1\in\mathcal K_n\}.
\]
For all sufficiently large $n$, we know counts in $\mathcal K_n$
satisfy $2\le k\le n-k$, as well as the lower bounds on
$k/n$ and $(n-k)/n$ and the margin $\delta$ required in Assumption~\ref{ass:cov}. 
Since $n_1\sim\operatorname{Bin}(n,\pi)$, 
Chernoff's inequality gives 
$\mathbb P(\mathcal J_n^c)\le2e^{-cn}$ for a fixed $c>0$.
In particular, $\mathbb P(\mathcal I_n)\to1$.

Write $a_0=\lambda_{\min}(A_0)>0$ and
$\delta_{A,n}=\|\hat A-A_0\|_{\mathrm{op}}$.
Use the same constants $r_n$, $\Gamma_n$ and $H_n(t)$ for all $k\in\mathcal K_n$. 
We consider the local-count event $\cE_n$ defined in Lemma~\ref{lem:count_stable}; 
following the same ideas there, for any fixed $b>0$, we can choose $c_{\text{grid}}$  so that
\[
\sup_{k\in\mathcal K_n}
\mathbb P(\mathcal E_n^c\mid n_1=k)\le n^{-b}
\]
for all sufficiently large $n$.
It follows by averaging that
$\mathbb P(\mathcal J_n\cap\mathcal E_n^c)\le n^{-b}$.
Also $\mathbb P(\delta_{A,n}\le a_0/2)\to1$.
Thus the event
$\mathcal J_n\cap\mathcal E_n\cap\{\delta_{A,n}\le a_0/2\}$
has probability tending to one.

On this event, the proofs of Theorem~\ref{thm:obj_gap} and Theorem~\ref{thm:hat_A_imbalance} give
\[
\begin{split}
\Delta_n
&:=Q_n(\hat\sigma_n;A_0)
       -\min_{\sigma\in\mathcal A_{n_1,n_0}}Q_n(\sigma;A_0) \le C\delta_{A,n}r_n^2+(1+C\delta_{A,n})\xi_n.
\end{split}
\] 
All constants are common to $k\in\mathcal K_n$.
The assumed unconditional rates therefore imply
\[
\Delta_n
=O_P(n^{-1/2}r_n^2)+o_P(n^{-1})
=o_P(n^{-1}),
\]
since $\sqrt n\,r_n^2\asymp
n^{1/2-2/d}(\log n)^{2/d}\to0$ for $d\le3$.
Choose a deterministic $\epsilon_n>0$ such that
$n\epsilon_n\to0$ and
$\mathbb P(\Delta_n>\epsilon_n)\to0$.
Such a sequence exists because $n\Delta_n\to0$ in probability,
by the construction near the end of the proof of Theorem~\ref{thm:obj_gap}.

For $k\in\mathcal K_n$, let
$\mathbb P_{n,k}=P_1^{\otimes k}\otimes P_0^{\otimes(n-k)}$.
We then write 
$\Delta_D(\sigma,A_0)=Q_D(\sigma,A_0)-\min_\tau Q_D(\tau,A_0)$
and  
\[
R_{n,k}(D)
=
\max_{\substack{\sigma\in\mathcal A_{k,n-k}\\
                 \Delta_D(\sigma,A_0)\le\epsilon_n}}
\sqrt n\sup_{g\in\mathcal F_L}|B_D(\sigma,g)|.
\]
The maximum is over a nonempty finite set, and the inner supremum is measurable since it takes 
the supremum over a fixed countable dense subset of $\mathcal F_L$. 
Choose the first maximizing injection in a fixed ordering and denote it as $\bar\sigma_{n,k}(D)$. Since it maximizes over finitely many measurable objectives, we know $\bar\sigma_{n,k}(D)$ is measurable.  

For every deterministic sequence $k_n\in\mathcal K_n$, all the
fixed-size assumptions hold with common constants and
$n\epsilon_n\to0$. 
Thus, Theorem~\ref{thm:uniform_balance} 
applied to $\bar\sigma_{n,k_n}$ gives 
$R_{n,k_n}\to0$ in $\mathbb P_{n,k_n}$-probability.
Consequently, 
\[
\sup_{k\in\mathcal K_n}
\PP_{n,k}\{R_{n,k}>u\}\to 0.
\]
for every $u>0$. 
To see this, if it failed, 
there would be a
subsequence of $(n,k_n)$ with 
$k_n\in \cK_n$ for which $\PP_{n,k_n}\{R_{n,k_n}>u\}\geq \delta$ for some constant $\delta>0$. 
This contradicts the stated result for the deterministic sequence $k_n$.

Now we return to the observational sample. 
On the event $\cJ_n\cap\{\Delta_n\le\epsilon_n\}$, the matching $\hat\sigma_n\in \cA_{n_1,n_0}$ satisfies $\Delta_{D}(\sigma,A_0)\leq \epsilon$.  
Thus, for every $u>0$,
\[
\begin{split}
\PP \Big\{\sqrt n\sup_{g\in\cF_L}|B_{n,\hat\sigma_n}(g)|>u\Big\}
&\le \PP(\cJ_n^c) +\PP(\Delta_n>\epsilon_n) + \sum_{k\in\cK_n}\PP(n_1=k) \PP_{n,k}\{R_{n,k}>u\}\\
&\le \PP(\cJ_n^c)
   +\PP(\Delta_n>\epsilon_n)
   +\sup_{k\in\cK_n}\PP_{n,k}\{R_{n,k}>u\}
\to 0.
\end{split}
\]
This concludes the proof of Theorem~\ref{thm:observational}.
\end{proof}

\subsection{Moment bounds under observational sampling}
\label{app:observational-moments}

\begin{prop} 
\label{prop:observational-moments}
Under Assumption~\ref{ass:cov}, fix a symmetric
positive-definite matrix $A_0$ and a function $g\in\mathcal F_L$.
On $\mathcal I_n=\{2\le n_1\le n_0\}$, let $\sigma_n^\star$
be an exact minimizer of $Q_n(\sigma;A_0)$, with the fixed
measurable tie-breaking rule. Set
$B_{n,\sigma_n^\star}(g)=0$ on $\mathcal I_n^c$.
Let $\Gamma_n=r_n^2\log n\asymp
n^{-2/d}(\log n)^{1+2/d}$ be the deterministic proof scale.
There is a constant $C_g<\infty$, independent of $n$, such that
for all sufficiently large $n$,
\[
|\mathbb E B_{n,\sigma_n^\star}(g)|\le C_g\Gamma_n,
\qquad
\mathbb E[B_{n,\sigma_n^\star}(g)^2]
\le C_g\left(\Gamma_n^2+\frac{\Gamma_n}{\sqrt n}\right).
\]
For $d\le3$, $\sqrt n B_{n,\sigma_n^\star}(g)\to0$ in $L^2$.
\end{prop}

\begin{proof}[Proof of Proposition~\ref{prop:observational-moments}]
Use $\mathcal K_n$ and $\mathcal J_n$ from the proof of Theorem~\ref{thm:observational}. 
Conditional on $n_1=k\in\mathcal K_n$, we obtain the deterministic two-sample law. 
By Theorem~\ref{thm:balance},  for every such $k$,
\[
\left|\mathbb E[B_{n,\sigma_n^\star}(g)\mid n_1=k]\right|
\le C_g\Gamma_n,
\qquad
\mathbb E[B_{n,\sigma_n^\star}(g)^2\mid n_1=k]
\le C_g\left(\Gamma_n^2+\frac{\Gamma_n}{\sqrt n}\right).
\]
for the common constant $C_g$. 
Otherwise for $k\notin \cK_n$, we have
$|B_{n,\sigma_n^\star}(g)|\le2\|g\|_\infty$.
Using
$\mathbb P(\mathcal J_n^c)\le2e^{-cn}$ gives
\[
\begin{split}
|\mathbb E B_{n,\sigma_n^\star}(g)|
&\le C_g\Gamma_n+4\|g\|_\infty e^{-cn},\\
\mathbb E[B_{n,\sigma_n^\star}(g)^2]
&\le C_g\left(\Gamma_n^2+\frac{\Gamma_n}{\sqrt n}\right)
   +8\|g\|_\infty^2e^{-cn}.
\end{split}
\]
The remainders can be absorbed into the first terms for all sufficiently large $n$. Finally, by definition we have
$n\Gamma_n^2+\sqrt n\Gamma_n\to0$ for $d\le3$,
which proves the $L^2$ conclusion.
\end{proof}

\section{Proof of supporting lemmas in Appendix~\ref{app:subsec_imbalance_fixed_theorem}}

\subsection{Matching as bipartite graph and its  stability}
\label{app:subsec_match_graph_setup}

For preparation, we represent the matching problem as bipartite graph and prove its stability when replacing one treated/control data point. Such stability will later be used together with the Efron--Stein inequality (Lemma~\ref{lem:ES}) to bound the imbalance. These will be useful in proving several supporting lemmas in Section~\ref{app:subsec_imbalance_fixed_theorem}. 

\paragraph{Matching and bipartite graph.}

Keep $0<n_1,n_0<n$ fixed. 
Any matching, an injection $\sigma$ from $\{1,\dots,n_1\}$ to $\{1,\dots,n_0\}$, can be equivalently represented by a \emph{graph} in which the $n_1$ treated units are put on the left side and the $n_0$ control units are put on the right side. 
Each unit is a \emph{vertex}. 
An \emph{edge} connects a treated unit with a control unit, i.e., $T_i$ and $C_{\sigma(i)}$. A control unit is connected to a treated unit if and only if they constitute a matched pair. The matching can thus be represented by the graph consisting of the vertices and edges. 
Since we consider one-to-one matching, each edge joins two vertices from distinct groups, so this is a bipartite group. 
Each treated vertex/unit must be connected by some edge to another control vertex/unit (having degree $1$), but a control vertex may have degree $1$ or $0$. 

Let $c_{ij}$ be the cost of pairing treated unit $i$ with control unit $j$. This gives the \emph{cost matrix}, a $n_1$-by-$n_0$ matrix $c=(c_{ij})$, with each row for a treated unit and each column for a control unit. 
For the quadratic cost~\eqref{eq:def_matching_obj}, we have $c_{ij} = \frac{1}{2}(T_i-C_j)^\top A_0(T_i-C_j)$. We define $\|T_i-C_j\|$ as the Euclidean length of an edge between $(i,j)$. 
By definition, replacing $T_i$ by another value only changes row $i$ in the cost matrix, and replacing $C_j$ by another value only changes column $j$ in the cost matrix. 

With these definitions ready, a matching design $\cM_\sigma=\{(i,\sigma(i)\colon 1\leq i\leq n_1\}$ can be represented by the injection $\sigma$ and equivalently the collection of edges that connects each $(i,\sigma(i))$. The matching objective~\eqref{eq:def_matching_obj} is therefore $Q_n(\sigma;A_0) = \sum_{i=1}^{n_1} c_{i\sigma(i)}$. Thus, a cost matrix defines a matching objective and the corresponding optimal solution that minimizes the matching cost. 

\paragraph{Selected control set.}
We define the \emph{selected control set}  $S_\sigma:=\{\sigma(i)\colon 1\leq i\leq n_1\}$. Those are the control vertices with degree $1$ in the corresponding graph. 
The key insight is that, since the imbalance measure~\eqref{eq:def_imbalance_g} aggregates over the matched control units, what matters is the set $S_\sigma$ rather than the precise matching membership in $\sigma$. Thus, we will focus on the stability of $S_\sigma$ instead of the optimal injection $\hat\sigma_n$ when replacing one data point, which, as we stated at the beginning of this part, can be used together with the Efron--Stein bound. 
For example, consider
two matching designs $\cM_1=\{(1,1),(2,2)\}$ and $\cM_2=\{(1,2),(2,3)\}$. Both treated units change partners, but the selected-control set changes only from $\{1,2\}$ to $\{2,3\}$.

\paragraph{Set differences.} For two matchings $\cM_1$ and $\cM_2$, we define their \emph{symmetric difference} as the graph obtained by preserving all edges that appears in exactly one graph only. 
For example, in the preceding example of $\cM_1=\{(1,1),(2,2)\}$ and $\cM_2=\{(1,2),(2,3)\}$, the symmetric difference is the chain
\$
 C_1\;\text{---}\;T_1\;\text{---}\;C_2\;\text{---}\;T_2\;\text{---}\;C_3.
\$ 
since no edge is removed. 
In such a graph, each vertex has degree at most $2$ (since each vertex has degree at most $1$ in each graph), and there is at most one edge connecting any pair $(i,j)$ (since repeated edges in the two graphs are removed). 
We define a \emph{path} as a sequence of edges successively joining  distinct vertices, such as the chain above. 
We define a \emph{cycle} as a sequence of edges successively joining distinct vertices and closing back at the starting vertex, without a repeated vertex. 
A \emph{connected component} is a maximal group of vertices joined by such paths, and different components do not share vertices. 

The next Lemma shows that if two cost matrices differ by only one row or one column, then the  selected control set of the optimal matching would differ by at most one member. Its proof is in Appendix~\ref{app:subsec_lem_stab}.

\begin{lemma}\label{lem:control_set_stability}
    Let $c=(c_{ij})$ and $c'=(c_{ij}')$ be two cost matrices that differ by at most one row or one column. Suppose $\sigma_\star$ and $\sigma_\star'$ are the unique optimal matchings minimizing the corresponding matching costs. Then the selected control set either agree, $S_{\sigma_\star}= S_{\sigma_\star'}$, or they differ by deleting one element and adding another,  $S_{\sigma_\star'} = S_{\sigma_\star} \cup \{j'\} \backslash \{j\}$ for some $j'\neq j$.  
\end{lemma}

\subsection{Local cells and replacement stability}
\label{app:subsec_local_cells}

This subsection continues with the split of $\Omega$ into cubes and develop definitions and intermediate results on the abundance of control samples in each cell, as well as the bounds of changes needed to repair a partial matching after one-coordinate replacement. 
Lemmas~\ref{lem:count_stable} and~\ref{lem:local_edge} in this subsection will be used in the proof of Lemma~\ref{lem:stab_opt_value}.

Recall that we have split $\Omega$ into $q_n^d$ equally-sized, non-overlapping cells. For a cell $Q$, we define the number of treated and control samples belonging to $Q$ as $N_{Q,1} = \sum_{i=1}^{n_1} \ind\{T_i\in Q\}$ and $N_{Q,0} = \sum_{j=1}^{n_0} \ind\{C_j \in Q\}$. Their mean values are denoted as $\mu_{Q,1} = n_1\int_Q p_1$ and $\mu_{Q,0} = n_0\int_Q p_0$, respectively. By Assumption~\ref{ass:cov}, we know $\mu_{Q,0}\geq (1+\delta) \mu_{Q,1}$, and both are bounded above and below by constant multiples of $L\log n$. 

Lemma~\ref{lem:count_stable} shows in every cell there are sufficiently more control samples than the treated, and the treated/control unit counts are stable under such replacement. The proof is in Appendix~\ref{app:subsec_proof_lem_count_stab}. 

\begin{lemma}\label{lem:count_stable}
    Fix $0<\zeta < \min\{\delta/8, 1/4\}$. For every $b>0$, $\cgrid$ can be chosen such that $\cE_n:=\{ N_{Q,0}\geq (1+2\zeta) N_{Q,1} +4 \text{ for every }Q\}$ satisfies $\PP(\cE_n^c)\leq n^{-b}$ for all sufficiently large $n$. Moreover, suppose one-coordinate replacement applies to a dataset $D$ in $\cE_n$ and yields the dataset $D'$, with the corresponding counts $N_{Q,1}'$ and $N_{Q,0}'$. Then the new counts satisfy $N_{Q,0}'\geq (1+\zeta) N_{Q,1}'$ for any cell $Q$. 
\end{lemma}

\paragraph{Local graph.} We construct a \emph{local graph} by connecting a treated unit to every control in the cell it belongs to. The (Euclidean) length of the edges in the local graph is at most $r_n$. 
We also call such edges \emph{local}. 
For a set $S$ of treated vertices, let $\cN(S)$ be the set of controls connected in the local graph to at least one vertex in $S$. 
We define a \emph{partial matching} as a set of treated-control edges with no repeated endpoints; its distinction from an injection (which matches all treated vertices) is that it may leave some treated vertices unmatched.

The next lemma provides a lower bound for $|\cN(S)|$ and shows that any partial matching with one unmatched treated vertex can always be repaired within a certain distance. 
The proof is in Appendix~\ref{app:subsec_proof_lem_local_edge}.

\begin{lemma}\label{lem:local_edge}
Suppose $N_{Q,0}\ge(1+\zeta)N_{Q,1}$ in every cell. Then $|\mathcal N(S)|\ge(1+\zeta)|S|$ for every set $S$ of treated vertices. Given any partial matching with $n_1-1$ edges (not necessarily local), one can match the remaining unmatched treated vertex by adding and removing edges along a path of length at most $2\lceil\log n_1/\log(1+\zeta)\rceil+3$ and every added edge is a local edge defined above.  
\end{lemma}

\subsection{Proof of Lemma~\ref{lem:control_set_stability}}
\label{app:subsec_lem_stab}

\begin{proof}[Proof of Lemma~\ref{lem:control_set_stability}]
Write $\mathcal M_1,\mathcal M_2$ as the unique optimal matchings, so $M_1 = \{(i,\sigma_\star(i))\colon 1\leq i\leq n_1\}$ and $M_2 = \{(i,\sigma_\star'(i))\colon 1\leq i\leq n_1\}$. Let $v_*$ be the single treated or control vertex whose row or column in the cost matrix changes between $(c,c')$. Therefore, every edge with $v_*$ as an endpoint has the same cost in $c$ and $c'$.

Consider the symmetric difference between the graphs defined by $\cM_1$ and $\cM_2$, i.e., the common edges are deleted after combining the edges in the two graphs. We denote this graph as $\cG_0$. 
In this graph, a treated vertex $i$ has either degree of $0$ (if $\sigma_\star(i)=\sigma_\star'(i)$), or two edges (if $\sigma_\star(i)\neq \sigma_\star'(i)$ hence the edges from $\cM_1$ and $\cM_2$ are both preserved. 
A control vertex $j$ has a degree of either $0$ (if it is matched to the same treated vertex in $\cM_1$ and $\cM_2$), $1$ (if it is selected in one graph but not the other), or $2$ (if it is selected in both and matched to distinct treated vertices). 
Whenever the degree is $2$, the two edges must come from each of the matchings and connect to distinct vertices in the other group.  

We claim that in $\cG_0$, each connected component is a path or a cycle. To see this, we consider two cases. 
\begin{itemize}
    \item First, if the component has any degree-$1$ vertex, we start a walk from any of such vertices following the edges. At any later, degree-two vertex, there is exactly one edge on which to continue without going straight back; for such vertex, this walk continues following that edge. This cannot branch since it would create a vertex of degree $3$. 
    Moreover, this walk cannot enter a cycle because otherwise this entry point would have a degree of three. Since there are finitely many vertices, this walk must end at another degree-one vertex, which traverses a sequence of distinct degree-two vertices and this component is a path. 

    \item On the other hand, if the component has no degree-one vertex, every vertex has degree two with distinct endpoints. Starting from any vertex must come back to it and create a cycle. No additional edge can attach to that cycle without creating a vertex of degree $3$. 
\end{itemize}


Next, we claim that, traversing  each of these paths and cycles, the edges must alternate between $\cM_1$ and $\cM_2$. 
To see this, because $\cG_0$ is the symmetric difference, every edge must appear only in exactly one of the two matchings. As we traverse in any path or cycle, the edges attached to a vertex of degree $2$ must come from different matchings; otherwise, if the two edges come from the same matching, it either connects to the same control vertex twice, which is impossible, or connects to distinct control vertices in the same matching, which is also impossible. 
In words, when traversing each path or cycle, the vertices alternate between treated and control groups, and the edges alternate between $\cM_1$ and $\cM_2$. 

As we argued earlier, a treated vertex cannot have degree $1$ in $\cG_0$, so every path must have two control endpoints. The number of edges in every path must be even, since edges run between the treated/control groups and both endpoints are controls. Thus, the first and last edges must belong to different matchings; one endpoint is in $S_{\sigma_\star}$ (selected in $\cM_1$) and one endpoint is in $S_{\sigma_\star'}$ (selected in $\cM_2$). 
On the other hand, on a cycle, every control vertex is in both $S_{\sigma_\star}$ and $S_{\sigma_\star'}$.

Recall that $v_\star$ is the single treated or control vertex whose row or column in the cost matrix changes. 

\begin{claim}\label{clm:contain_v*}
    Every nonempty component in the symmetric difference graph $\cG_0$ must contain $v_*$. 
\end{claim}

\begin{proof}[Proof of Claim~\ref{clm:contain_v*}]
    Suppose a non-empty component $H$ in  $\cG_0$ does not contain $v_*$. 
We define a \texttt{switching} operation, where \texttt{switching} this component means removing its edges in $\cM_1$ (roughly denoted by $H\cap \cM_1$) and inserting its edges in $\cM_2$ (denoted by $H\cap \cM_2)$, leaving all other edges unchanged. 

Starting from any injection (represented by a bipartite graph), \texttt{switching} the component $H$ yields another injection, 
since each treated vertex or each internal control (i.e., control vertices with degree $2$) still has degree $1$ (one edge in $\cM_1$ removed and one edge in $\cM_2$ added); moreover, the endpoint controls if $H$ is a path either changes from being selected to not or vice versa. No vertex outside $H$ is affected. Therefore, \texttt{switching} from $\cM_1$ or $\cM_2$ gives a feasible injection.

Define the total cost change for switching from $\cM_1$ (which leads to $\cM_2$) by
\begin{equation*}
 s_H=\sum_{(i,j)\in\mathcal M_2\cap H}c_{ij}
       -\sum_{(i,j)\in\mathcal M_1\cap H}c_{ij}.
\end{equation*}
That is, evaluated under $(c_{ij})$, the cost of $\cM_2$ increases from $\cM_1$ by $s_H\geq 0$ since $\cM_1$ is optimal under $(c_{ij})$.

Now consider switching from $\cM_2$, which leads to $\cM_1$. 
Since $v^*$ is not in $H$, every edge in $H$ has the same cost in $(c_{ij})$ and $(c_{ij}')$, so this switch evaluated under $c'$, changes its cost (from $\cM_2$ to $\cM_1$) by $-s_H$. Optimality of $\cM_2$ under $c'$ gives $-s_H\ge0$. Hence $s_H=0$. But $H$ is non-empty and contains only non-common edges, so the switch from $\cM_1$ to $\cM_2$ produces a different matching with the same minimum cost under $c$. This contradicts the uniqueness.
Thus, every nonempty component must contain $v_*$. 
\end{proof}

By Claim~\ref{clm:contain_v*}, since different components in $\cG_0$ have no common vertex, there is at most one non-empty component in $\cG_0$. 
If none exists, the graph $\cG_0$ is empty, and all edges in $\cM_1$ and $\cM_2$ agree, so $S_{\sigma^\star} = S_{\sigma_\star'}$. 
If there is one component, and if it is a cycle, then because each internal control vertex in the cycle is selected in both $\cM_1$ and $\cM_2$, we know $S_{\sigma^\star} = S_{\sigma_\star'}$.  
If it is a path, then every internal control vertex is selected in both $\cM_1$ and $\cM_2$. 
Only the two control endpoints are selected in only one of the two matchings; moreover, by our earlier arguments, they belong to different matchings, thus $S_{\sigma_\star'}$ is obtained by removing one endpoint control vertex $j$ from $S_{\sigma^\star}$ and add another $j'$. This proves Lemma~\ref{lem:control_set_stability}.
\end{proof}

\subsection{Proof of Lemma~\ref{lem:count_stable}}
\label{app:subsec_proof_lem_count_stab}

\begin{proof}[Proof of Lemma~\ref{lem:count_stable}]
Let $\beta=1+2\zeta$. Since $\zeta<\delta/8$, we have
$\beta<1+\delta$. Hence we may choose a fixed
$\epsilon\in(0,1)$ sufficiently small so that 
$\kappa:=(1-\epsilon)(1+\delta) -\beta(1+\epsilon)>0$. 

For a fixed cell $Q$, consider the event 
$\mathcal G_Q
=
\left\{
N_{Q,0}\ge(1-\epsilon)\mu_{Q,0},~
N_{Q,1}\le(1+\epsilon)\mu_{Q,1}
\right\}$. 
On $\mathcal G_Q$, Assumption~\ref{ass:cov} implies
\$
N_{Q,0}-\beta N_{Q,1} \ge
(1-\epsilon)\mu_{Q,0}
-\beta(1+\epsilon)\mu_{Q,1} \ge
\left\{(1-\epsilon)(1+\delta)
-\beta(1+\epsilon)\right\}\mu_{Q,1} =\kappa\mu_{Q,1}.
\$
By the construction of the partition and the fixed lower bounds
on $n_1/n$ and $p_1$, there exists a constant $c_0>0$ such that, 
$\inf_{Q}\mu_{Q,1}\ge c_0 c_{\mathrm{grid}}\log n$ 
for all sufficiently large $n$. Therefore, for sufficiently large
$n$, 
$\kappa\mu_{Q,1}\ge4$, 
and hence $\mathcal G_Q$ implies 
$N_{Q,0}\ge \beta N_{Q,1}+4$. 

It remains to bound the probability that $\mathcal G_Q$ fails.
Since $N_{Q,0}$ and $N_{Q,1}$ are binomial, Lemma~\ref{lem:chernoff} gives
\[
\begin{aligned}
\mathbb P(\mathcal G_Q^c)
&\le
\mathbb P\!\left\{
N_{Q,0}<(1-\epsilon)\mu_{Q,0}
\right\}
+
\mathbb P\!\left\{
N_{Q,1}>(1+\epsilon)\mu_{Q,1}
\right\}\\
&\le
e^{- \epsilon^2 \mu_{Q,0}/2} 
+
e^{- \epsilon^2\mu_{Q,1}/(2+\epsilon)} \le
2\exp\{-c_1c_{\mathrm{grid}}\log n\}
=
2n^{-c_1c_{\mathrm{grid}}},
\end{aligned}
\]
for a constant $c_1>0$ depending only on the fixed regularity
constants and on $\delta,\zeta$.
There are $q_n^d\le Cn/(c_{\mathrm{grid}}\log n)$ cells for all
sufficiently large $n$. A union bound therefore gives
\[
\mathbb P(\mathcal E_n^c)
\le
\frac{Cn}{c_{\mathrm{grid}}\log n}
\,n^{-c_1c_{\mathrm{grid}}}.
\]
For any fixed $b>0$, choosing $c_{\mathrm{grid}}$ sufficiently
large makes the right-hand side at most $n^{-b}$ for all sufficiently
large $n$.

We then verify the replacement claim. Suppose $D\in\mathcal E_n$.
If a treated observation is replaced, then for every cell $Q$, 
$N'_{Q,0}=N_{Q,0}$ and
$N'_{Q,1}\le N_{Q,1}+1$. 
Thus
\$
N'_{Q,0}-(1+\zeta)N'_{Q,1}
&\ge
(1+2\zeta)N_{Q,1}+4
-(1+\zeta)(N_{Q,1}+1) =
\zeta N_{Q,1}+3-\zeta>0.
\$
If instead a control observation is replaced, then 
$N'_{Q,1}=N_{Q,1}$ and 
$N'_{Q,0}\ge N_{Q,0}-1$, 
and hence
\$
N'_{Q,0}-(1+\zeta)N'_{Q,1}
&\ge
(1+2\zeta)N_{Q,1}+3
-(1+\zeta)N_{Q,1} = 
\zeta N_{Q,1}+3>0.
\$
Therefore 
$N'_{Q,0}\ge(1+\zeta)N'_{Q,1}$ 
for every cell $Q$.
\end{proof}

\subsection{Proof of Lemma~\ref{lem:local_edge}}
\label{app:subsec_proof_lem_local_edge}

\begin{proof}[Proof of Lemma~\ref{lem:local_edge}]
Let $\cQ(S)$ be the cells containing at least one vertex in $S$. Every control in these cells belongs to $\mathcal N(S)$, so
\[
 |\mathcal N(S)|=\sum_{Q\in\mathcal Q(S)}N_{Q,0}
 \ge(1+\zeta)\sum_{Q\in\mathcal Q(S)}N_{Q,1}
 \ge(1+\zeta)|S|.
\]

Now consider the stated partial mapping. 
We perform a search which starts at the unique unmatched treated vertex, forming a singleton treated vertex set $S_0$. 
Then, for each $k\geq 0$, we follow all \textit{local} edges from treated sets in $S_k$ to controls, and then follow all existing edges in the matching (not necessarily local) from controls to their treated partners. The new set $S_{k+1}$ is formed by adding the newly visited treated vertices to $S_k$. 
In this search, we ignore local edges already in the matching, because searching along them do not increase $S_k$.  
We claim that after a bounded number of rounds of search, $\cN(S_k)$ must reach an unmatched control. 
To see this, if $\mathcal N(S_k)$ contains no unmatched control, then due to partial matching, every vertex in $\mathcal N(S_k)$ has a distinct treated partner, and all these partners are included in $S_{k+1}$. Thus $|S_{k+1}|\ge|\mathcal N(S_k)|\ge(1+\zeta)|S_k|$. This cannot persist for more than $\lceil\log n_1/\log(1+\zeta)\rceil+1$ rounds, because there are only $n_1$ treated vertices.

Whenever a new vertex $v$ is discovered for the first time, we record the \textit{first-discovery predecessor} $u=\text{pred}(v)$ as the other endpoint of the corresponding edge. For example,  $v$ may be a control vertex first reached because it is connected to some $u\in S_k$ by a local edge; or $v$ may be a treated vertex first reached because it is not in $S_k$ but is matched to some $u\in \cN(S_k)$ hence first discovered in round $k+1$. 
Since each control vertex may be connected by multiple treated vertices with local edges, we fix a total ordering of the vertices and take $\text{pred}(v)$ as the first vertex in this ordering satisfying the \textit{first-discovery predecessor} criterion. 
In this way, every vertex $v$ has a unique predecessor. 

For the reached unmatched control vertex, we trace back the search recursively via first-discovery predecessors to the initial vertex. 
This path thus alternates between local edges not in the matching (since we ignore local edges in the matching) and edges in the matching, starting and ending with a local edge not in the partial matching. 
Now, we start from the partial matching, remove the matching edges on this path and add the other (local) edges. 
Thus, each internal vertex loses one edge and gains one, while both endpoints gain one. 
The resulting graph covers all treated vertices and uses every control at most once. Moreover, the added edges are local by construction. 
Since the search persists no more than $\lceil\log n_1/\log(1+\zeta)\rceil+1$ rounds and each round contributes at most $2$ edges to this path, the length of this path is at most $2\lceil\log n_1/\log(1+\zeta)\rceil+3$.
\end{proof}

\subsection{Proof of Lemma~\ref{lem:unique}}
\label{app:subsec_lem_unique_tilt}

\begin{proof}[Proof of Lemma~\ref{lem:unique}]
Fix distinct injections $\sigma,\tau$ and an index $i$ with $\sigma(i)\ne\tau(i)$. Conditional on all controls and on all treated observations except $T_i$, the difference of their objectives is affine in $T_i$. The quadratic terms in $T_i$ and the two terms $tg(T_i)$ cancel. The linear coefficient is $n_1^{-1}A\{C_{\tau(i)}-C_{\sigma(i)}\}$, which is nonzero almost surely because the controls are independent with a density and $A$ is invertible. Equality requires $T_i$ to lie on a hyperplane, which has conditional probability zero. There are finitely many pairs of injections, so by the union bound, we know there are no ties almost surely.  
\end{proof}

\subsection{Proof of Lemma~\ref{lem:stab_opt_value}}
\label{app:subsec_proof_lem_opt_stab}

\begin{proof}[Proof of Lemma~\ref{lem:stab_opt_value}]
This proof uses Lemma~\ref{lem:count_stable} and Lemma~\ref{lem:local_edge}. 
We suggest reading Appendix~\ref{app:subsec_match_graph_setup} and Appendix~\ref{app:subsec_local_cells} before reading this proof.

Start from an optimizer (an injection, one-to-one matching which can be represented by a graph with $n_1$ edges as before) for $D$. We consider three cases for the one-coordinate replacement:
\begin{enumerate}[label=(\arabic*).,leftmargin=2.5em]\setlength\itemsep{-0.25em}
    \item If a treated sample is replaced, remove its edge and retain the other $n_1-1$ edges with endpoints in $D'$. 
    \item If a selected control is replaced, remove its edge and retain the other $n_1-1$ edges with endpoints in $D'$. 
    \item If an unselected control is replaced, the original injection remains feasible with identical objective in $D'$.
\end{enumerate}

In the first two cases, the new counts bounds in $D'$ in Lemma~\ref{lem:count_stable} applies, i.e., $N_{Q,0}'\geq (1+\zeta)N_{Q,1}'$ for any cell $Q$. 
Additionally applying Lemma~\ref{lem:local_edge} to $D'$, we know the partial matching obtained after the stated operations can be turned to a complete matching using $O(\log n)$ added and removed edges.
Every added edge costs at most $C(r_n^2+t^2)$ by~\eqref{eq:A2:edge-upper}. Every removed edge, including the initially deleted one evaluated in $D$, costs at least $-Ct^2$ by~\eqref{eq:A2:edge-lower}; deleting it can therefore increase the cost by at most $Ct^2$. All retained edges have the same endpoints and costs in the two datasets. The total increase (before normalization of $1/n_1$) is at most $C(r_n^2+t^2)\log n=CH_n(t)$. Division by $n_1\ge a_{\mathrm s}n$ and minimization in $D'$ give $V_{D'}(t)\le V_D(t)+CH_n(t)/n$.

For the other direction, we start from an optimizer in $D'$ and remove the affected edge if it is present. Performing the stated three cases yields a partial matching in $D$. Note that $D$  satisfies $N_{Q,0}\geq (1+\zeta) N_{Q,1}$ because $D\in\cE_n$, hence Lemma~\ref{lem:local_edge} applies to $D$. The same edge bounds and normalization give $V_D(t)\le V_{D'}(t)+CH_n(t)/n$. 
Combining the two directions yields the desired result.
\end{proof}

\subsection{Proof of Lemma~\ref{lem:value-bounds}}
\label{app:subsec_lem_value_bounds}
\begin{proof}[Proof of Lemma~\ref{lem:value-bounds}]
    Fix $|t|\le1$. Since $\|x-y\|^2\le d$, every edge satisfies $|\ell_t(x,y)|\le b_A d/2+2\|g\|_\infty=:M_0$. 
    Thus $|V_{\bar D}(t)|\le M_0$ for every dataset $\bar D\in\Omega^n$. In particular, $|V_D(t)-V_{\bar{D}}(t)|\leq 2M_0$ for any $D,\bar D\in \Omega^n$. 
    Choose a fixed constant $b>2+4/d$ and use Lemma~\ref{lem:count_stable} to 
    choose a fixed $\cgrid$ so that $\PP(\cE_n^c)\le n^{-b}$.

For notational simplicity, denote $(W_1,\dots,W_n)=(T_1,\dots,T_{n_1},C_1,\dots,C_{n_0})$, i.e., let  $W_k=T_k$ for $k\le n_1$ and $W_k=C_{k-n_1}$ for $n_1<k\le n$. 
For dataset $D=(W_k)_{k=1}^n$, for each $k$, let $D^{(k)}$ be the data obtained by replacing $W_k$ in $D$ by an independent copy from the corresponding treated or control covariate distribution. On the event $D\in\cE_n$, Lemma~\ref{lem:stab_opt_value} gives $|V_D(t)-V_{D^{(k)}}(t)|\leq CH_n(t)/n$; on the event $D\notin \cE_n$, the arguments early on gives $|V_D(t)-V_{D^{(k)}}(t)|\leq 2M_0$. Putting the two cases together we obtain 
\[
 \EE\{V_D(t)-V_{D^{(k)}}(t)\}^2
 \le \frac{CH_n(t)^2}{n^2}+4M_0^2n^{-b}.
\]
Since $W_k$'s are independent and $V_D(t)$ is square-integrable, Lemma~\ref{lem:ES}, summed over $k\in[n]$, yields
\[
 \Var V_D(t)\le\frac12\sum_{k=1}^n
       \EE\{V_D(t)-V_{D^{(k)}}(t)\}^2
 \le\frac{CH_n(t)^2}{n}+2M_0^2n^{1-b},\quad \forall t\in [-1,1].
\] 
Indeed, $H_n(t)\ge\Gamma_n$ and $\Gamma_n^2/n\asymp n^{-1-4/d}(\log n)^{2+4/d}$, so $n^{1-b}/(\Gamma_n^2/n)=O\{n^{2+4/d-b}/(\log n)^{2+4/d}\}\to0$ by the choice of $b$. This proves the first bound in~\eqref{eq:value-bounds}.

We now proceed to prove the second bound in~\eqref{eq:value-bounds}. 
Let $\bar T_1\sim p_1$ be independent of the remaining observations. Define the coupled datasets
\[
 \bar D=(\bar T_1,T_2,\ldots,T_{n_1},C_1,\ldots,C_{n_0}),\qquad
 \bar D_x=(x,T_2,\ldots,T_{n_1},C_1,\ldots,C_{n_0}).
\]
Then $D$ and $\bar D$ has the same law, and 
\[
 \mu_t(x)-\EE V_D(t)=\EE\{V_{\bar D_x}(t)-V_{\bar D}(t)\}.
\]
On the event $\bar D\in \cE_n$, Lemma~\ref{lem:stab_opt_value} gives $|V_{\bar{D}_x}(t)-V_{\bar D}(t)|\leq CH_n(t)/n$ for every $x\in\Omega$. On its complement the bound is $2M_0$ for every $x\in \Omega$. Consequently,
\[
 \sup_{x\in\Omega}|\mu_t(x)-\EE V_D(t)|
 \leq \EE\Big\{ \frac{CH_n(t)}{n} \ind\{\bar D\in \cE_n\} +2M_0\ind\{\bar D\notin \cE_n\} \Big\} 
 \leq \frac{CH_n(t)}{n}  + 2M_0 n^{-b}.
\]
The remainder is negligible because $H_n(t)\geq \Gamma_n$ and  
 $\frac{n^{-b}}{\Gamma_n/n}
 =O (\frac{n^{1+2/d-b}}{(\log n)^{1+2/d}} )\to 0$. 
Since $H_n(t)\ge\Gamma_n$, this proves the second bound in~\eqref{eq:value-bounds}. 

Finally, we prove the last claim. 
Write the measurable optimizer as $\sigma_t = s_t(D)$. 
On $\cE_n$, by Lemma~\ref{lem:count_stable}, every cell $Q$ contains enough controls to match all the treated observations within that cell. Combining the cellwise injections gives a feasible $\bar\sigma$, since distinct cells have disjoint vertices. Each of its edges has length at most $r_n$, so~\eqref{eq:A2:edge-upper} and optimality give $V_D(t)\le C_1(r_n^2+t^2)$. On the other hand, averaging both sides of~\eqref{eq:A2:edge-lower} over any optimizer gives $V_D(t)\ge(a_A/4)L_D(\sigma_t)-C_2t^2$. Therefore $L_D(\sigma_t)\le C(r_n^2+t^2)$ on $D\in\cE_n$.  
On the other hand, on the event $D\in \cE_n^c$, the diameter of the cube gives $L_D(\sigma_t)\le d$. Combining the two cases, we obtain  $\EE [L_D(\sigma_t)]\le C(r_n^2+t^2)+dn^{-b}$. Finally, noting that $n^{-b}/r_n^2=O\{n^{2/d-b}/(\log n)^{2/d}\}\to0$ yields the desired bound.
\end{proof}

\subsection{Proof of Lemma~\ref{lem:slope}}
\label{app:subsec_lem_slope}

\begin{proof}[Proof of Lemma~\ref{lem:slope}]
Fix a deterministic $t\in[-1,1]$. By Lemma~\ref{lem:unique},
the tilted objective has a unique minimizer at $t$ with
probability one. Fix a dataset in this event, and denote
its minimizer by $\sigma_t$.

For each injection $\sigma\in\mathcal A_{n_1,n_0}$, write
\[
b_\sigma=B_{D,\sigma,A}(g),
\qquad
\varphi_\sigma(s)=Q_D(\sigma,A)+s b_\sigma.
\]
Here $D,A,g$ are held fixed, so $b_\sigma$ is a constant
and $\varphi_\sigma$ is affine in $s$. By definition,
$V_D(s)=\min_\sigma\varphi_\sigma(s)$. Recall $\sigma_t$ is the minimizer defined in~\eqref{eq:def_Vd_minimizer}.

If there is only one feasible injection, $V_D$ is affine
and the derivative formula follows immediately.
Otherwise, uniqueness and finiteness imply
\[
\gamma
:=
\min_{\sigma\ne\sigma_t}
\{\varphi_\sigma(t)-\varphi_{\sigma_t}(t)\}>0.
\]
Set 
$M=\max_{\sigma\ne\sigma_t}|b_\sigma-b_{\sigma_t}|$ and  
$r=\frac{\gamma}{2(1+M)}>0$. 
For every $s$ with $|s-t|<r$ and every
$\sigma\ne\sigma_t$,
\[
\begin{aligned}
\varphi_\sigma(s)-\varphi_{\sigma_t}(s)
&=
\varphi_\sigma(t)-\varphi_{\sigma_t}(t)
+(s-t)(b_\sigma-b_{\sigma_t}) \ge \gamma-|s-t|M
\ge \gamma-rM
\ge \gamma/2>0.
\end{aligned}
\]
Thus $\sigma_t$ remains the unique minimizer whenever
$|s-t|<r$, and in which case 
$V_D(s)-V_D(t)=(s-t)b_{\sigma_t}$. 
Taking the difference quotient shows that almost surely 
\@\label{eq:def_VD't}
V_D'(t)=b_{\sigma_t}=B_{D,\sigma_t,A}(g).
\@ 
Finally, 
$|B_{D,\sigma_t,A}(g)|
\le
\frac1{n_1}\sum_{i=1}^{n_1}
\bigl\{|g(T_i)|+|g(C_{\sigma_t(i)})|\bigr\}
\le 2\|g\|_\infty$. 
Therefore, the derivative $V_D'(t)$ can be well-defined to be a bounded measurable function because one can set it to be zero on the null event where
uniqueness fails.  
By Lemma~\ref{lem:Taylor}, each matched pair satisfies 
$g(T_i)-g(C_{\sigma_t(i)})=\nabla g(T_i)^\top\{T_i-C_{\sigma_t(i)}\}+R_i(t)$, 
where $|R_i(t)|\le(G_2/2)\|T_i-C_{\sigma_t(i)}\|^2$. Averaging over $i$ therefore gives $\big|V_D'(t)-G_D(t) \big|\le(G_2/2)L_D(\sigma_t)$, where we define 
\[
 G_D(t)=\frac1{n_1}\sum_{i=1}^{n_1}
    \nabla g(T_i)^\top\{T_i-C_{\sigma_t(i)}\}, \quad \text{and}\quad   J_g=\EE\{\nabla g(T_1)^\top A^{-1}\nabla g(T_1)\},
\]
and $J_g\geq 0$ is fixed constant to be used later. 
Therefore, by triangle inequality, 
\@\label{eq:bd_VD't_1}
|\EE V_D'(t)|\leq |\EE G_D(t) | + (G_2/2) \EE[L_D(\sigma_t)] \leq |\EE G_D(t) | + C(r_n^2+t^2).
\@ 
almost surely, where the second inequality uses the last claim in Lemma~\ref{lem:value-bounds}, and
Here $G_D(t)$ depends on $t$ through $\sigma_t$, whereas $J_g$ is a fixed nonnegative number.  
We shall further control the term $G_D(t)$ in~\eqref{eq:bd_VD't_1} together with $J_g$. 
The proof of Lemma~\ref{lem:IBP} is in Appendix~\ref{app:subsec_lem_IBP}. 
As preparation, define the functions
\@\label{eq:def_v_a}
v(x)=A^{-1}\nabla g(x), \quad a(x)=p_1(x)v(x).
\@
For the vector field $a=(a_1,\ldots,a_d)^\top$, its divergence is the scalar $\nabla\!\cdot a=\sum_{k=1}^d\partial_k a_k$. Both $a$ and its first derivatives are continuous and bounded. Let $\nu(x)$ denote the outward unit normal to
$\Omega=[0,1]^d$ at $x\in\partial\Omega$, and let $dS$
denote surface measure. On the relative interior of the face $x_k=1$, $\nu(x)=e_k$; on the face $x_k=0$,
$\nu(x)=-e_k$, where $e_k$ is the $k$th standard basis vector.
Intersections of faces have zero surface measure.

\begin{lemma}\label{lem:IBP}
Recall the functions defined in~\eqref{eq:def_v_a}. 
For every deterministic $|t|\le1$,
\@\label{eq:IBP} 
 \EE [G_D(t)]+tJ_g
 =n_1\bigg[&\int_{\partial\Omega}\{\mu_t(x)-\EE V_D(t)\}
 a(x)^\top\nu(x)\ud S(x) -\int_\Omega\{\mu_t(x)-\EE V_D(t)\}
 \{\nabla\!\cdot a(x)\}\ud x\bigg].
\@
\end{lemma}

Now, write $G_2=\sup_\Omega\|\nabla^2g\|_{\mathrm{op}}$ and define the finite constant
\[
 K_a=\int_{\partial\Omega}|a^\top\nu|\ud S
               +\int_\Omega|\nabla\!\cdot a|\ud x.
\]
Indeed, $\sup_\Omega\|a\|\le\|p_1\|_\infty L_g/a_A$ and
$\nabla\!\cdot a=\nabla p_1^\top A^{-1}\nabla g+p_1\operatorname{tr}(A^{-1}\nabla^2g)$. Thus
$\|\nabla\!\cdot a\|_\infty\le(\sup_\Omega\|\nabla p_1\|L_g+d\|p_1\|_\infty G_2)/a_A$.
The cube has volume one and $2d$ faces of area one, so $K_a$ is bounded independently of $n,t$. Also by the eigenvalue bounds we have $0\le J_g\le L_g^2/a_A$.

Taking absolute values in~\eqref{eq:IBP} and applying Lemma~\ref{lem:value-bounds} gives
\[
 \big|\EE [G_D(t)] +tJ_g \big|
 \le n_1K_a\sup_{x\in\Omega}|\mu_t(x)-\EE V_D(t)|
 \le CK_a\frac{n_1}{n}H_n(t)
 \le C'H_n(t).
\]
Hence we have 
\$ \big|\EE [G_D(t)] \big|\leq C'H_n(t) + |t| J_g\leq  C'H_n(t)+|t|L_g^2/a_A .
\$
Combining this with~\eqref{eq:bd_VD't_1} yields $|\EE V_D'(t)|\le C\{H_n(t)+|t|+r_n^2+t^2\}$. For large $n$, since $\log n\ge1$ we know $r_n^2+t^2\le H_n(t)$, which proves the claim.
\end{proof}

\subsection{Proof of Lemma~\ref{lem:IBP}}
\label{app:subsec_lem_IBP}

\begin{proof}[Proof of Lemma~\ref{lem:IBP}]
For a dataset with a unique optimizer, continuity of the finitely many objectives keeps this optimizer unchanged in a sufficiently small neighborhood of the data. This allows us to obtain the differentiation of $V_D(t)$ along an entry $T_i$ in $D$. We prove the following after this proof.  

\begin{lemma}\label{lem:diff_VDt}
    It holds almost surely that 
    \@\label{eq:diff_VDt}
    \nabla_{T_i}V_D(t) = \frac1{n_1}
\left[A\{T_i-C_{\sigma_t(i)}\}+t\nabla g(T_i)\right].
    \@
\end{lemma}

Taking the inner product of~\eqref{eq:diff_VDt} with $v(T_i)=A^{-1}\nabla g(T_i)$, we have
\[
 \nabla_{T_i}V_D(t)^\top v(T_i)
 =\frac1{n_1}\nabla g(T_i)^\top\{T_i-C_{\sigma_t(i)}\}
    +\frac{t}{n_1}\nabla g(T_i)^\top A^{-1}\nabla g(T_i)
\]
almost surely. By Lemma~\ref{lem:diff_VDt}, equation~\eqref{eq:diff_VDt} holds almost surely under the joint distribution. 
By Fubini's theorem, equation~\eqref{eq:diff_VDt} holds for almost every $x$ for almost every fixed $D_{-i}$; we assign any bounded values for derivatives on null sets. Thus, summarizing over $i\in[n_1]$ and taking expectation yields 
\$
\EE[G_D(t)+tJ_g] =  \sum_{i=1}^{n_1} \EE\big[ \nabla_{T_i}V_D(t)^\top v(T_i) \big].
\$
We shall prove~\eqref{eq:IBP} by investigating each summation term. Throughout we fix $t\in [-1,1]$.

Fix $i\in[n_1]$ and denote all observations other than $T_i$ as $D_{-i}$. We condition on $D_{-i}$. Let $F_i(x)=V_{D^{i\to x}}(t)$. 
Then by definition, $\nabla_{T_i}V_D(t) = \nabla F_i(x)$. 
Conditional on $D_{-i}$, the density of $T_i$ is still $p_1$, so for a.s.~$D_{-i}$,
\[
 \EE[\nabla_{T_i}V_D(t)^\top v(T_i)\mid D_{-i}]
 =\int_\Omega\nabla F_i(x)^\top v(x)p_1(x)\ud x
 =\int_\Omega\nabla F_i(x)^\top a(x)\ud x.
\]
Marginalizing over $D_{-i}$ yields 
\$
\EE[\nabla_{T_i}V_D(t)^\top v(T_i)] = \EE\bigg[ \int_\Omega\nabla F_i(x)^\top a(x)\ud x\bigg].
\$

For fixed $D_{-i}$, every objective $Q_{D^{i\to x}}(\sigma,A)+t B_{D^{i\to x}}(\sigma,A)$  has the common $x$-dependent term $\{x^\top Ax/2+tg(x)\}/n_1$, plus an affine function of $x$. Thus $F_i(x) = \min_{\sigma}\{Q_{D^{i\to x}}(\sigma,A)+t B_{D^{i\to x}}(\sigma,A)\}$ is continuous and piecewise $C^1$ on every coordinate line, with finitely many breakpoints. Its Lipschitz constant is at most $(b_A\sqrt d+L_g)/n_1$ for $|t|\le1$ as the same bound holds for $Q_{D^{i\to x}}(\sigma,A)+t B_{D^{i\to x}}(\sigma,A)$. 
Moreover, by the arguments in the proof of Lemma~\ref{lem:value-bounds}, we know $|F_i(x)|\leq M_0:= b_Ad/2 + 2\|g\|_\infty$.  

Fix a spatial coordinate $k$ and write $x^{(k,u)}$ for a point whose $k$-th coordinate is $u$, with the other coordinates $x_{-k}$ held fixed. Apply the product rule to $F_i a_k$ on each of the finitely many smooth subinterval of $[0,1]$ and take the summation. Continuity cancels all internal endpoints, leaving
\[
\begin{split}
 &\int_0^1\partial_kF_i(x^{(k,u)})a_k(x^{(k,u)})\ud u =F_i(x^{(k,1)})a_k(x^{(k,1)})
   -F_i(x^{(k,0)})a_k(x^{(k,0)})
   -\int_0^1F_i(x^{(k,u)})\partial_k a_k(x^{(k,u)})\ud u.
\end{split}
\]
Now integrate over $x_{-k}\in[0,1]^{d-1}$ and sum over $k$. The endpoint terms are the face integrals with outward normals, so
\[
 \int_\Omega\nabla F_i^\top a\ud x
 =\int_{\partial\Omega}F_i a^\top\nu\ud S
     -\int_\Omega F_i\nabla\!\cdot a\ud x.
\] 

For every fixed $x$, we know $\EE[F_i(x)]=\mu_t(x)$, since the remaining treated observations have the same joint distribution for every $i$, and relabeling treated indices does not change the minimum objective. Summing the preceding identities over $i$ therefore gives
\[
\EE[G_D(t)+tJ_g] = \sum_{i=1}^{n_1} \EE\big[ \nabla_{T_i}V_D(t)^\top v(T_i) \big] 
 =n_1\left\{\int_{\partial\Omega}\mu_t a^\top\nu\ud S
                  -\int_\Omega\mu_t\nabla\!\cdot a\ud x\right\}.
\] 

Apply the same coordinate calculation to the constant function $F_i\equiv1$ gives $\int_{\partial\Omega}a^\top\nu\ud S=\int_\Omega\nabla\!\cdot a\ud x$. 
Multiplying both sides by $\EE[V_D(t)]$ and subtract from the above equation gives the desired result.
\end{proof}

\begin{proof}[Proof of Lemma~\ref{lem:diff_VDt}]
Fix a deterministic $t\in[-1,1]$. By Lemma~\ref{lem:unique}
and absolute continuity of the covariate distributions,
the tilted optimizer is unique and all treated points lie in $\Omega^\circ$ (the interior of $\Omega$) with probability one. 
In the following we work on this event. 
Fix $i$, hold all other observations fixed, and define
\[
\varphi_\sigma(x)
=
Q_{D^{i\to x}}(\sigma,A)
+tB_{D^{i\to x},\sigma,A}(g),
\qquad \sigma\in\mathcal A_{n_1,n_0}.
\]
It's straightforward to see that $\varphi_\sigma$ is continuously differentiable in $x$,
and $V_{D^{i\to x}}(t)=\min_\sigma\varphi_\sigma(x)$ takes the minimum of finitely many such curves.
Let $\sigma_t$ denote the unique minimizer at $x=T_i$.

If $|\cA_{n_1,n_0}|=1$, then $\phi_\sigma(x)$ is clearly differentiable in $x$. Otherwise, finiteness and uniqueness give
\[
\gamma
:=
\min_{\sigma\ne\sigma_t}
\{\varphi_\sigma(T_i)-\varphi_{\sigma_t}(T_i)\}>0.
\]
By continuity of $\phi_\sigma(x)$ in $x$ and finiteness of $\sigma\in \cA_{n_1,n_0}$, one can choose a constant $r_0>0$ with
$B_{r_0}(T_i)=\{x\colon \|x-T_i\|\leq r_0\}\subset\Omega^\circ$ such that 
$\max_{\sigma\in\mathcal A_{n_1,n_0}}
|\varphi_\sigma(x)-\varphi_\sigma(T_i)|
<\gamma/3$ 
whenever  $\|x-T_i\|\leq r_0$. 
For every such $x$ and every $\sigma\ne\sigma_t$, we know  
$\varphi_\sigma(x)-\varphi_{\sigma_t}(x)
>
\varphi_\sigma(T_i)-\varphi_{\sigma_t}(T_i)
- 2\gamma/3 \geq \gamma/3>0$.  
Consequently,
$V_{D^{i\to x}}(t)=\varphi_{\sigma_t}(x)$ whenever $x\in B_{r_0}(T_i)$.
Note that only the $i$th matched pair's cost depends on $x$, so differentiation
at $x=T_i$ yields
\begin{equation}\label{eq:A2:spatial-gradient}
\nabla_{T_i}V_D(t) := \frac{\ud V_{D^{i\to x}}(t)}{\ud x}\bigg|_{x=T_i}
=
\frac{\ud \varphi_{\sigma_t}(x)}{\ud x}\bigg|_{x=T_i} 
=
\frac1{n_1}
\left[A\{T_i-C_{\sigma_t(i)}\}+t\nabla g(T_i)\right]
\quad\text{almost surely}.
\end{equation}
This proves Lemma~\ref{lem:diff_VDt}.
\end{proof}

\subsection{Proof of Lemma~\ref{lem:cont_SDt}}
\label{app:subsec_lem_cont_SDt}

\begin{proof}[Proof of Lemma~\ref{lem:cont_SDt}]
  By definition, $V_{D,g}(0)=V_{D,f}(0)$ for any $g,f\in \cF_L$, thus $S_{D,t}(g) - S_{D,t}(f) = \frac{V_{D,g}(t)-V_{D,f}(t)}{t}$, where 
\$
\frac{V_{D,g}(t)-V_{D,f}(t)}{t} &\leq \frac{Q_{D}(\sigma_{D,f}(t),A_0) + t B_D(\sigma_{D,f}(t),g)- Q_{D}(\sigma_{D,f}(t),A_0) - t B_D(\sigma_{D,f}(t),f)}{t} \\ 
&= B_D(\sigma_{D,f}(t),g) - B_D(\sigma_{D,f}(t),f) \\
&= \frac{1}{n_1}\sum_{i=1}^{n_1} \big\{ g(T_i) - g(C_{\sigma_{D,f}(t)(i)}) - f(T_i) +f(C_{\sigma_{D,f}(t)(i)}) \big\} 
\leq 2\|g-f\|_\infty. 
\$
Applying the same idea to $\sigma_{D,g}(t)$ yields 
\$
\frac{V_{D,g}(t)-V_{D,f}(t)}{t}\geq B_D(\sigma_{D,g}(t),g) - B_D(\sigma_{D,g}(t),f)\geq  -2 \|g-f\|_\infty.
\$
Combining the two directions yields the desired result.   
\end{proof}

\subsection{Proof of Lemma~\ref{lem:unif:replacement}}
\label{app:unif:replacement-proof}

\begin{proof}[Proof of Lemma~\ref{lem:unif:replacement}]
Fix $g,f\in\mathcal F_L$, $t\ne0$, and the replacement coordinate $k$.
Write $q=g-f$ and $g_u=f+uq$ for $u\in[0,1]$. Since $\mathcal F_L$ is convex, we know $g_u\in\mathcal F_L$. Let $\sigma_u$ and $\sigma'_u$ be the selected exact optimizers of
$V_{D,g_u}(t)$ and $V_{D^{(k)},g_u}(t)$, respectively.
For these fixed functions and $t$, the corresponding objectives $V_{D,g_u}(t)$ and $V_{D^{(k)},g_u}(t)$ 
are jointly measurable
in the data and $u$.  
 
Given the dataset $D$ and a fixed injection $\sigma$, the objective involved in $V_{D,g_u}(t)$ is affine in $u$:
\[ 
Q_D(\sigma,A_0)+tB_D(\sigma,g_u)
 =Q_D(\sigma,A_0)+tB_D(\sigma,f)+utB_D(\sigma,q).
\]
Thus $u\mapsto V_{D,g_u}(t)$ is the minimal of a finite family of affine functions.
By Lemma~\ref{lem:envelope}, this map is absolutely continuous and at every differentiability
point its derivative is $tB_D(\sigma_u,q)$.
Integration over $u\in [0,1]$
gives
\begin{equation}\label{eq:unif:interpolation}
 S_{D,t}(g)-S_{D,t}(f)
 =\frac{V_{D,g}(t)-V_{D,f}(t)}t
 =\int_0^1 B_D(\sigma_u,q) \ud u.
\end{equation}
The first equality used $V_{D,g}(0)=V_{D,f}(0)$ for any $g,f\in \cF_L$. 
The same result holds for $D^{(k)}$ and $\sigma'_u$. 

For each fixed $u$, Lemma~\ref{lem:unique} applied to  $g=g_u$ gives
uniqueness of $\sigma_u$ and $\sigma_{u'}$ with probability one.
The event that either optimizer is non-unique is jointly measurable, so Tonelli's theorem
gives
\[
 \EE\int_0^1
 \ind\{\text{either optimizer at }u\text{ is non-unique}\}\,du=0.
\]
That is, for almost every coupled pair $(D,D^{(k)})$, both $(\sigma_u,\sigma_u')$
are unique for (Lebesgue)-almost every $u$. 

We now revisit the language in Appendix~\ref{app:subsec_match_graph_setup}. For a fixed $u$, the cost matrix which defines the objective $Q_{D}(\sigma,A_0)+B_{D,A}(g)$ is given by 
$c_{ij}= \frac12(T_i-C_j)^\top A_0(T_i-C_j)
+t\{g_u(T_i)-g_u(C_j)\}$. 
Since the corresponding cost matrix $(c_{ij})$ and $(c_{ij}')$ for $D$ and $D'$ only differs from one row/column by replacing one treated/control data vector, Lemma~\ref{lem:control_set_stability} implies that the
selected control index sets $\mathcal S = \{\sigma_u(i)\colon i \in [n_1]\}$   and $\mathcal S'=\{\sigma_u'(i)\colon i\in [n_1]\}$ of $\sigma_u$
and $\sigma'_u$ either agree or differ by one deletion and one addition. 

We define the constant 
$\omega_q=\sup_\Omega q-\inf_\Omega q\le2\|g-f\|_\infty$. 
Write $D=(T_1,\dots,T_{n_1},C_1,\dots,C_{n_0})$ and $D^{(k)}=(T_1',\dots,T_{n_1}', C_1', \dots, C_{n_0}')$ for easier reference. 
Consider two cases of the replacement on $D$:
\begin{itemize}
    \item If $k\leq n_1$, i.e., a treated unit is replaced by a  copy, then $\big| \sum_{i=1}^{n_1} q(T_j') -  \sum_{i=1}^{n_1} q(T_j)\big|\leq \omega_q$. The control units remain the same, and due to the above arguments on $\cS$ and $\cS'$, we know $\big| \sum_{j\in \cS'} q(C_j') - \sum_{j\in \cS} q(C_j)\big|\leq  \omega_q$.  
    \item If $k>n_1$, i.e., a control unit is replaced, write $j_0 = k-n_1$ as its index. Then the treated summation is unchanged, while $ \sum_{j\in\mathcal S'}q(C'_j)-\sum_{j\in\mathcal S}q(C_j)
 =  \sum_{j\in\mathcal S'}q(C_j)
                   -\sum_{j\in\mathcal S}q(C_j) 
  +\ind\{j_0\in\mathcal S'\} \{q(C'_{j_0})-q(C_{j_0})\}.$
  The absolute value of the first term is at most $\omega_q$ since $\cS$ and $\cS'$ differs by at most one removal and one addition, and the absolute value of the second term is at most $\omega_q$. 
\end{itemize}
Therefore, in both cases, we have 
\[
 |B_D(\sigma_u,q)-B_{D^{(k)}}(\sigma'_u,q)|
 \leq \frac{2\omega_q}{n_1}
 \leq \frac4{n_1}\|g-f\|_\infty.
\]
This holds for almost every $u$ on the probability-one event we
constructed. Following~\eqref{eq:unif:interpolation}, 
we then know 
\$
\left|\{S_{D,t}(g)-S_{D,t}(f)\}
       -\{S_{D^{(k)},t}(g)-S_{D^{(k)},t}(f)\}\right|
= \bigg| \int_0^1 B_D(\sigma_u, q)\ud u 
- \int_0^1 B_{D^{(k)}}(\sigma_u, q)\ud u  \bigg| \leq \frac4{n_1}\|g-f\|_\infty.
\$
Note that the aforementioned probability-one event is needed for each fixed values of $g,f,t,k$ without requiring an intersection over the entire function
class.
\end{proof}

%% file: 132_dim_balance.tex

\section{Proof of Theorem~\ref{thm:dimension-obs}}\label{app:sec_proof_boundary} 

This entire section is devoted to the proof of Theorem~\ref{thm:dimension-obs}. 
As in Appendix~\ref{app:sec_proof_uniform_balance}, we first prove the counterexample conditional on $n_1,n_0$ (Theorem~\ref{thm:dimension_fixed}), and then return to the observational sampling model.

We continue on $\Omega=[0,1]^d$, now with the fixed matrix $A_0=I_d$. 
For now, ${n_1},{n_0}$ are deterministic integers with $1\le {n_1}\le {n_0}$, the treated observations $T_1,\ldots,T_{n_1}$ are i.i.d. with density $p_1$, the controls $C_1,\ldots,C_{n_0}$ are i.i.d. with density $p_0$, and the two samples are independent. 
In our construction, we let $p_1(x)\equiv 1$, i.e., the treated samples are uniformly distributed on $\Omega$.  

For a list of control units $C=(C_1,\dots,C_{n_0})$, define  the  optimal matching objective (since $A_0=I_d$) 
\$
\mathfrak C_k(T_1,\dots,T_k;C)=\min_{\sigma\in\cA_{k,n_0}}\frac12\sum_{i=1}^{k}\norm{T_i-C_{\sigma(i)}}^2,\quad 
\bar{\mathfrak{C}}_k = \EE[\mathfrak{C}_k], \quad \mathfrak{C}_0 = 0.
\$
Throughout, we use a fixed lexicographic order to define the first minimizer, so the optimizers are unique almost surely. 
Thus $\mathfrak C_{n_1}={n_1}\min_\sigma Q_D(\sigma,I_d)$ and we denote the imbalance 
\$
B_{{n_1},{n_0}}(g)=B_D(\sigma_n^\star,g)
= \frac{1}{n_1}\sum_{i=1}^{n_1} \{g(T_i)-g(C_{\sigma_n^\star(i)})\}.
\$

Theorem~\ref{thm:dimension_fixed} gives the lower bound for fixed group sample sizes. 

\begin{theorem}\label{thm:dimension_fixed}
For every fixed integer $d\ge4$, there are a density $p_0$ and a function $g$, both extending to $C^\infty$ functions near $\Omega=[0,1]^d$, and constants $\rho,c,C>0$, such that the following holds. The density $p_0$ is bounded above and away from zero, $0\le g\le1$, and $\|\nabla g\|_\infty\le1$. Let $p_1(x)\equiv 1$ for $x\in \Omega$, let the controls have density $p_0$, and let ${n_1}=\lfloor\rho {n_0}\rfloor$ for some fixed $\rho>0$, and ${n_0}p_0(x)\ge16{n_1}$ for all $x$. For all sufficiently large ${n_0}$, 
\@\label{sharp:eq:main-moments}
 c{n_0}^{-2/d}\le\EE [B_{{n_1},{n_0}}(g)]\le C{n_0}^{-2/d},
 \qquad \Var\{B_{{n_1},{n_0}}(g)\}\leq 2(n_1+n_0)/{n_1^2}.
\@
For $d=4$, there are fixed $c_0,p_*>0$ such that
\@\label{sharp:eq:four-prob}
 \liminf_{{n_0}\to\infty}\PP\{\sqrt{{n_1}+{n_0}}\,B_{{n_1},{n_0}}(g)\ge c_0\}\ge p_*.
\@
For $d>4$, $\sqrt{{n_1}+{n_0}}\,B_{{n_1},{n_0}}(g)\to+\infty$ in probability. 
\end{theorem} 

In the following, Appendix~\ref{app:subsec_dim_fixed_roadmap} to Appendix~\ref{app:subsec_dimension_var} are devoted to the proof of Theorem~\ref{thm:dimension_fixed}. 
Then in Appendix~\ref{app:subsec_dimension_obs} we extend this result to observational sampling. 

\subsection{Roadmap for proving Theorem~\ref{thm:dimension_fixed}}
\label{app:subsec_dim_fixed_roadmap}

We begin by a general result to motivate the plan for proving Theorem~\ref{thm:dimension_fixed}. 
We shall show that in the construction, the imbalance has a lower bounded expectation $\EE[B_{n_1,n_0}(g)]$ and a small variance $\Var\{B_{n_1,n_0}(g)\}$, which yields a high-probability lower bound on the imbalance.

For $x\in\Omega$, define the expected marginal cost of inserting a treated point
at $x$:
\begin{equation}\label{sharp:eq:J}
 J_{n_1,n_0}(x)
 =\E\left[
 \mathfrak C_{n_1}(x,T_2,\ldots,T_{n_1};C)
 -\mathfrak C_{n_1-1}(T_2,\ldots,T_{n_1};C)
 \right].
\end{equation}
Both optimization problems use the  control set $C$, and the expectation is taken over $(T_2,\dots,T_{n_1})$ and controls in $C$. We write $J(x)=J_{n_1,n_0}(x)$ in the proof when the group sizes are fixed.  

The following lemma connects the expected imbalance to this marginal cost function. 
Its proof relies on some integration identities, and is in
Appendix~\ref{sharp:sec:reduction-proof}.

\begin{lemma} 
\label{sharp:lem:IBP} 
For every $x\in\Omega$, it holds that  
$ J(x)\ge\frac12\E[\min_{1\le j\le n_0}\|x-C_j\|^2]\ge0$. 
If $2n_1\le n_0$, then
\begin{equation}\label{sharp:eq:J-integral}
 \int_\Omega J(x)\,dx
 =\overline{\mathfrak C}_{n_1}-\overline{\mathfrak C}_{n_1-1}
 \le\frac{\overline{\mathfrak C}_{2n_1}}{n_1}.
\end{equation}
Let $g(x)=q(x_1)$ for some function $q\colon [0,1]\to \RR$ which has a $C^2$ extension near $[0,1]$ and
$q'(0)=q'(1)=0$. If $\sup_{x\in [0,1]} q''(x)\le M$  for some fixed $M\ge0$, then
\begin{equation}\label{sharp:eq:bias-lower}
 \EE [B_{n_1,n_0}(g)]
 \ge-\int_\Omega q''(x_1)J(x)\,dx
      -\frac{M}{n_1}\overline{\mathfrak C}_{n_1}.
\end{equation}
Also, if $2n_1\le n_0$ and $H=\|q''\|_\infty$, then
\begin{equation}\label{sharp:eq:bias-upper}
 |\EE [B_{n_1,n_0}(g)]|
 \le\frac{H}{n_1}
       \{\overline{\mathfrak C}_{2n_1}+\overline{\mathfrak C}_{n_1}\}.
\end{equation}
\end{lemma}

We use Lemma~\ref{sharp:lem:IBP}  to lower bound the expected imbalance in the constructed example. 
In this construction, we shall choose the ``hard-to-balance'' function $q(x_1)$ by setting the curvature in the form
$q''=u-v_w$, where $u,v_w\ge0$, each has integral one on $[0,1]$, and
$M=\|u\|_\infty$. 
The function $v_w$ is supported  in a shallow region, with the parameter $w$ controlling the width of the region. Since $q''\le M$, \eqref{sharp:eq:bias-lower} becomes
\@\label{sharp:eq:three-terms}
 \EE [B_{n_1,n_0}(g)]  \geq  
 \int_\Omega v_w(x_1)J(x)\,dx 
 - \int_\Omega u(x_1)J(x)\,dx 
 - \frac{M}{n_1}\overline{\mathfrak C}_{n_1} 
\@
We shall put $v_w$ in regions where the control density is small, so that $J$ is large on its
support. 
In addition, we shall put  $u$ outside that region and keep $u\le M$ with a fixed $M$.
The negative term is then at most $M\int J$, which is controlled by
\eqref{sharp:eq:J-integral}. The term $M\overline{\mathfrak C}_{n_1} /n_1$ will be controlled using the expected total matching cost.

We present the construction in Appendix~\ref{app:subsec_dimension_construct} and we lower bound its expected imbalance in Appendix~\ref{app:subsec_dimension_exp_imb}. We then upper bound the variance of imbalance and finalize the proof of Theorem~\ref{thm:dimension_fixed} in Appendix~\ref{app:subsec_dimension_var}.

\subsection{Construction of the counterexample}
\label{app:subsec_dimension_construct}
Fix a nonnegative smooth function $u$ supported on $[1/8,1/4]$, with
$\int_0^1u=1$, and put $M=\|u\|_\infty$. 
Let $a\in(0,1/4)$ and $w=2^{-j}<1/8$ for an integer $j$ so that 
$w\le a^{2/d}$. Define
\[
 I=[1/2,1/2+w],\qquad
 I_{\mathrm{flat}}=[1/2+w/4,1/2+3w/4],\qquad
 J_w=[1/2+3w/8,1/2+5w/8].
\] 

Define
\[
 \psi(s)=c_\psi\exp\{-1/(1-s^2)\}\mathbf1\{|s|<1\},
 \qquad \int_{\mathbb R}\psi(s)\,ds=1,
 \qquad \Psi(s)=\int_{-\infty}^s\psi(t)\,dt.
\]
This function is nonnegative, supported on $[-1,1]$, and smooth on
$\mathbb R$. By definition, we see that its derivatives extend by zero at the endpoints $\{0,1\}$.
We then set 
\[
 u(s)=16\psi\{16(s-3/16)\},\qquad
 v_w(s)=\frac8w\psi\{8(s-1/2-w/2)/w\}.
\]
Their supports are $[1/8,1/4]$ and
$J_w=[1/2+3w/8,1/2+5w/8]$, respectively, and one can check that they both integrate to $1$. We take $M=\|u\|_\infty$.
Then we define
\[
 \chi_w(s)=
 \Psi\!\left(\frac{s-1/2-w/8}{w/8}\right)
 \Psi\!\left(\frac{1/2+7w/8-s}{w/8}\right).
\]
Since $\Psi(t)=0$ for $t\le-1$ and $\Psi(t)=1$ for $t\ge1$,
$\chi_w$ is supported on $I$, takes values in $[0,1]$, and equals one on
$I_{\mathrm{flat}}$. 
Set
\begin{equation}\label{sharp:eq:p0}
 s_w=\int_0^1\chi_w(s)\,ds,\qquad
 b=\frac{1-a s_w}{1-s_w},\qquad
 p_0(x)=b-(b-a)\chi_w(x_1).
\end{equation}
In particular $0\le s_w\le w<1/8$, so
\[
 1\le b=\frac{1-a s_w}{1-s_w}\le(1-w)^{-1}<2,
 \qquad b(1-s_w)+a s_w=1.
\]
The density constructed in \eqref{sharp:eq:p0} has integral one, is bounded between $a$ and
$b$. The cutoff is
smooth on the real line, so $p_0$ extends smoothly across every face of the cube $\Omega$.

Then $p_0=a$ on the flat part of $\xi_w(x_1)$, and $p_0=b$ outside the set 
$I\times[0,1]^{d-1}$, and $a\le p_0\le b<2$. Also, one can check that 
$p_0$ integrate to one.
Then, define 
\begin{equation}\label{sharp:eq:g}
 q'(s)=\int_{-\infty}^s\{u(t)-v_w(t)\}\,dt,\qquad
 q(s)=\int_0^s q'(t)\,dt,\qquad g(x)=q(x_1).
\end{equation}
The support of $u$ lies to the left of that of $v_w$. Consequently
$0\le q'(x)\le1$, $q'(0)=q'(1)=0$, $0\le q(x)\le1$, and $q''(x)=u(x)-v_w(x)\le M$. 
The functions $q(x)$ and $g(x)=q(x_1)$ extend smoothly beyond their domains, and
$\sup_\Omega\|\nabla g\|\le1$. 

\subsection{Bounds on expected imbalance}
\label{app:subsec_dimension_exp_imb}

In this part, we  prove the first expected imbalance bound in~\eqref{sharp:eq:main-moments} in Theorem~\ref{thm:dimension_fixed}. 
The next lemma shows that the constructed example gives bounds on the last two terms in
\eqref{sharp:eq:three-terms}. 
The proof is in Appendix~\ref{sharp:sec:cost}. 

\begin{lemma}
\label{sharp:lem:two-estimates}
Let $p_0$ be given by \eqref{sharp:eq:p0}, with $a,w$ constructed as above.
There are constants $c_d,K_d>0$ depending only on $d$ such that, for all
sufficiently large $n_0$,
\begin{equation}\label{sharp:eq:uniform-cost}
 \overline{\mathfrak C}_k\le K_d k n_0^{-2/d}
 \quad\text{for every integer }1\le k\le a n_0/4,
\end{equation}
and, for every $1\le n_1\le n_0$,
\begin{equation}\label{sharp:eq:local-insertion-lower}
 \inf_{x\in\Omega,x_1\in J_w}J_{n_1,n_0}(x)
 \ge c_d a^{-2/d}n_0^{-2/d}.
\end{equation}
One may take $c_d=\tfrac14(2v_d)^{-2/d}$, where $v_d$ is the volume of the
unit Euclidean ball. The constants $c_d,K_d$ do not depend on $a,w$ or the
group sizes; the lower threshold on $n_0$ may depend on $a,w$.
\end{lemma}

We now choose the parameters and compare the three terms. First, choose a fixed
$a\in(0,1/4)$ sufficiently small such that 
\begin{equation}\label{sharp:eq:choose-a}
 c_d a^{-2/d}\ge12M K_d.
\end{equation}
This is feasible since $c_d$, $M$ and $K_d$ are fixed and have no dependence on $a$. 
Next, we choose $w=2^{-j}<1/8$ with $w\le a^{2/d}$, and construct $p_0(x),g(x)$ as above.
Finally set $\rho=a/16$. 
For $n_1=\lfloor\rho n_0\rfloor$, we know 
$n_0p_0(x)\ge a n_0\ge16n_1$, and $2n_1\le a n_0/4$.
This construction thus satisfies the stated conditions in Theorem~\ref{thm:dimension_fixed}.

We now combine the above results to prove the first expected imbalance bound in~\eqref{sharp:eq:main-moments}. by controlling the three terms in~\eqref{sharp:eq:three-terms} of Lemma~\ref{sharp:lem:IBP}.
First, since $v_w$ integrates to one and is supported on $J_w$, by~\eqref{sharp:eq:local-insertion-lower} we know
\[
 \int_\Omega v_w(x_1)J(x)\,dx
 \ge c_d a^{-2/d}n_0^{-2/d}.
\]
Second, since $0\le u(s)\le M$ and $J\ge0$, \eqref{sharp:eq:J-integral} gives
\[
 \int_\Omega u(x_1)J(x)\,dx
 \le\frac{M}{n_1}\overline{\mathfrak C}_{2n_1}
 \le2M K_d n_0^{-2/d}.
\]
The third term is at most
$M\overline{\mathfrak C}_{n_1}/n_1\le M K_d n_0^{-2/d}$ by Lemma~\ref{sharp:lem:two-estimates}. 
Plugging these bounds into \eqref{sharp:eq:three-terms} yields
\begin{equation}\label{sharp:eq:positive-mean}
 \EE[ B_{n_1,n_0}(g)]
 \ge\{c_d a^{-2/d}-3M K_d\}n_0^{-2/d}
 \ge\beta n_0^{-2/d},
 \qquad \beta=\tfrac12c_d a^{-2/d}>0.
\end{equation}
This proves the desired positive lower bound. For the upper bound, let
$H=\|q''\|_\infty<\infty$. Equations~\eqref{sharp:eq:bias-upper}
and~\eqref{sharp:eq:uniform-cost} give
\[
 |\EE [B_{n_1,n_0}(g)]|\le3H K_d n_0^{-2/d} = C n_0^{-2/d}.
\]
The constant $H$ does not depend on the sample sizes. These prove the first expected imbalance bound in~\eqref{sharp:eq:main-moments}.

\subsection{Vanishing variance and completing the proof of Theorem~\ref{thm:dimension_fixed}}
\label{app:subsec_dimension_var}

We now upper bound the variance of the imbalance, whose proof is deferred to
Section~\ref{sharp:sec:variance}.  

\begin{prop}
\label{sharp:prop:variance}
Let $1\le n_1\le n_0$. Suppose all $n_1+n_0$ covariate vectors are independent
and have Lebesgue densities on $\Omega$. Let $R_g=\sup_\Omega g-\inf_\Omega g$. Then
\begin{equation}\label{sharp:eq:variance}
 \Var\big\{B_{n_1,n_0}(g)\big\}\le  2R_g^2(n_1+n_0) / n_1^2.
\end{equation}
\end{prop}

\paragraph{Completing the proof of Theorem~\ref{thm:dimension_fixed}.}
For the constructed $g$, we know $R_g\le1$. For all large sufficiently large $n_0$,
$n_1\ge\rho n_0/2$, so the variance in~\eqref{sharp:eq:variance} is at most $C_v/n_0$ with $C_v=4/\rho+8/\rho^2$.
The one-sided variance bound in Lemma~\ref{lem:expect_var_bd}, applied with
mean at least $\beta n_0^{-2/d}$ in~\eqref{sharp:eq:main-moments}, gives
\begin{equation}\label{sharp:eq:prob-explicit}
 \PP\Big\{B_{n_1,n_0}(g)\ge\tfrac12\beta n_0^{-2/d}\Big\}
 \ge\frac{\beta^2}{4C_v n_0^{4/d-1}+\beta^2},
\end{equation}
where we used the fact that   $(\beta/2)n_0^{-2/d}$ is at most half the true mean, and the
bound $\mu^2/(4v+\mu^2)$ increases with $\mu>0$.
When $d=4$, the right side is the fixed positive constant
$p_*=\beta^2/(4C_v+\beta^2)$; the event implies
$\sqrt{n_1+n_0}\,B_{n_1,n_0}(g)\ge\beta/2$.
When $d>4$, the right side tends to one, and
$\sqrt{n_1+n_0}\,n_0^{-2/d}\ge n_0^{1/2-2/d}\to\infty$.
This proves Theorem~\ref{thm:dimension_fixed}.

\subsection{Proof of Theorem~\ref{thm:dimension-obs} under observational sampling}\label{app:subsec_dimension_obs}

For the fixed $d\ge4$, use the corresponding $p_0,g,a,w,\rho$ in the above construction.
Let $\pi=\rho/(1+\rho)$ and specify
$\PP(Z=1)=\pi$, $X\mid Z=1\sim\operatorname{Unif}(\Omega)$, and
$X\mid Z=0\sim p_0$. Its marginal density and propensity are
\begin{equation}\label{sharp:eq:iid-law}
 f(x)=\frac{\rho+p_0(x)}{1+\rho},\qquad
 e(x)=\frac{\rho}{\rho+p_0(x)}.
\end{equation}
They are smooth, $f$ is positive and bounded, and
$0<\rho/(\rho+b)\le e(x)\le1/17$.
Thus the covariate and propensity regularity conditions in the observational
setup hold with fixed constants.

Under the observational sampling, $n_1=\sum_{i=1}^n Z_i$ and $n_0=n-n_1$ are random. On
$\mathcal I_n=\{2\le n_1\le n_0\}$, 
we use the exact globally optimal matching $\sigma_n^\star$ with $A_0=I_d$ and write
$B_n(g)=B_{n,\sigma_n^\star}(g)$; we 
set $B_n(g)=0$ on $\mathcal I_n^c$.
For subsequent inference statements, on the event $\cI_n^c$, all estimates and studentized statistics are set to be 
zero and we do not reject any test. 
On the event $\mathcal I_n$, we set a studentized statistic to zero when its variance
estimate is zero. Since $n_1/n\to\pi\in(0,1/2)$, we know $\PP(\mathcal I_n)\to1$.

\begin{proof}[Proof of Theorem~\ref{thm:dimension-obs}]
Conditional on a treatment-label vector, the covariates have law
$P_1^{\otimes n_1}\otimes P_0^{\otimes n_0}$. Define
\[
 \mathcal A_n=\mathcal I_n\cap\{\rho/2\le n_1/n_0\le2\rho\}.
\]
Then we know $\PP(\mathcal A_n)\to1$. On this event,
$2n_1/n_0\le4\rho=a/4$ and $n_1\ge\rho n_0/2$. The bounds in Appendix~\ref{app:subsec_dimension_construct} therefore hold with the same constants for all
these counts. In particular, these results state that
\[
 \E\{B_n(g)\mid n_1,n_0\}\ge\beta n_0^{-2/d},\qquad
 \Var\{B_n(g)\mid n_1,n_0\}\le C_v/n_0,
 \qquad C_v=4/\rho+8/\rho^2.
\]
Applying Lemma~\ref{lem:expect_var_bd} conditionally gives
\[
 \PP\{B_n(g)\ge(\beta/2)n_0^{-2/d}\mid n_1,n_0\}
 \ge\frac{\beta^2}{4C_v n_0^{4/d-1}+\beta^2}
\]
uniformly on $\mathcal A_n$ for all sufficiently large $n$.
If $d=4$, its right side is $p_*>0$, and its event implies
$\sqrt{n_1}B_n(g)\ge(\beta/2)\sqrt{\rho/2}=:b_*>0$.
Averaging over $\mathcal A_n$ proves \eqref{sharp:eq:iid-prob}.
If $d>4$, the conditional lower bound tends to one uniformly on
$\mathcal A_n$, because $n_0\ge n/(1+2\rho)$ there. On the corresponding event,
\[
 \sqrt n B_n(g)\ge(\beta/2)\sqrt n\,n_0^{-2/d}
 \ge(\beta/2)n^{1/2-2/d}\to\infty.
\]
Averaging again proves divergence in probability. 
\end{proof}

\section{Supporting lemmas and proofs for Appendix~\ref{app:sec_proof_boundary}}\label{sharp:sec:technical}

This appendix proves the inputs used in Appendix~\ref{app:sec_proof_boundary}.

\subsection{Proof of Lemma~\ref{sharp:lem:IBP}}
\label{sharp:sec:reduction-proof}
\begin{proof}[Proof of Lemma~\ref{sharp:lem:IBP}]

We use the graph language in Appendix~\ref{app:subsec_match_graph_setup} to state the proof. 
For any values of $(T_2,\dots,T_{n_1})$ and $C$, matching the treated point at $x$ pays
at least $\tfrac12\min_j\|x-C_j\|^2$ in the full matching objective. 
Removing this edge $(x,C_j)$ yields an injection that matches the remaining $n_1-1$ treated points to the full
control list. The cost of those remaining edges is at least
$\mathfrak C_{n_1-1}$. 
Subtracting this term from both sides and taking the expectation therefore gives
$J(x)\ge\tfrac12\E\min_j\|x-C_j\|^2\ge0$.

By definition, integrating $x$ under the uniform distributed treated samples gives
\[
 \int_\Omega J(x)\,dx
 =\overline{\mathfrak C}_{n_1}-\overline{\mathfrak C}_{n_1-1}.
\]
Next, we prove the upper bound on the right-hand  side. The proof of the following lemma is in Appendix~\ref{app:subsec_proof_lem_supermod}.

\begin{lemma}\label{sharp:lem:supermod}
For a finite set $S$ of such vertices, let
$\mathfrak C_C(S)$ be the minimum total cost of matching every unit in $S$
to a distinct member of $C$.
Let $S\subseteq S'$ be finite sets of labeled treated vertices, let $x$ be an
additional treated unit not in $S'$, and suppose $|S'|+1\le n_0$. Then for the same control list $C$,
\begin{equation}\label{sharp:eq:supermod}
 \mathfrak C_C(S\cup\{x\})-\mathfrak C_C(S)
 \le
 \mathfrak C_C(S'\cup\{x\})-\mathfrak C_C(S').
\end{equation} 
\end{lemma}

Apply Lemma~\ref{sharp:lem:supermod} with 
\[
 S=\{T_1,\ldots,T_k\},\quad
 S'=S\cup\{T_{k+1}\},\quad x=T_{k+2}.
\]
The expectation of the left increment in~\eqref{sharp:eq:supermod} is
$\overline{\mathfrak C}_{k+1}-\overline{\mathfrak C}_k$, and the expectation
of the right increment in~\eqref{sharp:eq:supermod} is
$\overline{\mathfrak C}_{k+2}-\overline{\mathfrak C}_{k+1}$.
Thus successive expected increments are nondecreasing. 
Furthermore,  $\mathfrak C_C(S)$ cannot decrease when a treated point is added since the cost of matching a treated-control pair is nonnegative. 
If $2n_1\le n_0$, every increment from
$\overline{\mathfrak C}_{n_1}$ through $\overline{\mathfrak C}_{2n_1}$ is at
least $\overline{\mathfrak C}_{n_1}-\overline{\mathfrak C}_{n_1-1}$. Hence
\[
 n_1\{\overline{\mathfrak C}_{n_1}-\overline{\mathfrak C}_{n_1-1}\}
 \le\overline{\mathfrak C}_{2n_1}-\overline{\mathfrak C}_{n_1}
 \le\overline{\mathfrak C}_{2n_1},
\]
or equivalently
\begin{equation}\label{sharp:eq:increment-average}
 0\le\overline{\mathfrak C}_{n_1}-\overline{\mathfrak C}_{n_1-1}
 \le\frac{\overline{\mathfrak C}_{2n_1}}{n_1}.
\end{equation}

Fix the controls and all treated points except the inserted $x$. If $x$ is
matched to control $j$, it is straightforward to see that the other treated points must be optimally matched to
$C\setminus\{C_j\}$. The total optimum as a function of $x$ is therefore
\[
 F(x)=\min_{1\le j\le n_0} F_j(x),\quad 
 \text{where} \quad F_j(x) :=  \frac12\|x-C_j\|^2+
 \mathfrak C_{C\setminus\{C_j\}}(\{T_2,\ldots,T_{n_1}\}) .
\]
The second term in $F_j(x)$ is constant in $x$ for each fixed $j$. 
All the $F_j(x)$ functions have the common 
quadratic part $\|x\|^2/2$, and their pairwise differences $F_j(x)-F_{j'}(x)$ are affine in $x$. 
Whenever $C_j\neq C_{j'}$, 
the affine difference $F_j(x)-F_{j'}(x)$ is non-constant. Since controls have a
density, the $C_j$'s are distinct almost surely. Thus, except for a finite union of
hyperplanes, the minimizing index $j(x)$ of $F_j(x)$ is unique and
$\partial_1F(x)=x_1-C_{j(x),1}$.
Moreover function $F$ is Lipschitz, since the gradient of every $F_j(x)$ has a norm of at most
$\sqrt d$ on $\Omega$.
It is continuously differentiable on finitely many
intervals and continuous at their endpoints.

For fixed $x_{-1}=(x_2,\ldots,x_d)$, we integrate over $q'(x_1)\partial_1F(x)$ by parts
on these subintervals. By continuity, the interior endpoints are cancelled, and
$q'(0)=q'(1)=0$ cancels the two exterior terms. This gives
\[
 \int_0^1q'(s)\partial_1F(s,x_{-1})\ud s
 =-\int_0^1q''(s)F(s,x_{-1})\ud s.
\]
Since $T_{11}$ (the first coordinate of the sample $T_1$) follows the uniform distribution, the left-hand side is
\$
\int_0^1q'(s)\partial_1F(s,x_{-1})\ud s = \EE\Big[ q'(T_{11}) \cdot\big\{ T_{11} - C_{\sigma_n^\star(1),1} \big\} \biggiven T_{-1}=x_{-1},T_{2:n_1},C\Big].
\$
Since $\int_0^1q''(x)\ud x=0$,  we can subtract  the constant
$\mathfrak C_{n_1-1}(T_2,\ldots,T_{n_1};C)$ on the right-handed side, hence 
\$
 \EE\Big[ q'(T_{11}) \cdot\big\{ T_{11} - C_{\sigma_n^\star(1),1} \big\} \biggiven T_{-1}=x_{-1},T_{2:n_1},C\Big] = -\int_0^1q''(s) \big\{ F(s,x_{-1})- \mathfrak C_{n_1-1}(T_2,\ldots,T_{n_1};C) \big\} \ud s .
\$
Integrating over $x_{-1}$ (each with  uniform distribution on $[0,1]$) gives (Fubini's Theorem applies for these bounded terms)
\$
 \EE\Big[ q'(T_{11}) \cdot\big\{ T_{11} - C_{\sigma_n^\star(1),1} \big\} \biggiven T_2,\dots,T_{n_1},C\Big] = -\int_\Omega q''(x_1)\big\{ F(s,x_{-1})- \mathfrak C_{n_1-1}(T_2,\ldots,T_{n_1};C) \big\}\ud x .
\$
Since the treated samples are exchangeable, the same calculation applies for each of the remaining observations. Averaging over these observations, we obtain
\begin{equation}\label{sharp:eq:gradient-identity}
 \E\left[\frac1{n_1}\sum_{i=1}^{n_1}
 q'(T_{i1})(T_{i1}-C_{\sigma_n^\star(i),1})\right]
 =-\int_\Omega q''(x_1)J(x)\,dx.
\end{equation}

For $s,t\in[0,1]$, Taylor expansion and $q''\le M$ give
\[
 q(t)\le q(s)+q'(s)(t-s)+\frac M2(t-s)^2.
\]
Applying this $s=T_{i1}$ and $t=C_{\sigma_n^\star(i),1}$ and averaging gives 
\[
 \EE [B_{n_1,n_0}(g)]
 \ge-\int_\Omega q''(x_1)J(x)\,dx
 -\frac{M}{2n_1}\E\sum_{i=1}^{n_1}
                  \|T_i-C_{\sigma_n^\star(i)}\|^2.
\]
The last sum has expectation $2\overline{\mathfrak C}_{n_1}$, giving
\eqref{sharp:eq:bias-lower}. For the absolute bound, use the absolute Taylor
remainder with $H=\|q''\|_\infty$. The absolute gradient contribution is at
most $H\int J\le H\overline{\mathfrak C}_{2n_1}/n_1$, and the remainder is at
most $H\overline{\mathfrak C}_{n_1}/n_1$. This proves
\eqref{sharp:eq:bias-upper}.
\end{proof}

\subsection{Proof of Lemma~\ref{sharp:lem:supermod}}\label{app:subsec_proof_lem_supermod}

\begin{proof}[Proof of Lemma~\ref{sharp:lem:supermod}]
Take an optimal matching $M_1$ for treated units in $S$ and an optimal matching $M_2$ for treated samples in 
$S'\cup\{x\}$. 
We use the graph language in Appendix~\ref{app:subsec_match_graph_setup}, viewing each as a graph with edges between matched pairs. 
We color the edges from the two matching in distinct colors, where two copies of any common edge are preserved. 
Starting at $x$, we follow its $M_2$ edge to reach a control. If this
control is unmatched in $M_1$, then stop; otherwise, we follow its $M_1$ edge to a treated
vertex in $S$, and then that vertex's edge in $M_2$. We continue this process until stopping at a control vertex that is unmatched in $M_1$. 
Every treated vertex reached after $x$ belongs to $S$ and has edges in both $M_1$ and $M_2$.
This component must be a path, because along the way the vertices reached must alternate between a control matched in $M_2$ and a treated vertex matched to the preceding control in $M_1$, and finally stop at a control unit matched in $M_2$ but not in $M_1$. This path cannot be attached to a cycle because every internal vertex must have a degree of $1$ and the vertex $x$ has a degree of $1$, while attaching a cycle would create a vertex with degree at least $3$. 

By construction, edges along this path alternates between $M_1$ and $M_2$. We now exchange colors along the path. The first matching now covers treated vertices in $S\cup\{x\}$, 
and the second covers treated vertices in $S'$. Each internal vertex loses one edge of each relevant
color and gains one, so each matching of the corresponding colors remain injections. Their combined cost
is unchanged. Minimizing separately over the two new treated sets gives
\[
 \mathfrak C_C(S\cup\{x\})+\mathfrak C_C(S')
 \le\mathfrak C_C(S)+\mathfrak C_C(S'\cup\{x\}),
\]
which is equivalent to \eqref{sharp:eq:supermod}. This argument holds even if there are ties. 
\end{proof}

\subsection{Proof of Lemma~\ref{sharp:lem:two-estimates}}\label{sharp:sec:cost}

\begin{proof}[Proof of Lemma~\ref{sharp:lem:two-estimates}]
Partition $\Omega$ into adjacent cubes of side length $w=2^{-j}$, where we use half-open cubes  except on the outer faces to obtain a disjoint partition. Let $\cQ_w$ be the collection of these cubes, and let $Q\in \cQ_w$. 
Suppose $p_0\ge a_Q$ on $Q$ for some fixed numbers $a_Q\ge a>0$.
Let $k$ be the number of treated points which obeys $1\le k\le a n_0/4$. 
For any cube $Q$, we call $H$ a subcube of $Q$ if there exists some $0<h\leq 2$ and $y\in \RR^d$ such that $H=Q(y,h)$ and $z_\ell\leq y_\ell\leq z_\ell +w-h$ for $\ell=1,\dots,d$, where we define $Q(y,h)=\prod_{\ell=1}^d I(y_\ell,h)$, and $I(s,h)=[s,s+h)$ if $s+h<1$ and $[s,1]$ if $s+h=1$. 

Define the constant 
\[
 C_d=2d+4d\sum_{j=1}^\infty4^j\exp\{-2^{d(j-1)}/4\}<+\infty. 
\]

We will use the following lemma, whose proof is in Appendix~\ref{app:subsec_dimension_lem_cost}. 
\begin{lemma}\label{sharp:lem:cost} 
For all sufficiently large $n_0$, it holds  uniformly over $1\le k\le a n_0/4$ that 
\begin{equation}\label{sharp:eq:cost-bound}
 \overline{\mathfrak C}_k
 \le C_d k n_0^{-2/d}\sum_{Q\in \cQ_w}|Q|a_Q^{-2/d}
       +d k^2\exp(-n_0a w^d/4).
\end{equation}  
\end{lemma}

We first prove~\eqref{sharp:eq:uniform-cost} by applying Lemma~\ref{sharp:lem:cost}. 
Recall that
$I=[1/2,1/2+w]$ and $w=2^{-j}$. Both endpoints of $I$
are grid points of the partition $\mathcal Q_w$.
We may therefore take $a_Q=a$ for the cells whose
first-coordinate intervals lie between these endpoints,
and $a_Q=b$ for all other cells. The two groups of cells
have total volumes $w$ and $1-w$, respectively. Hence
\$
\textstyle{\sum_{Q\in\mathcal Q_w}}|Q|a_Q^{-2/d}
=
w a^{-2/d}+(1-w)b^{-2/d}
\le 2,
\$
where we used $w\le a^{2/d}$ (by construction) and $b\ge1$.
Substituting this into~\eqref{sharp:eq:cost-bound} gives
\[
\overline{\mathfrak C}_k
\le
2C_d k n_0^{-2/d}
+
d k^2\exp(-n_0aw^d/4).
\]
 
We next bound the second term. Since $k\le n_0$, we have 
\[
\frac{d k^2\exp(-n_0aw^d/4)}{k n_0^{-2/d}}
=
d k n_0^{2/d}\exp(-n_0aw^d/4)
\le
d n_0^{1+2/d}\exp(-n_0aw^d/4) \leq 1
\]
for all sufficiently large $n_0$ and $1\le k\le an_0/4$. Consequently, we have
\[
\overline{\mathfrak C}_k
\le (2C_d+1)k n_0^{-2/d}.
\]
This proves~\eqref{sharp:eq:uniform-cost} with $K_d=2C_d+1$.
Note that the constant $K_d$ depends only on $d$.

We now prove~\eqref{sharp:eq:local-insertion-lower}. Fix $x\in\Omega$ with $x_1\in J_w$.
Let $v_d$ denote the volume of the unit Euclidean ball
in $\mathbb R^d$, and set 
$s_{n_0}=(2v_d n_0a)^{-1/d}$. 
Since $a,w$ are fixed, we have $s_{n_0}\le w/8$ for all
sufficiently large $n_0$.

Recall that the distance between every point of $J_w$ and either endpoint of $I_{\mathrm{flat}}$ is at least $w/8$.
If $y\in\Omega$ satisfies $\|y-x\|<s_{n_0}$, then we must have 
\[
|y_1-x_1|\le\|y-x\|<s_{n_0}\le w/8.
\]
Thus $y_1\in I_{\mathrm{flat}}$, and so $p_0(y)=a$.
It follows that, for each control observation $C_j$,
\$
\mathbb P\{\|C_j-x\|<s_{n_0}\}
&=
\int_{\{y\in\Omega:\,\|y-x\|<s_{n_0}\}}p_0(y)\,dy \le a v_d s_{n_0}^d
=\frac{1}{2n_0}.
\$ 
The probability that at least one control lies within
distance $s_{n_0}$ of $x$ is therefore at most
\[
\mathbb P\left\{\min_{1\le j\le n_0}\|x-C_j\|<s_{n_0}\right\}
\le
\sum_{j=1}^{n_0}
\mathbb P\{\|x-C_j\|<s_{n_0}\}
\le\frac12.
\]
Equivalently,
\[
\mathbb P\left\{\min_{1\le j\le n_0}\|x-C_j\|\ge s_{n_0}\right\}
\ge\frac12.
\]
On this event, the squared distance from $x$ to its nearest
control is at least $s_{n_0}^2$. Hence Lemma~\ref{sharp:lem:IBP} gives
\$
J(x) \ge
\frac12\EE\Big[\min_{1\le j\le n_0}\|x-C_j\|^2 \Big]\geq 
\frac{s_{n_0}^2}{2}
\mathbb P\left\{
\min_{1\le j\le n_0}\|x-C_j\|\ge s_{n_0}
\right\} \ge \frac{s_{n_0}^2}{4}
=
\frac14(2v_d)^{-2/d}a^{-2/d}n_0^{-2/d}.
\$
The required condition $s_{n_0}\le w/8$ does not depend
on $x$ or $n_1$. Thus the same bound holds for every
$x\in\Omega$ with $x_1\in J_w$ and every
$1\le n_1\le n_0$, which proves~\eqref{sharp:eq:local-insertion-lower}.
\end{proof}

\subsection{Proof of Lemma~\ref{sharp:lem:cost}}
\label{app:subsec_dimension_lem_cost}

\begin{proof}[Proof of Lemma~\ref{sharp:lem:cost}]
We continue with the language in the proof of Lemma~\ref{sharp:lem:two-estimates}. In each cube $Q\in \cQ_w$, choose an integer $J_Q\ge0$ and a finest side length
$h_0=w2^{-J_Q}$ satisfying
$1\le n_0a_Qh_0^d<2^d$. Such a choice exists for all sufficiently large $n_0$.
In particular, $h_0^2\le4(n_0a_Q)^{-2/d}$. 
We denote the set of these finest cubes as $\cQ_{w,h_0}$. Merging every collection of $2^d$ cubes in $\cQ_{w,h_0}$ produces a collection of cubes of side length $2h_0$, which we denote as $\cQ_{w,2h_0}$. We can continue this process to produce the sets of cubes $\cQ_{w,2^j h_0}$ for $j=1,\dots,J_Q$, each of which is obtained by merging $2^d$ adjacent cubes in $\cQ_{w,2^{j-1} h_0}$. 

We construct a matching as follows. 
First, we match as many points as possible within each finest cube $H\in \cQ_{w,h_0}$. 
Then, we merge the $2^d$ ``child'' cubes as described before, and within each $H\in \cQ_{w,2h_0}$, we match as many of their remaining treated and
control points as possible. 
We continue this process to finally produce the cube $Q$ of side length $w$. 
In this process, after producing and matching within any cell $H$, one can see that there are exactly $(N_{H,1}-N_{H,0})_+$ treated points remain.  

We now bound the total cost $\|T_i-C_j\|^2/2$ for the matched pairs $(T_i,C_j)$ in the above process. 
First, all pairs matched within $H\in \cQ_{w,h_0}$ satisfy  $\|T_i-C_j\|^2/2\leq dh_0^2/2$, 
and there are at most as many such
edges in $Q$ as the number of treated points in $Q$, whose expectation is $k|Q|$. Their expected
total cost is thus at most $2d k|Q|(n_0a_Q)^{-2/d}$.
Similarly, for pairs matched after producing $\cQ_{w,2^j h_0}$ at each level $j\ge1$, the side length is $h_j=2^jh_0$; 
moreover, every new pair uses a treated
point left at the preceding level, so the number of new pairs are upper bounded by 
\$
\textstyle{\sum_{H\in \cQ_{w,2^{j-1}h_0}}} ( N_{H,1}-N_{H,0})_+,
\$
where $N_{H,1}$ and $N_{H,0}$ are the number of treated and control samples in $H$, respectively. 
The following lemma bounds the expectation of this number.

\begin{lemma}\label{sharp:lem:deficit}
For a cube $H\subseteq Q$ of volume $v=|H|=h^d$, let $N_{H,1}$ and $N_{H,0}$ be the
treated and control counts in $H$. Then
\begin{equation}\label{sharp:eq:deficit}
 \EE\big[(N_{H,1}-N_{H,0})_+\big]
 \le2kv\exp(-n_0a_Qv/4),
 \quad \PP(N_{H,1}>N_{H,0})\leq 2kv\exp(-n_0a_Qv/4).
\end{equation} 
\end{lemma}
\begin{proof}[Proof of Lemma~\ref{sharp:lem:deficit}]
Write $M_H=N_{H,1}$ and $K_H=N_{H,0}$. 
We know they are independent of each other, and $M_H\sim \text{Binom}(k,v)$, $K_H\sim \text{Binom}(n_0,p_H)$ for the marginal probability $p_H\ge a_Qv$.
For $t>0$, we know $(M_H-K_H)_+\le M_H e^{t(M_H-K_H)}$, so by independence and the
generating function of the Binomial distribution,  
\$
 \EE\big[(M_H-K_H)_+\big] \le kv e^t(1-v+ve^t)^{k-1}
           (1-p_H+p_He^{-t})^{n_0}.
\$
Since $1+u\le e^u$, taking  $t=\log2$, the right-handed side is upper bounded by 
\[
2kv(1+v)^{k-1}(1-p_H/2)^{n_0} \leq  2kv\exp(kv-n_0p_H/2)
 \le2kv\exp(-n_0a_Qv/4),
\]
where the second inequality uses the facts that $v\leq 1$ and $k\le n_0a/4\le n_0a_Q/4$. 
Finally, we know $\ind\{M_H>K_H\} \leq (N_{H,1}-N_{H,0})_+$, which gives the second inequality.  
\end{proof}

Therefore, we know the expected number of these pairs satisfies
\[
\textstyle{\sum_{H\in \cQ_{w,2^{j-1}h_0}}} \EE\big[( N_{H,1}-N_{H,0})_+\big] \leq  2k|Q|\exp\{-n_0a_Q h_0^d2^{d(j-1)}/4\}
 \le2k|Q|\exp\{-2^{d(j-1)}/4\}.
\]
These pairs incur the cost $\|T_i-C_{\sigma(i)}\|^2/2$ of at most $dh_0^22^{2j}/2$, the diameter of cubes in $\cQ_{w,2^j h_0}$ at level $j$. 
Thus, the expected total cost incurred by new pairs matched at level $j$ is at most
\[
2k|Q|\exp\{-2^{d(j-1)}/4\} \cdot 4d(n_0a_Q)^{-2/d} 4^j/2 \leq 
 4d k|Q|(n_0a_Q)^{-2/d}
       4^j\exp\{-2^{d(j-1)}/4\}.
\]
Summing the above cost over all levels $j$ therefore upper bounds the expected matching cost in $Q$ produced in this process by
$C_d k|Q|(n_0a_Q)^{-2/d}$.

Let $\cE_n'$ be the event where every $Q\in \cQ_w$ has at least as many controls as treated points. 
Then we know on $\cE_n'$, the matched pairs from the above process form a full injection. 
In addition, a union bound over the probability bound in~\eqref{sharp:eq:deficit} yields 
\[
\PP(\cE_n')\geq 1- \textstyle{\sum_{Q\in \cQ_w}} 2k|Q|\exp(-n_0a_Q|Q|/4)
 \geq 1- 2k\exp(-n_0a w^d/4).
\]
On the event $(\cE_n')^c$, we discard the above construction and take any injection whose total matching cost is at most $dk/2$.  
Combining the two cases yields~\eqref{sharp:eq:cost-bound}.
\end{proof}

\subsection{Proof of Proposition~\ref{sharp:prop:variance}}\label{sharp:sec:variance}
\begin{proof}[Proof of Proposition~\ref{sharp:prop:variance}]
Similar to the imbalance upper-bound proof, we consider the one-coordinate replacement on the data and the resulting changes in $B_{n_1, n_0}(g)$ to prove its variance bound. Recall 
\[
D=(T_1,\ldots,T_{n_1},C_1,\ldots,C_{n_0}),
\qquad
\sigma_D^\star\in\arg\min_{\sigma\in\mathcal A_{n_1,n_0}}
Q_D(\sigma,A).
\]
Write $S_D=\{\sigma_D^\star(i)\colon 1\leq i\leq n_1\}$ for the selected control set and the imbalance measure
\[
 B_D 
 =\frac1{n_1}\sum_i g(T_i)-\frac1{n_1}\sum_{j\in S_D}g(C_j),
\]
so $B_{n_1,n_2}(g)=B_D$. 

Let $D^{(j)}$ be a one-coordinate replacement of $D$ (recall the definition in Appendix~\ref{app:subsec_imbalance_fixed_theorem}). 
When $j\leq n_1$, a treated sample is replaced by an independent copy from the treated covariate distribution. When $j>n_1$, a control sample is replaced by an independent copy from the control covariate distribution.

First, the optimizer $\sigma_D^\star$ is unique almost surely for both the original dataset $D$ and its one-coordinate replacement since the covariates have a density; see, e.g., the proof of Lemma~\ref{sharp:lem:IBP} for the argument. 

We use the selected-set stability property proved in Lemma~\ref{lem:control_set_stability}: when the cost matrix $c_{ij}=\|T_i-C_j\|^2$ differ in one row or one column and both optima
are unique, their selected-control index sets differ by at most one deletion
and one addition.  
We consider two cases for the replacement in $D^{(j)}$:
\begin{itemize}
    \item If a treated point is replaced, then the cost matrices $(c_{ij})$ of $D$ and $(c_{ij}')$ of $D^{(j)}$ differ by one row. The selected-control set differs by at most one deletion and one addition. Thus $|B_{D^{(j)}}-B_D|\le 2R_g/n_1$. 
    \item If a control point is replaced, with abuse of notation we denote the original index as $j\in [n_0]$. Then the treated average doesn't change, and the selected-control set differs by at most one deletion and one addition. We have the identity
    \$
n_1\{B_{D^{(j)}}-B_{D}\}
= 
- \sum_{\ell\in S_{D^{(j)}}}g(C_\ell)
+\sum_{\ell\in S_D}g(C_\ell)
-\mathbf1\{j\in S_{D^{(j)}}\}\{g(C_j')-g(C_j)\}.
\$
The difference of the first two terms are at most $R_g$ due to Lemma~\ref{lem:control_set_stability}. The last term is also bounded by $R_g$.
Therefore 
$|B_{D^{(j)}}-B_D|\le 2R_g/n_1$. 
\end{itemize} 
Combining the two cases and applying Lemma~\ref{lem:ES}, we have 
\[
 \Var(B_{n_1,n_2}(g)) \leq \frac{1}{2}\sum_{j=1}^n \EE[(B_D - B_{D^{(j)}})^2] \le\frac{n}{2}
 \frac{4R_g^2}{n_1^2} = \frac{2R_g^2(n_1+n_0)}{n_1^2}.
\] 
This proves Proposition~\ref{sharp:prop:variance}.
\end{proof}

%% file: 141_wald_att.tex

\section{Proofs for super-population inference}\label{app:sec_inference}

\subsection{Proof of Theorem~\ref{thm:ATT}}
\label{app:subsec_att_proof} 
Recall that $\cG_n = \sigma((X_i,Z_i)_{i=1}^n)$ is the $\sigma$-field used in constructing the matching. 
We first restate the balance results from Theorem~\ref{thm:observational} as preparation.

\begin{lemma}\label{inf:lem:geometry}
Under Assumptions~\ref{assump:obs-sample} and~\ref{assump:match}, and with the zero convention on $\mathcal I_n^c$, we have 
\[
 \sqrt n\sup_{g\in\cF_L}|B_{n,\hat\sigma_n}(g)|\to0,\qquad
 L_n(\hat\sigma_n)=O_P(r_n^2)+o_P(n^{-1}),
\]
where $r_n\asymp(\log n/n)^{1/d}$. In particular, $L_n(\hat\sigma_n)\to0$ and $\sqrt{L_n(\hat\sigma_n)}\log n\to0$ in probability. 
\end{lemma}

\begin{proof}[Proof of Lemma~\ref{inf:lem:geometry}]
The first claim follows from Theorem~\ref{thm:observational}. For the distance bound, use the event $\cE_n$ and the choice of $r_n\asymp (\log n/n)^{1/d}$ from Lemma~\ref{lem:count_stable}. 
Conditional on the treated/control counts, the probability of $\cE_n^c$ is at most $n^{-b}$, uniformly over $|n_1/n-\pi|\le\eta$ for some fixed $\eta>0$, by the arguments in the proof of Theorem~\ref{thm:observational}. On the event $\cE_n$, there are abundant control samples within every cubic $Q$, so there exists an injection $\bar\sigma$ so that each $(T_i,C_{\bar\sigma(i)})$ are within the same cubic of diameter $r_n$. On the event  $\cE_n \cap \{\|\hat A-A_0\|_{\mathrm{op}}\le\lambda_{\min}(A_0)/2\}$  whose probability tends to one, we have 
\[
 \frac{\lambda_{\min}(A_0)}4L_n(\hat\sigma_n)
 \le Q_n(\hat\sigma_n;\hat A)
 =\min_\sigma Q_n(\sigma;\hat A)+\xi_n
 \le Q_n(\bar\sigma;\hat A)+\xi_n
 \le Cr_n^2+\xi_n.
\]
The fact that $\xi_n=o_P(n^{-1})$ then proves the bound on $L_n(\hat\sigma_n)$. The last conclusion follows from $r_n\log n\to0$ and $n^{-1/2}\log n\to0$.
\end{proof}

\begin{lemma}\label{lem:transfer}
If $h$ is continuous on $\Omega$ and $L_n(\hat\sigma_n)\to0$ in probability, then
\[
 \frac1{n_1}\sum_{i=1}^{n_1}|h(C_{\hat\sigma_n(i)})-h(T_i)|\to0
 \quad\text{in probability}.
\]
\end{lemma}
\begin{proof}[Proof of Lemma~\ref{lem:transfer}]
For any $u>0$, define  $\omega_h(u)=\sup_{\|x-y\|\le u}|h(x)-h(y)|$. Then, the displayed average is at most $\omega_h(u)+2\|h\|_\infty L_n(\hat\sigma_n)/u^2$. 
For every fixed $u>0$, letting $n\to\infty$ yields 
\$
 \frac1{n_1}\sum_{i=1}^{n_1}|h(C_{\hat\sigma_n(i)})-h(T_i)| \leq \omega_h(u) + o_P(1). 
\$
Since $h$ is continuous on the compact set $\Omega$, we know $\lim_{u\downarrow 0} \omega_h(u)=0$. Hence by the arbitrariness of $u>0$ we must have the desired result. 
\end{proof}

\begin{proof}[Proof of Theorem~\ref{thm:ATT}]
We now begin the proof of Theorem~\ref{thm:ATT}. 
\paragraph{Asymptotic normality for $\hat\theta_{\att}$.} Write $b_i=\mu_0(T_i)-\mu_0(C_{\hat\sigma_n(i)})$,
$\varepsilon_{1,i}=Y_{1,i}-\mu_1(T_i)$,
$\varepsilon_{0,j}=Y_{0,j}-\mu_0(C_j)$, and
$\eta_i=\varepsilon_{1,i}-\varepsilon_{0,\hat\sigma_n(i)}$.
Then $R_i=\tau(T_i)+b_i+\eta_i$, which yields the decomposition
\@\label{eq:ATT-decomp}
 \sqrt {n_1}(\hat\theta_\att-\theta_\att)
 = \underbrace{\frac1{\sqrt {n_1}}\sum_i\{\tau(T_i)-\theta\}}_{\textnormal{(i)}}
 +\underbrace{\sqrt {n_1} B_{n,\hat\sigma_n}(\mu_0)}_{\textnormal{(ii)}}
 +\underbrace{\frac1{\sqrt {n_1}}\sum_i\eta_i}_{\textnormal{(iii)}}.
\@

Conditional on $\cG_n$, the residuals of distinct observations are independent and mean zero. This follows from i.i.d. sampling, conditioning on their individual covariates and labels, and unconfoundedness. Since the matching is an injection, different $\eta_i$ involve disjoint observations. Thus one can show that conditional on $\cG_n$, the $\eta_i$'s are conditionally independent and mean zero, with conditional variances $q_i=\sigma_1^2(T_i)+\sigma_0^2(C_{\hat\sigma_n(i)})$. 
Since $|x-y|^4\le8(|x|^4+|y|^4)$ for any $x,y\in \RR$, the conditional fourth moment (and hence the third moment) is uniformly bounded.  
Lemma~\ref{lem:transfer} applied to $\sigma_0^2(\cdot)$ gives 
\$
\frac{1}{n_1}\sum_i\sigma_0^2(C_{\hat\sigma_n(i)})=\frac{1}{n_1}\sum_i\sigma_0^2(T_i) +o_P(1)
=  \EE[\sigma_0^2(T_i)] +o_P(1).
\$ 
The last identity is because conditional on the treated and control group sizes, $(T_1,\dots,T_{n_1})$ are i.i.d.~random variables. 
Thus we have $\frac{1}{n_1}\sum_{i=1}^{n_1} q_i = v_\epsilon + o_P(1)$, where recall $v_\epsilon = \EE_1\{\sigma_1^2(X)+\sigma_0^2(X)\}$. 

For the first term in \eqref{eq:ATT-decomp}, write
\[
 \frac1{\sqrt{n_1}}\sum_{i=1}^{n_1}\{\tau(T_i)-\theta\}
 =\sqrt{\frac n{n_1}}\frac1{\sqrt n}
       \sum_{k=1}^n Z_k\{\tau(X_k)-\theta\}.
\]
The summands on the right are i.i.d. and bounded, have mean zero and variance $\pi v_\tau$. The i.i.d. CLT and $n_1/n\to\pi$ give a $N(0,v_\tau)$ limit, including $v_\tau=0$. 

Finally, the middle term in \eqref{eq:ATT-decomp} obeys 
\$
\sqrt{n_1} B_{n,\hat\sigma_n}(\mu_0) = o_P(1) 
\$
by Theorem~\ref{thm:observational}, since $\mu_0$ belongs to a fixed $\mathcal F_L$. Apply Lemma~\ref{lem:conditional-CLT} on the original probability space, and using the fact that term (iii) is asymptotically normal conditional on $\cG_n$ while the other two terms are $\cG_n$-measurable, we obtain the the normal limit with variance $v_\tau+v_\epsilon$. In particular, $\hat\theta_\att\to\theta_\att$.

\paragraph{Convergence of $\hat{V}_\att$.}
Let $a_i=\tau(T_i)+b_i$, and recall that $b_i = \mu_0(T_i) - \mu_0(C_{\hat\sigma_n(i)})$. These quantities are uniformly bounded. 
Lipschitz continuity of $\mu_0$ gives
${n_1}^{-1}\sum_i b_i^2\le\norm{\nabla\mu_0}_\infty^2L_D(\hat\sigma_n)\to0$ in probability. 
Then by the Cauchy-Schwarz inequality, we have 
${n_1}^{-1}\sum_i\{a_i^2-\tau(T_i)^2\}\to0$. 
Since conditional on treatment indicators, the treated samples are i.i.d.~from the treated group distribution, we know  ${n_1}^{-1}\sum_i a_i^2\to\E\{\tau(X)^2\mid Z=1\}$.

Recall that $\eta_i = \varepsilon_{1,i} - \varepsilon_{0,\hat\sigma_n(i)}$, and $q_i = \sigma_1^2(T_i)+\sigma_0^2(C_{\hat\sigma_n(i)})$. 
Conditional on $\cG_n$, we know ${n_1}^{-1}\sum_i(\eta_i^2-q_i)$ has mean zero and variance at most $C/{n_1}$ by the fourth-moment bound. For every fixed $\epsilon>0$, we have 
\$
\PP\Big( {n_1}^{-1}\sum_i(\eta_i^2-q_i)\geq \epsilon\Big) \leq \EE\big[\min\{1,C/({n_1}\epsilon^2)\} \big] + \PP(\cI_n^c).
\$
up to the vanishing probability of $\mathcal I_n^c$. 
The right-handed side tends to zero because ${n_1}\to\infty$ in probability and the minimum is bounded by one. Similarly, ${n_1}^{-1}\sum_i a_i\eta_i$ has conditional mean zero and variance at most $C/{n_1}$, because the $a_i$ and $q_i$ are bounded, and the same averaging argument applies. Expanding $R_i^2=a_i^2+2a_i\eta_i+\eta_i^2$ gives
${n_1}^{-1}\sum_iR_i^2\to\E\{\tau(X)^2\mid Z=1\}+v_\epsilon$.
Finally,
$\hat V_\att=\{{n_1}/({n_1}-1)\}\{{n_1}^{-1}\sum_iR_i^2-\hat\theta^2\}\to v_\tau+v_\epsilon$.
Thus by Slutsky's theorem, we have the desired coverage of the Wald-CI $\hat\theta_\att\pm z_{1-\alpha/2}\sqrt{\hat V_\att/{n_1}}$. 

\paragraph{Validity of bootstrap inference.}
We now prove the validity of the matched-pair bootstrap inference. 
We first check empirical moments of the original paired differences, then analyze the bootstrap draws.

The conditional fourth moment of every $R_i$ is bounded by a common constant, from the bounded conditional means and the outcome assumption. Hence $M_{4n}={n_1}^{-1}\sum_i|R_i|^4=O_P(1)$ by Markov's inequality. 
By Jensen's inequality, we have $|\hat\theta|_\att^4\le M_{4n}$, and so ${n_1}^{-1}\sum_i |R_i-\hat\theta_\att|^4\le16M_{4n}=O_p(1)$. Write $M_{kn}^{c}={n_1}^{-1}\sum_i |R_i-\hat\theta_\att|^k$ for $k=3,4$. 
By the fourth-moment condition and Jensen's inequality, we know $M_{kn}^{c} = O_P(1)$ for $k=3,4$. The centered empirical variance is $s_n^2=\frac{1}{n_1}\sum_i(R_i-\hat\theta_\att)^2= \frac{ n_1 -1}{n_1} \hat V_\att\to V$.

Conditional on the full observed data and the original matching,
let $I_1^*,\ldots,I_{n_1}^*$ be independent and uniform on
$\{1,\ldots,n_1\}$. Let $\PP^*$ and $\EE^*$ denote the probability and expectation under the bootstrap resampling law, conditional on the data. Define
\[
W_i^*=R_{I_i^*}-\widehat\theta_\att,
\qquad
S_n^*=\frac1{\sqrt{n_1}}\sum_{i=1}^{n_1}W_i^*
=\sqrt{n_1}
(\hat\theta_{\att}^*-\hat\theta_\att).
\]
Under the conditional bootstrap resampling draw, the $W_i^*$ are i.i.d.
and satisfy
\[
\EE^*[W_b^*]=0,\qquad
\EE^*[(W_b^*)^2]=s_n^2,\qquad
\EE^*[|W_b^*|^3]=M_{3n}^c.
\]

On the event $\{n_1\ge2,\ s_n>0\}$, the Berry--Esseen
inequality \citep[Theorem~3.4.17]{Durrett2019}, applied to
this conditional distribution, gives
\[
\sup_{x\in\mathbb R}
\big|
\PP^*\{S_n^*/s_n\le x\}-\Phi(x)
\big|
\le
\frac{3M_{3n}^c}{\sqrt{n_1}s_n^3}
=o_P(1).
\]
The last equality follows from $M_{3n}^c=O_P(1)$,
$n_1\to\infty$ in probability, and $s_n^2\to V>0$
in probability. These conditions also imply that the
event on which the inequality is applied has probability
tending to one.

To replace $s_n$ by $\sqrt V$, note that for any $s,t>0$,
\[
\sup_{x\in\RR}\big|\Phi(x/s)-\Phi(x/t)\big|
\le C_\Phi\big|\log(s/t) \big|,
\qquad
C_\Phi=\sup_{y\in\mathbb R}|y\phi(y)|<\infty,
\]
where $\phi(\cdot)$ is the standard normal density function. This follows
by differentiating $\Phi(xe^{-u})$ with respect to $u$
and integrating between $\log s$ and $\log t$.
Consequently,
\$
&\sup_{x\in\mathbb R}
\left|
\PP^*\{S_n^*\le x\}-\Phi(x/\sqrt V)
\right| \le
\frac{3M_{3n}^c}{\sqrt{n_1}s_n^3}+C_\Phi|\log(s_n/\sqrt V)|
\xrightarrow{P}0.
\$
It remains to justify the random denominator used in the bootstrap
$t$-statistic. Write
\[
 \overline W_n^*  = \frac1{n_1}\sum_{i=1}^{n_1} W_i^*
 =
 \hat\theta_{\att}^*
 -\hat\theta_{\att},
 \qquad
 A_n^* = \frac1{n_1}\sum_{i=1}^{n_1}(W_i^*)^2 ,\qquad 
 \hat V_{\att}^*
 =
 \frac{n_1}{n_1-1}
 \left\{
 A_n^*-(\overline W_n^*)^2
 \right\}.
\]
Conditional on the observed data, we know 
\[
 \EE^* [A_n^*]=s_n^2,
 \qquad
 \operatorname{Var}^*(A_n^*)
 \le
 \frac{M^c_{4n}}{n_1},
 \qquad
 \mathbb E^*\{(\overline W_n^*)^2\}
 =
 \frac{s_n^2}{n_1}.
\]
Consequently, for every fixed $\varepsilon>0$,
conditional Chebyshev's and Markov's inequalities give
\[
 \PP^*
 \left(
 |A_n^*-s_n^2|>\varepsilon
 \right)
 \le
 \frac{M^c_{4n}}{n_1\varepsilon^2},
 \qquad
 \PP^*
 \left(
 (\overline W_n^*)^2>\varepsilon
 \right)
 \le
 \frac{s_n^2}{n_1\varepsilon}.
\]
Since
$M^c_{4n}=O_{P}(1)$,
$s_n^2 = V_\att+o_P(1)$, and
$n_1\to\infty$ in probability, it follows that $\PP^* ( |\hat{V}_{\att}^*-s_n^2|>\varepsilon
  )\to 0$ in probability, 
 and in particular,
\[
 \PP^*\big(
 | {\hat V_{\att}^*}/{s_n^2}-1
 |>\varepsilon \big)
 \xrightarrow{\PP}0,
\]
and therefore 
$ \PP^*(\hat{V}_{\att}^*=0)
 \xrightarrow{\PP}0$. 
Now let 
 $Y_n^*=\frac{S_n^*}{s_n}$ and 
 $R_n^*=
 \sqrt{\hat{V}_{\att}^*}/s_n$ 
on the event $\hat{V}_{\att}^*>0$.
The conditional Berry--Esseen bound above gives
\[
 \sup_x
 \left|
 \PP^*(Y_n^*\le x)-\Phi(x)
 \right|
 \xrightarrow{\PP}0,
\]
while the previous arguments give  
 $\PP^*(|R_n^*-1|>\varepsilon)
 \xrightarrow{\PP}0$ for any $\varepsilon>0$. 
Hence   Slutsky's theorem yields
\[
 \sup_{x\in\mathbb R}
 \left|
 \PP^* \{  {Y_n^*}/{R_n^*}\le x \}
 -\Phi(x)
 \right|
 \xrightarrow{\PP}0.
\]
Since $Y_n^*/R_n^*=T^{*(1)}$ whenever
$\hat{V}_{\att}^*>0$, and the complementary event has
conditional probability tending to zero in probability, we know that 
\$
 \sup_{x\in\mathbb R}
 \big|
 \PP^*\{T^{*(1)}\le x\}
 -\Phi(x)
 \big| \xrightarrow{\PP}0. 
\$

We then prove the validity of the bootstrap confidence interval. 
Fix $\alpha\in(0,1)$ and let $B=B_n\to\infty$ be a
deterministic sequence of positive integers. Write 
$\hat{\mathrm{se}}_n=\sqrt{\hat V_\att/n_1}$ and 
$T_n=\frac{\hat\theta_{\att}-\theta_{\att}}
{\hat{\mathrm{se}}_n}$. 
The sampling central limit theorem and variance consistency
proved above give 
$T_n\Rightarrow N(0,1)$. 
Let
\[
F_n^*(t)=\PP^*\{T^{*(1)}\le t\},
\qquad t\in\mathbb R.
\]
The studentized bootstrap result in the proof of Theorem~\ref{thm:ATT} gives 
$\sup_{t\in\mathbb R}|F_n^*(t)-\Phi(t)|
\xrightarrow{\PP}0$. 
Here $\PP^*$ denotes probability over the bootstrap
draws conditional on the full observed data. In what follows, the probability $\PP$ includes both the
original sampling and the
bootstrap sampling randomness. 
 
Define the empirical distribution function
\[
\hat F_{n,B_n}^*(t)
=
\frac1{B_n}\sum_{b=1}^{B_n}
\ind\{T^{*(b)}\le t\}.
\]
For each fixed $t$, conditional on the data, we know
\[
\EE^*[\hat{F}_{n,B_n}^*(t)]=F_n^*(t),
\qquad
\operatorname{Var}^*\{\hat F_{n,B_n}^*(t)\}
=
\frac{F_n^*(t)\{1-F_n^*(t)\}}{B_n}
\le\frac1{4B_n}.
\]
For every $\varepsilon>0$, conditional Chebyshev's
inequality and averaging over the original data give
\[
\begin{aligned}
\PP\left\{
|\hat F_{n,B_n}^*(t)-F_n^*(t)|>\varepsilon
\right\} = 
\EE\left[
\PP^*\left\{
|\hat F_{n,B_n}^*(t)-F_n^*(t)|>\varepsilon
\right\}
\right] \le\frac1{4B_n\varepsilon^2} \to 0.
\end{aligned}
\]
Combining this with $|F_n^*(t)-\Phi(t)|=o_P(1)$ yields
\[
\hat F_{n,B_n}^*(t)\xrightarrow{\PP}\Phi(t)
\qquad\text{for every fixed }t\in\mathbb R.
\]

For $q\in(0,1)$, recall that 
$c_q^*=\inf\{t\in\RR:\hat F_{n,B_n}^*(t)\ge q\}$,
where
$z_q=\Phi^{-1}(q)$. 
Fix $\varepsilon>0$. By the definition of the empirical
quantile,
\[
\begin{aligned}
\PP\{|c_q^*-z_q|>\varepsilon\}
\le{}&
\PP\{\hat F_{n,B_n}^*(z_q-\varepsilon)\ge q\} +
\PP\{\hat F_{n,B_n}^*(z_q+\varepsilon)<q\}.
\end{aligned}
\]
Since 
$\Phi(z_q-\varepsilon)<q<\Phi(z_q+\varepsilon)$, 
both probabilities tend to zero by the pointwise
convergence established before. Thus 
$c_q^*- z_q = o_P(1)$. 
In particular, this holds jointly for
$q=\alpha/2$ and $q=1-\alpha/2$.
Consider
\[
\mathrm{CI}^{\mathrm{boot}}_{n,1-\alpha}
=
\big[ 
\hat\theta_{\att}
-c_{1-\alpha/2}^*\widehat{\mathrm{se}}_n,
\quad
\hat\theta_{\att}
-c_{\alpha/2}^*\widehat{\mathrm{se}}_n\big].
\]
On the event $\cI_n\cap\{\hat V_\att>0\}$,
its coverage event is  
$\{
\theta_{\att}\in
\mathrm{CI}^{\mathrm{boot}}_{n,1-\alpha}
\}=\{
c_{\alpha/2}^*\le T_n\le c_{1-\alpha/2}^*\}$. 
Note that $\PP(\cI_n\cap\{\hat V_\att>0\})\to 1$. 
Taking the confidence interval to be $\RR$ on its
complement therefore gives
\[
\PP\left\{
\theta_{\att}\in
\mathrm{CI}^{\mathrm{boot}}_{n,1-\alpha}
\right\}
=
\PP\left\{
c_{\alpha/2}^*\le T_n\le c_{1-\alpha/2}^*
\right\}
+o(1).
\]

Let $z_\ell=z_{\alpha/2}$ and $z_u=z_{1-\alpha/2}$.
For $0<\delta<(z_u-z_\ell)/2$, define 
$E_{n,\delta}
=
 \{
|c_{\alpha/2}^*-z_\ell|\le\delta,~
|c_{1-\alpha/2}^*-z_u|\le\delta
 \}$. 
The quantile convergence gives
$\PP(E_{n,\delta}^c)\to0$. A decomposition gives 
\$
&\PP\{z_\ell+\delta\le T_n\le z_u-\delta\}
-\PP(E_{n,\delta}^c)\leq
\PP\{c_{\alpha/2}^*\le T_n\le c_{1-\alpha/2}^*\} \le
\PP\{z_\ell-\delta\le T_n\le z_u+\delta\}
+\PP(E_{n,\delta}^c).
\$
Using $T_n\Rightarrow N(0,1)$, the limiting lower and upper
bounds are, respectively, 
$\Phi(z_u-\delta)-\Phi(z_\ell+\delta)$ 
and
$\Phi(z_u+\delta)-\Phi(z_\ell-\delta)$. 
Letting $\delta\downarrow0$ makes both bounds equal to
$\Phi(z_u)-\Phi(z_\ell)=1-\alpha$. Hence
\[
\PP\big\{
\theta_{\att}\in
\mathrm{CI}^{\mathrm{boot}}_{n,1-\alpha}
\big\} \to 1-\alpha,
\]
and we complete the proof of Theorem~\ref{thm:ATT}.
\end{proof}

%% file: 142_wald_cdf.tex

\subsection{Proof of Theorem~\ref{thm:wald_cdf}}
\label{app:subsec_thm_wald_cdf}

\paragraph{Recap on matching theoretical properties.}
Let $\pi=\PP (Z=1)$. The law of large numbers gives
$n_1/n = \pi+o_P(1)$ for $\pi\in(0,1/2)$, so
$n_1\to\infty$ in probability and $\PP (\mathcal I_n)\to1$.
Recall
\[
B_{n,\hat\sigma_n}(g)=\frac1{n_1}\sum_{i=1}^{n_1}
                 \{g(T_i)-g(C_{\hat\sigma_n(i)})\},\qquad
L_n(\hat\sigma_n)=\frac1{n_1}\sum_{i=1}^{n_1}
                 \|T_i-C_{\hat\sigma_n(i)}\|^2.
\]
Theorem~\ref{thm:observational} applied to the fixed class
containing all $H_0(y\given\cdot)$, gives
\begin{equation}\label{cdf32:eq:beta}
\beta_n:=\sqrt{n_1}\sup_{y\in\mathbb R}
 |B_{n,\hat\sigma_n}\{H_0(y\given\cdot)\}|
 \xrightarrow{\PP }0.
\end{equation} 
Analogous to the proof of Theorem~\ref{thm:ATT}, we can show that on the event $\cI_n \cap \{\|\hat{A}-A_0\|\leq \lambda_{\min}(A_0)/2\}$ whose probability tends to one,  
$c L_n(\hat\sigma_n)
\le Q_n(\hat\sigma_n;\hat A)
\le Q_n(\bar\sigma_n;\hat A)+\xi_n
\le C r_n^2+\xi_n$ 
for fixed $c,C>0$. Consequently,
\begin{equation}\label{cdf32:eq:length}
L_n(\hat\sigma_n)=O_P(r_n^2)+o_P(n^{-1}),
\qquad
\sqrt{L_n(\hat\sigma_n)}\log n_1\xrightarrow{P}0.
\end{equation} 

Define the average distance between the conditional control CDFs by
\begin{equation}\label{cdf32:eq:delta}
\delta_n=\frac1{n_1}\sum_{i=1}^{n_1}
 \sup_{y\in\mathbb R}
 |H_0(y\given C_{\hat\sigma_n(i)})-H_0(y\given T_i)|.
\end{equation}
The gradient bound in $\mathcal F_L$, convexity of $\Omega$, and
Cauchy--Schwarz inequality give
\begin{equation}\label{cdf32:eq:delta-rate}
0\le\delta_n\le1,\qquad
\delta_n\le\sqrt d L\frac1{n_1}\sum_i
  \|C_{\hat\sigma_n(i)}-T_i\|
 \le\sqrt d L\sqrt{L_n(\hat\sigma_n)},
\qquad \delta_n\log n_1\xrightarrow{P}0.
\end{equation}
The probabilities are over the original data sampling randomness.

We define the \emph{generalized inverse function} of a function $F$ by 
\$
F^{\leftarrow}(u) := \inf\{x\in \RR\colon F(x)\geq u\},\quad 0<u<1. 
\$
We will use two technical lemmas in the proof, whose proofs are in Appendix~\ref{cdf32:app:comparisons}.

\begin{lemma}
\label{cdf32:lem:crossing}
Fix $m\ge2$ and CDFs $F_i,G_i$, $i=1,\ldots,m$. Let $V_i$ be
independent uniform variables on $(0,1)$, and define
\[
A_i=F_i^{\leftarrow}(V_i),\qquad B_i=G_i^{\leftarrow}(V_i),\qquad
e_i(y)=\ind\{A_i\le y\}-\ind\{B_i\le y\},\qquad
\delta=\frac1m\sum_i\|F_i-G_i\|_\infty.
\]
For each $r>0$, there exists a constant $C_r$ independent of $m$ and of these CDFs  
such that, with probability at least $1-C_r m^{-r}$, both inequalities hold:
\begin{align}
\sup_y\left|\frac1{\sqrt m}\sum_i\{e_i(y)-\mathbb Ee_i(y)\}\right|
&\le C_r\left\{\sqrt{\delta\log m}+\frac{\log m}{\sqrt m}\right\},
\label{cdf32:eq:crossing-centered}\\
\sup_y\frac1m\sum_i e_i(y)^2
&\le\delta+C_r\left\{\sqrt{\frac{\delta\log m}{m}}
                          +\frac{\log m}{m}\right\}.
\label{cdf32:eq:crossing-square}
\end{align} 
\end{lemma}

\begin{lemma} 
\label{cdf32:lem:resample}
Fix any $m\ge2$ pairs of real numbers $(A_i,B_i)$ and define $e_i(y)$ as
in Lemma~\ref{cdf32:lem:crossing}. Let 
\[
\bar e_m(y)=\frac1m\sum_i e_i(y),\qquad
\kappa_m=\sup_y\frac1m\sum_i e_i(y)^2.
\]
For independent uniform indices $I_1^*,\ldots,I_m^*$ on $\{1,\ldots,m\}$,
with conditional probability at least $1-C_r m^{-r}$,
\begin{equation}\label{cdf32:eq:resampling-bound}
\sup_y\left|\frac1{\sqrt m}\sum_{i=1}^m
                   \{e_{I_i^*}(y)-\bar e_m(y)\}\right|
\le C_r\left\{\sqrt{\kappa_m\log m}+\frac{\log m}{\sqrt m}\right\}.
\end{equation} 
\end{lemma}

\paragraph{Joint distribution of $\hat{F}_1$ and $\hat{F}_0$.} In the following, we shall construct auxiliary samples analyze the distributional differences. 
Keep the original observed data unchanged. Extend the probability
space by independent variables $W_i\sim\operatorname{Unif}(0,1)$,
independent of the full observed data and the matching randomization.
For each matched pair, write
\[
F_i(y)=H_0(y\given C_{\hat\sigma_n(i)}),\quad F_i(y-)=\lim_{t\uparrow y}F_i(t)
\]
and define
\[
\begin{aligned}
V_{0i}
={}&F_i(Y_{0,\hat\sigma_n(i)}-) +W_i\bigl\{
F_i(Y_{0,\hat\sigma_n(i)})
-F_i(Y_{0,\hat\sigma_n(i)}-)
\bigr\}.
\end{aligned}
\] 
We then know that  
$V_{0i}\given\mathcal G_n\sim\operatorname{Unif}(0,1)$ and 
$F_i^{\leftarrow}(V_{0i})=Y_{0,\hat\sigma_n(i)}$ almost surely. 
Conditional on $\mathcal G_n$, the selected control outcomes
are mutually independent and independent of the treated outcomes.
The same is therefore true of the $V_{0i}$ and the treated
outcomes, since each $V_{0i}$ is a function of its corresponding
control outcome and the independent variable $W_i$.
Define the auxiliary outcomes by
\[
Y_{1,i}^{\dagger}=Y_{1,i},
\qquad
Y_{0,i}^{\dagger}=H_0^{\leftarrow}(V_{0i}\given T_i).
\]
Conditional on $\mathcal G_n$, the auxiliary pairs are independent,
and pair $i$ has distribution
$H_1(\cdot\given T_i)\otimes H_0(\cdot\given T_i)$.
Consequently, given $n_1$, these pairs are i.i.d. with law
$P^\dagger$, where $P^\dagger$ is the law obtained by drawing $X\sim P_1$ and draw $Y_1^\dagger\sim H_1(\cdot\given X)$ and $Y_0^\dagger\sim H_0(\cdot\given X)$. 
Their joint distribution depends on the group sizes
only through $n_1$. Define
\[
\hat F_z^{\dagger}(y)=\frac1{n_1}\sum_{i=1}^{n_1}
                        \ind\{Y_{z,i}^{\dagger}\le y\}.
\]
Define 
\[
e_i(y)=\ind\{Y_{0,\hat\sigma_n(i)}\le y\}
             -\ind\{Y_{0,i}^{\dagger}\le y\},\qquad
a_i(y)=\mathbb E\{e_i(y)\given\mathcal G_n\}.
\]
We know 
$a_i(y)=H_0(y\given C_{\hat\sigma_n(i)})-H_0(y\given T_i)$ and 
$\mathbb E\{e_i(y)^2\given\mathcal G_n\}=|a_i(y)|$. 
Hence for every $y$, 
\begin{equation}\label{cdf32:eq:exact-decomposition}
\sqrt{n_1}(\hat F_0-\hat F_0^{\dagger})(y)
=\frac1{\sqrt{n_1}}\sum_i\{e_i(y)-a_i(y)\}
 -\sqrt{n_1}B_{n,\hat\sigma_n}\{H_0(y\given\cdot)\}.
\end{equation}
Apply Lemma~\ref{cdf32:lem:crossing} conditionally on $\mathcal G_n$, with
$m=n_1$, $F_i=H_0(\cdot\given C_{\hat\sigma_n(i)})$, and
$G_i=H_0(\cdot\given T_i)$. For a fixed constant $C$, set
\[
b_n=C\left\{\sqrt{\delta_n\log n_1}
                      +\frac{\log n_1}{\sqrt{n_1}}\right\} =o_P(1).
\]
Taking $r=2$ in the lemma, for every $\varepsilon>0$,
\begin{align*}
\PP \left\{
\sup_y\bigg|\frac1{\sqrt{n_1}}\sum_i(e_i(y)-a_i(y))\bigg|
 >\varepsilon,\ \mathcal I_n\right\}
\le \PP \{b_n>\varepsilon,\mathcal I_n\}
 +\mathbb E[\ind_{\mathcal I_n}\min\{1,Cn_1^{-2}\}]
\to 0.
\end{align*}
The last expectation tends to zero because its integrand is bounded and
$n_1\to\infty$ in probability. Combining this with \eqref{cdf32:eq:beta}
and \eqref{cdf32:eq:exact-decomposition}, and using
$\hat F_1=\hat F_1^{\dagger}$, yields
\begin{equation}\label{cdf32:eq:sampling-comparison}
\sqrt{n_1}\max_{z\in\{0,1\}}\|\hat F_z-\hat F_z^{\dagger}\|_\infty
\xrightarrow{P}0.
\end{equation}

We now apply a standard i.i.d.~empirical-process result only to the auxiliary
pairs. On $\mathbb R^2$ consider
\[
\mathcal H=\{h_{z,y}:z\in\{0,1\},\ y\in\mathbb R\},\qquad
h_{z,y}(w_1,w_0)=\ind\{w_z\le y\}.
\]
We first consider a deterministic sample size $j$.
Let 
$W_i^\dagger=(Y_{1,i}^\dagger,Y_{0,i}^\dagger)$. For a measurable function $h$,
write 
$P^\dagger h=\int h\,dP^\dagger$ and 
$\hat P_j^\dagger h=\frac1j\sum_{i=1}^j h(W_i^\dagger)$. Consider the empirical process
\[
\mathbb G_j^\dagger(h)
=
\sqrt j\,(\hat P_j^\dagger h-P^\dagger h),
\qquad h\in\mathcal H.
\]
Since the $z$th marginal of $P^\dagger$ has CDF $F_{z,T}$, we know  
$P^\dagger h_{z,y}=F_{z,T}(y)$ and 
$\mathbb G_j^\dagger(h_{z,y}) 
=
\sqrt j\{\hat F_{z,j}^\dagger(y)-F_{z,T}(y)\}$, 
where 
$\hat F_{z,j}^\dagger(y)
=
\frac1j\sum_{i=1}^j\ind\{Y_{z,i}^\dagger\le y\}$.  

We verify the conditions for the standard VC-class
Donsker theorem. For each fixed $z$, the sets
$\{w:w_z\le y\}$ are nested as $y$ increases. On any
$r$ points in $\mathbb R^2$, varying $y$ produces at most
$r+1$ distinct indicator vectors for each $z$, and hence
at most $2(r+1)$ for the full class $\mathcal H$.
Since $2(r+1)<2^r$ for $r\ge4$, the associated class of
sets has finite VC dimension. Its indicator class
$\mathcal H$ is therefore a VC-subgraph class.
Every member of $\mathcal H$ is measurable and bounded
by the constant $1$, which is square integrable
under $P^\dagger$. The class is also pointwise measurable:
for every $y$, choose rational numbers $q_l\downarrow y$.
Then 
$h_{z,q_l}(w)\to h_{z,y}(w)$
for every $w\in\mathbb R^2$. 
Thus the countable subclass indexed by rational thresholds
provides pointwise approximations to every member of
$\mathcal H$.
These properties imply that $\mathcal H$ is
$P^\dagger$-Donsker; see
\citet[Theorem~2.5.2 and Section~2.6]{VW1996}.
Consequently,
\[
\mathbb G_j^\dagger
\Rightarrow \mathbb G_{P^\dagger}
\qquad\text{in }\ell^\infty(\mathcal H),
\]
where $\mathbb G_{P^\dagger}$ is the centered 
Gaussian process with  
$\operatorname{Cov}\{
\mathbb G_{P^\dagger}(h),\mathbb G_{P^\dagger}(g)
\}
=
P^\dagger(hg)-(P^\dagger h)(P^\dagger g)$ 
for $h,g\in\mathcal H$. 
Thus we have  
\[
\sqrt j\bigl(
\hat F_{1,j}^\dagger-F_{1,T},
\hat F_{0,j}^\dagger-F_{0,T}
\bigr)
\Rightarrow
(\mathbb B_1,\mathbb B_0)
\qquad\text{in }\ell^\infty(\mathbb R)^2,\quad \mathbb B_z(y)=\mathbb G_{P^\dagger}(h_{z,y}),
\qquad z\in\{0,1\}.
\] 

To pass from deterministic $j$ to $n_1$, one may realize an independent
infinite sequence of $P^{\dagger}$ pairs and preserve its first $n_1$ members:
conditional on $n_1$, the previous arguments on the auxiliary samples still apply. 
Write $Z_j^{\dagger}$ for the size-$j$ auxiliary process.
For a bounded continuous test functional $\varphi$, let $a_j$ be the difference
of its expectations under $Z_j^{\dagger}$ and the limit (using outer
expectations under the outer probability $\PP ^{\mathrm{out}}$ when necessary). Then $a_j\to0$ and,
for any fixed integer $M$, we know 
$|\EE[ a_{n_1}]|
\le\sup_{j\ge M}|a_j| +2\|\varphi\|_\infty\PP (n_1<M)\to 0$ 
by first taking $n\to\infty$ and then $M\to\infty$.
The tightness bounds pass through the same mixture. In fact, for any compact
set $K$ in the process space and its open $\delta$-neighborhood $K^\delta$,
\[
\PP ^{\mathrm{out}}\{Z_{n_1}^{\dagger}\notin K^\delta\}
\le\PP (n_1<M)+
       \sup_{j\ge M}\PP ^{\mathrm{out}}\{Z_j^{\dagger}\notin K^\delta\}.
\]
Here $\PP ^{\mathrm{out}}$ denotes outer probability. The asymptotic
tightness of the deterministic-size processes therefore transfers by taking
$n\to\infty$ and then $M\to\infty$.
Thus the auxiliary process limit holds also with the random sample size.
Together with \eqref{cdf32:eq:sampling-comparison}, this proves~\eqref{dist:eq:CDFCLT}. 

\paragraph{Point-wise Wald interval.} We then show the point-wise validity of the Wald interval. 
Fix $y\in\mathbb R$. Define  
\[
R_i^{\dagger}(y)=\ind\{Y_{1,i}^{\dagger}\le y\}
                     -\ind\{Y_{0,i}^{\dagger}\le y\},\qquad
\kappa_n=\sup_t\frac1{n_1}\sum_i e_i(t)^2.
\]
By \eqref{cdf32:eq:crossing-square}, conditionally and then unconditionally, we have 
\begin{equation}\label{cdf32:eq:kappa}
\kappa_n\le\delta_n+C\left\{\sqrt{\frac{\delta_n\log n_1}{n_1}}
                    +\frac{\log n_1}{n_1}\right\}
\end{equation}
with probability tending to one. In particular $\kappa_n\to_{\PP }0$.
Since $R_i(y)-R_i^{\dagger}(y)=-e_i(y)$, all these differences are bounded,
and $|e_i(y)|=e_i(y)^2$,
\@
\left|\frac1{n_1}\sum_i R_i(y)-\frac1{n_1}\sum_iR_i^{\dagger}(y)\right|
\le\kappa_n, \quad 
\left|\frac1{n_1}\sum_i R_i(y)^2
                 -\frac1{n_1}\sum_iR_i^{\dagger}(y)^2\right|
&\le2\kappa_n.
\label{cdf32:eq:moment-comparison}
\@
Given $n_1$, the $R_i^{\dagger}(y)$ are i.i.d. and bounded by one. Their
sample first and second moments converge to the corresponding $P^{\dagger}$
moments. For example, conditional Chebyshev bounds either centered sample
moment's exceedance probability by $1/(n_1\varepsilon^2)$; averaging
$\min\{1,1/(n_1\varepsilon^2)\}$ proves the result for random $n_1$.
Their mean is $D(y)$ and their variance is $V_D(y)$ defined earlier. 
It follows from \eqref{cdf32:eq:moment-comparison} that
\begin{equation}\label{cdf32:eq:variance-consistency}
\hat V_D(y)=\frac{n_1}{n_1-1}
 \left\{\frac1{n_1}\sum_iR_i(y)^2-\hat D(y)^2\right\}
 \xrightarrow{ P}V_D(y).
\end{equation}
Combined with~\eqref{dist:eq:CDFCLT}, this gives 
$\sqrt{n_1}\{\hat D(y)-D(y)\}\Rightarrow N(0,V_D(y))$.  
If $V_D(y)>0$, Slutsky's theorem and \eqref{cdf32:eq:variance-consistency}
yield the asymptotic normality of  
$\frac{\sqrt{n_1}\{\hat D(y)-D(y)\}}{\sqrt{\hat V_D(y)}} 
\Rightarrow N(0,1)$. This proves the validity of the Wald-type interval in Theorem~\ref{thm:wald_cdf}.

\paragraph{Bootstrap inference validity.}
Fix $\alpha\in(0,1)$ and let $B_n\to\infty$. For
$H\in\{F,D\}$, write
\[
G_H(t)=\PP (W_H\le t),\qquad t\in\mathbb R,
\]
where $W_F,W_D$ are defined in Theorem~\ref{thm:wald_cdf}. We establish each
band's coverage under its corresponding regular-quantile condition.
All matching-based quantities below are considered on $\mathcal I_n$.
On $\mathcal I_n^c$, take the CDF bands to be $[0,1]$ and the
contrast band to be $\mathbb R$; this convention affects coverage
by at most $\PP (\mathcal I_n^c)=o(1)$.

We use the original and auxiliary outcomes already constructed.
In the comparison below, $\PP ^*$ denotes probability over the
resampling indices with both sets of outcomes held fixed.  
Use the same indices $I_i^{*(1)}$ for the two sets of outcome pairs and define
\[
\hat F_z^{\dagger *(1)}(y)
=\frac1{n_1}\sum_{i=1}^{n_1}
\ind\{Y_{z,I_i^{*(1)}}^\dagger\le y\},
\qquad
\hat D^\dagger=\hat F_1^\dagger-\hat F_0^\dagger,
\qquad
\hat D^{\dagger *(1)}
=\hat F_1^{\dagger *(1)}-\hat F_0^{\dagger *(1)}.
\]
The auxiliary counterparts of $T_F^{*(1)}$ and $T_D^{*(1)}$ are
\$
T_F^{\dagger *(1)}
 =\sqrt{n_1}\max_{z\in\{0,1\}}\sup_{y\in\mathbb R}
  |\hat F_z^{\dagger *(1)}(y)-\hat F_z^\dagger(y)|, \quad
T_D^{\dagger *(1)}
 = \sqrt{n_1}\sup_{y\in\mathbb R}
  |\hat D^{\dagger *(1)}(y)-\hat D^\dagger(y)|.
\$
Recall $e_i(y)=\ind\{Y_{0,\hat\sigma_n(i)}\le y\}
-\ind\{Y_{0,i}^\dagger\le y\}$ and let
$\bar e_{n_1}(y)=n_1^{-1}\sum_i e_i(y)$. The centered control-CDF
difference has the exact expression
\$
E_n^*(y)
&:=\sqrt{n_1}\left\{
(\hat F_0^{*(1)}-\hat F_0)(y)
-(\hat F_0^{\dagger *(1)}-\hat F_0^\dagger)(y)
\right\} =\frac1{\sqrt{n_1}}\sum_{i=1}^{n_1}
\{e_{I_i^{*(1)}}(y)-\bar e_{n_1}(y)\}.
\$
The treated CDFs agree both before and after resampling 
as $Y_{1,i}^\dagger=Y_{1,i}$. The  triangle inequality thus gives 
\@\label{eq:cdfboot:statistics-close}
|T_H^{*(1)}-T_H^{\dagger *(1)}|
\le\|E_n^*\|_\infty, \quad H\in \{F,D\}.
\@

The preceding bound for $\kappa_n$ implies, with probability tending
to one,
\[
\kappa_n\log n_1
\le\delta_n\log n_1
+C\sqrt{\frac{\delta_n(\log n_1)^3}{n_1}}
+C\frac{(\log n_1)^2}{n_1}
=o_P(1).
\]
Here we used $\delta_n\log n_1\to_{\PP }0$,
$0\le\delta_n\le1$, and $n_1\to\infty$ in probability.
Applying Lemma~\ref{cdf32:lem:resample} with $m=n_1$ and $r=2$, define
\[
b_n^{\mathrm{boot}}
=C\left\{\sqrt{\kappa_n\log n_1}
             +\frac{\log n_1}{\sqrt{n_1}}\right\}
=o_{ P}(1).
\]
For every fixed $\eta>0$, that lemma gives, on $\mathcal I_n$,
\begin{equation}
p_n^*(\eta):=\PP ^*(\|E_n^*\|_\infty>\eta)
\le\ind\{b_n^{\mathrm{boot}}>\eta\}
       +\min\{1,Cn_1^{-2}\} \xrightarrow{P} 0.
\label{eq:cdfboot:conditional-comparison}
\end{equation}
The probability in this convergence includes the auxiliary
randomization.

We now approximate the distributions used to calibrate the confidence bands.
First consider a deterministic number $j$ of i.i.d. auxiliary pairs
with law $P^\dagger$, and add a subscript $j$ to the preceding
auxiliary bootstrap statistics. The class $\mathcal H$ was shown above
to be pointwise measurable, uniformly bounded, and $P^\dagger$-Donsker.
The empirical-bootstrap theorem
\citep[Section~3.6.1]{VW1996} therefore applies to its centered
bootstrap empirical process.
The maps 
$(f_1,f_0)\longmapsto\max_z\|f_z\|_\infty $ and 
$(f_1,f_0)\longmapsto\|f_1-f_0\|_\infty$ 
are Lipschitz, with constants one and two under the sup-norm. Applying the continuous mapping theorem
to these two maps yields, for every continuity point $t$ of $G_H$,
\begin{equation}
G_{H,j}^{\dagger *}(t)
:=\PP ^*\{T_{H,j}^{\dagger *}\le t\}
\xrightarrow{\PP }G_H(t),\qquad j\to\infty.
\label{eq:cdfboot:auxiliary-fixed}
\end{equation}
Note that this argument concerns only the i.i.d.~auxiliary pairs.  

The conclusion remains valid at the random sample size $n_1$.
To see this explicitly, fix $H\in\{F,D\}$, a continuity point
$t$ of $G_H$, and $\varepsilon>0$. Let $r_j$ be the probability,
for an auxiliary sample of size $j$, that
$|G_{H,j}^{\dagger *}(t)-G_H(t)|>\varepsilon$.
Then $r_j\to0$. Conditional on $n_1=j$ on the admissible count range,
the auxiliary sample has law $(P^\dagger)^{\otimes j}$. Hence, writing
$G_{H,n}^{\dagger *}(t)=\PP ^*\{T_H^{\dagger *(1)}\le t\}$,
for every fixed integer $M\ge2$,
\[
\PP \{|G_{H,n}^{\dagger *}(t)-G_H(t)|>\varepsilon\}
\le\PP (\mathcal I_n^c)+\PP (n_1<M)
       +\sup_{j\ge M}r_j.
\]
Letting first $n\to\infty$ and then $M\to\infty$ proves
$G_{H,n}^{\dagger *}(t)\to_{\PP }G_H(t)$.

For the actual bootstrap statistic, put 
$G_{H,n}^*(t)=\PP ^*\{T_H^{*(1)}\le t\}$. 
By \eqref{eq:cdfboot:statistics-close}, for every $\eta>0$,
\begin{equation}
G_{H,n}^{\dagger *}(t-\eta)-p_n^*(\eta)
\le G_{H,n}^*(t)
\le G_{H,n}^{\dagger *}(t+\eta)+p_n^*(\eta).
\label{eq:cdfboot:cdf-comparison}
\end{equation}
For a continuity point $t$ of $G_H$, choose $\eta\downarrow0$
so that $t-\eta$ and $t+\eta$ are also continuity points.
This is possible because a CDF has at most countably many
discontinuous points. For fixed such $\eta$, the two outer terms
in \eqref{eq:cdfboot:cdf-comparison} converge in probability to
$G_H(t-\eta)$ and $G_H(t+\eta)$, respectively.
Both constants tend to $G_H(t)$ as $\eta\downarrow0$.
It follows that
\begin{equation}
G_{H,n}^*(t)\xrightarrow{\PP }G_H(t)
\quad\text{at every continuity point of }G_H.
\label{eq:cdfboot:actual-cdf}
\end{equation}
Since $G_{H,n}^*(t)$ depends only on the observed data,
this also holds under the original sampling law.

We now show the consistency of the simulated critical values. 
Define
\[
\hat G_{H,n}^*(t)
=\frac1{B_n}\sum_{b=1}^{B_n}\ind\{T_H^{*(b)}\le t\}.
\]
Conditional on the observed data, these indicators are independent
Bernoulli variables with common mean $G_{H,n}^*(t)$. Therefore,
\[
\mathbb E^*\hat G_{H,n}^*(t)=G_{H,n}^*(t),\qquad
\operatorname{Var}^*\{\hat G_{H,n}^*(t)\}
=\frac{G_{H,n}^*(t)\{1-G_{H,n}^*(t)\}}{B_n}
\le\frac1{4B_n}.
\]
Chebyshev's inequality followed by averaging gives, for every
$\varepsilon>0$,
\[
\PP \{|\hat G_{H,n}^*(t)-G_{H,n}^*(t)|>\varepsilon\}
\le\frac1{4B_n\varepsilon^2}\to 0.
\]
Together with \eqref{eq:cdfboot:actual-cdf}, this proves
$\hat G_{H,n}^*(t)\to_{\PP }G_H(t)$ at every continuity
point of $G_H$.

Suppose $W_H$ has a regular $(1-\alpha)$-quantile $c_H$.
Thus $G_H$ is continuous at $c_H$, $G_H(c_H)=1-\alpha$,
and its values on either side strictly bracket $1-\alpha$.
For any $\varepsilon>0$, choose continuity points $t_-,t_+$ with 
$c_H-\varepsilon<t_-<c_H<t_+<c_H+\varepsilon$. 
Regularity gives $G_H(t_-)<1-\alpha<G_H(t_+)$.
By the definition of 
$c_H^*$, we have  
$\PP \{|c_H^*-c_H|>\varepsilon\}
 \le\PP \{\hat G_{H,n}^*(t_-)\ge1-\alpha\}
 +\PP \{\hat G_{H,n}^*(t_+)<1-\alpha\} \to 0$. 
Consequently,
\begin{equation}
c_H^*\xrightarrow{\PP }c_H.
\label{eq:cdfboot:critical-value}
\end{equation}
Here and below, probabilities involving $c_H^*$ include the randomness in the bootstrap draws.  

We proceed to prove the coverage of the confidence bands.
Define the original-sample statistics
\[
T_{F,n}=\sqrt{n_1}\max_{z\in\{0,1\}}\sup_{y\in\mathbb R}
                 |\hat F_z(y)-F_{z,T}(y)|,
\qquad
T_{D,n}=\sqrt{n_1}\sup_{y\in\mathbb R}|\hat D(y)-D(y)|.
\]
The existing arguments give $T_{H,n}\Rightarrow W_H$.
For any $\eta>0$, on the event $\{|c_H^*-c_H|\le\eta\}$, we have  
\begin{align*}
\PP (T_{H,n}\le c_H-\eta)
 -\PP (|c_H^*-c_H|>\eta)
&\le\PP (T_{H,n}\le c_H^*) \le\PP (T_{H,n}\le c_H+\eta)
 +\PP (|c_H^*-c_H|>\eta).
\end{align*}
Choose $\eta\downarrow0$ with $c_H\pm\eta$ continuity points of
$G_H$. For fixed such $\eta$, take lower and upper limits as
$n\to\infty$, using $T_{H,n}\Rightarrow W_H$ and
\eqref{eq:cdfboot:critical-value}. The resulting bounds are
$G_H(c_H-\eta)$ and $G_H(c_H+\eta)$. Letting $\eta\downarrow0$ yields 
\begin{equation}
\PP (T_{H,n}\le c_H^*)\to G_H(c_H)=1-\alpha.
\label{eq:cdfboot:calibrated-coverage}
\end{equation} 

Since $F_{z,T}(y)\in[0,1]$, intersecting the CDF bands with
$[0,1]$ does not change their coverage event. On $\mathcal I_n$,
\begin{align*}
&\{F_{z,T}(y)\in\mathrm{CI}^{\mathrm{joint}}_{z,n,1-\alpha}
       \text{ for every }y\in\mathbb R,\ z\in\{0,1\}\}
       =\{T_{F,n}\le c_F^*\},\\
&\{D(y)\in\mathrm{CI}^{\mathrm{contrast}}_{n,1-\alpha}(y)
       \text{ for every }y\in\mathbb R\}
       =\{T_{D,n}\le c_D^*\}.
\end{align*}
Applying \eqref{eq:cdfboot:calibrated-coverage} with $H=F$ and
$H=D$, respectively, proves the two bootstrap coverage conclusions.  We thus complete the proof of Theorem~\ref{thm:wald_cdf}.

%% file: 143_wald_quantile.tex

\subsection{Proof of Theorem~\ref{thm:quantile}}
\label{app:subsec_proof_quantile}

Throughout this proof, all quantities involving the matching are considered on the event 
$\mathcal I_n=\{2\le n_1\le n_0\}$. On its complement, define the
estimation-error statistics to be zero and the confidence sets to be
$\mathbb R$; this changes the probabilities by $o(1)$. 
The notation $\mathbb P^*$ denotes the resampling law conditional on the
observed data. For a random quantity involving a bootstrap replicate,
$\mathbb P$ includes both original sampling and bootstrap sampling randomness.

\paragraph{Empirical and bootstrap CDF processes and a general inversion.}
For $z\in\{0,1\}$, define
\[
\mathbb A_{z,n}(y)=\sqrt{n_1}\{\hat F_z(y)-F_{z,T}(y)\},\qquad
\mathbb A_{z,n}^*(y)=\sqrt{n_1}\{\hat F_z^*(y)-\hat F_z(y)\},
\]
where the empirical CDFs $\hat{F}_z(y)$ and bootstrapped CDFs $\hat{F}_z^*(y)$ are defined as in Section~\ref{subsec:wald_cdf}. 
For a bounded function $g$ on an interval $J$, write
\[
\|g\|_J=\sup_{y\in J}|g(y)|,\qquad
\omega_J(g,\delta)=
\sup_{s,t\in J,\,|s-t|<\delta}|g(s)-g(t)|,\quad\delta>0.
\]
Assumption~\ref{assump:quantile} implies that $F_{z,T}$ is continuous on $J_z$. 
Recall the Gaussian processes $\BB_z$ defined in Section~\ref{subsec:wald_cdf}. 
The marginal process $\mathbb B_z$ has the law of a standard
Brownian bridge evaluated at $F_{z,T}$.
Consequently, the joint limit can be chosen to have continuous paths on
$J_1$ and $J_0$, respectively. Theorem~\ref{thm:wald_cdf} therefore gives
\begin{equation}\label{eq:q33:cdf-sampling-input}
 \max_z\|\mathbb A_{z,n}\|_{J_z}=O_P(1),\qquad
 \lim_{\delta\downarrow0}\limsup_{n\to\infty}
 \mathbb P\{\omega_{J_z}(\mathbb A_{z,n},\delta)>\varepsilon\}=0
 \quad(\varepsilon>0).
\end{equation} 

Let $\mathbb A_{z,n}^{\dagger *}=\sqrt{n_1}
(\hat F_z^{\dagger *}-\hat F_z^\dagger)$ be the
centered auxiliary bootstrap process from Appendix~\ref{app:subsec_thm_wald_cdf} which gives 
\begin{equation}\label{eq:q33:cdf-bootstrap-comparison}
 \max_z\|\mathbb A_{z,n}^*-\mathbb A_{z,n}^{\dagger *}\|_{\mathbb R}
 =\|E_n^*\|_\infty,\qquad
 \mathbb P^*(\|E_n^*\|_\infty>\eta)
 \xrightarrow{\mathbb P}0\quad(\eta>0).
\end{equation}
The treated difference is zero. For a bounded Lipschitz functional $\varphi$
with $\|\varphi\|_\infty\le1$ and Lipschitz constant $L_\varphi$, the
corresponding conditional expectation difference is at most 
 $L_\varphi\eta+2\mathbb P^*(\|E_n^*\|_\infty>\eta)$. 
Thus the auxiliary empirical-bootstrap theorem already used in
Appendix~\ref{app:subsec_thm_wald_cdf}, \citep[Section~3.6.1]{VW1996}, including the same random-$n_1$
argument, applies to bounded Lipschitz functions of the actual process involving original samples 
as well. In particular, the norm has Lipschitz constant one and
$g\mapsto\omega_{J_z}(g,\delta)$ has Lipschitz constant two.
Applying this fact to
$g\mapsto\min\{1,\omega_{J_z}(g,\delta)/\varepsilon\}$ gives
\[
 \limsup_{n\to\infty}
 \mathbb P\{\omega_{J_z}(\mathbb A_{z,n}^*,\delta)>\varepsilon\}
 \le
 \mathbb E\min\{1,\omega_{J_z}(\mathbb B_z,\delta)/\varepsilon\}.
\]
The right side tends to zero as $\delta\downarrow0$ by path continuity
and bounded convergence. Tightness of the norms follows in the same way
by using $g\mapsto\min\{1,(\|g\|_{J_z}-M)_+\}$ and letting $M\to\infty$.
Together with \eqref{eq:q33:cdf-sampling-input}, these observations imply
that each of
\begin{equation}\label{eq:q33:two-processes}
 \mathbb A_{z,n}
 \quad\text{and}\quad
 \mathbb A_{z,n}+\mathbb A_{z,n}^*
 =\sqrt{n_1}(\hat F_z^*-F_{z,T})
\end{equation}
has a tight norm on $J_z$ and satisfies~\eqref{eq:q33:cdf-sampling-input}.

We now study the inversion of the CDF process to prepare for studying the quantiles. Fix an arm $z$ and, within this paragraph only, we write $F=F_{z,T}$, $f=f_{z,T}$, $Q=Q_z$, and $J=J_z=[a_z,b_z]$.
Let $c_z=\inf_{y\in J}f(y)>0$.
Consider a sequence of CDFs $\{H_n\}$. Later on we shall apply 
the arguments below  to both choices in
\eqref{eq:q33:two-processes} for $H_n = \hat{F}_z$ and $\hat{F}_z^*$ to prove the asymptotics of~\eqref{dist:eq:Qlimit} and the bootstrap inference, respectively.   
We assume $H_n$ satisfies the following properties:
\begin{itemize}\setlength\itemsep{-0.25em}
    \item Tight norm, i.e., $\|E_n\|_J = O_P(1)$;
    \item Modulus condition of the error, i.e., $\lim_{\delta\downarrow 0}\limsup_{n\to \infty} \PP(\omega_J(E_n,\delta)>\varepsilon)=0$ for any $\varepsilon>0$. 
\end{itemize}
The preceding paragraph has established these conditions
for both $H_n=\hat F_z$ and $H_n=\hat F_z^*$. 
We now show that these conditions imply the uniform
inverse expansion
\[
\sup_{u\in\mathcal U}
\left|
\sqrt{n_1}\{H_n^{\leftarrow}(u)-Q(u)\}
+\frac{E_n\{Q(u)\}}{f\{Q(u)\}}
\right|
=o_P(1).
\]

The set $Q(\mathcal U)$ is a compact subset of the interior of $J$,
and $F\{Q(u)\}=u$ on $\mathcal U$. Let
$e_n=\|H_n-F\|_J=O_{P}(n_1^{-1/2})$ and
$\rho_n=2e_n/c_z$. For all sufficiently small $e_n>0$, uniformly in
$u\in\mathcal U$, the points $Q(u)\pm \rho_n$ belong to $J$ and 
 $H_n\{Q(u)-\rho_n\}\le u-e_n<u$, 
 $H_n\{Q(u)+\rho_n\}\ge u+e_n>u$. 
Thus $\tilde Q_n=H_n^{\leftarrow}$ lies in the interior of $J$ on
$\cU$, and 
 $\|\tilde Q_n-Q\|_{\mathcal U}
 =O_{P}(n_1^{-1/2})$. 
If $e_n=0$, the inverses agree and the remainder
is zero.

We next control the jumps of $H_n$. Since $F$ is continuous on $J$, for every $t$ in its interior and every
$\delta>0$,
\[
 \sqrt{n_1}\{H_n(t)-H_n(t-)\}
 =E_n(t)-E_n(t-)\le\omega_J(E_n,\delta).
\]
The inequality follows by taking points in $J$ increasing to $t$.
Thus, if
$j_n=\sup_{t\in J^\circ}\{H_n(t)-H_n(t-)\}$, then
\begin{equation}\label{eq:q33:small-jumps}
 \sqrt{n_1}\,j_n=o_{P}(1).
\end{equation}
Indeed, for each fixed $\delta>0$, its exceedance probability is bounded
by that of $\omega_J(E_n,\delta)$; take a limit superior in $n$ and then
let $\delta\downarrow0$.

For $q=Q(u)$ and $\tilde q=\tilde Q_n(u)$, the generalized-inverse
inequalities give 
 $H_n(\tilde q-)\le u\le H_n(\tilde q)$, 
 $0\le H_n(\tilde q)-u\le j_n$. 
By Taylor expansion, 
\[
 F(\tilde q)-F(q)=f(q)(\tilde q-q)+R_n(u),\qquad
 |R_n(u)|\le \rho_n\,\omega_J(f,\rho_n).
\]
Combining these relations with $F(q)=u$ yields 
\[
 \sqrt{n_1}f(q)(\tilde q-q)+E_n(q)
 =\sqrt{n_1}\{H_n(\tilde q)-u\}
   -\{E_n(\tilde q)-E_n(q)\}-\sqrt{n_1}R_n(u).
\]
Since $|\tilde q-q|\le \rho_n$, we may bound the increment by
$\omega_J(E_n,2\rho_n)$ when $\rho_n>0$. Hence
\@
 &\sup_{u\in\mathcal U}
 \left|\sqrt{n_1}\{\tilde Q_n(u)-Q(u)\}
        +\frac{E_n\{Q(u)\}}{f\{Q(u)\}}\right|
        \le\frac1{c_z}\left\{
 \sqrt{n_1}j_n+\omega_J(E_n,2\rho_n)
       +\sqrt{n_1}\rho_n\,\omega_J(f,2\rho_n)\right\}
 =o_{P}(1).
 \label{eq:q33:inverse-expansion}
\@
The first term is negligible by \eqref{eq:q33:small-jumps}.
For the second, $\rho_n=o_P(1)$ and the modulus conclusion imply
$\omega_J(E_n,2\rho_n)=o_{P}(1)$ by bounding $2\rho_n$ by a fixed
$\delta$ with probability tending to one and then letting $\delta\downarrow0$.
The last term is negligible because $\sqrt{n_1}\rho_n=O_{P}(1)$
and $f$ is uniformly continuous on $J$. 

\paragraph{Asymptotics of the quantile effect.}  Apply \eqref{eq:q33:inverse-expansion} to $H_n=\hat F_z$.
For $z=0,1$, uniformly on $\cU$, we have
\begin{equation}\label{eq:q33:sampling-expansion}
 \sqrt{n_1}(\hat Q_z-Q_z)(u)
 =-\frac{\mathbb A_{z,n}\{Q_z(u)\}}
             {f_{z,T}\{Q_z(u)\}}+o_{P}(1).
\end{equation}
Define a fixed map on pairs $(g_1,g_0)$ of functions on $J_1,J_0$:
\[
 (\mathcal T_Q g)(u)
 =-\frac{g_1\{Q_1(u)\}}{f_{1,T}\{Q_1(u)\}}
   +\frac{g_0\{Q_0(u)\}}{f_{0,T}\{Q_0(u)\}}.
\]
It is Lipschitz under the maximum of the two sup norms, with constant
$C_Q=c_1^{-1}+c_0^{-1}$. Subtracting the two expansions and applying
Theorem~\ref{thm:wald_cdf} and the continuous mapping theorem proves
\[
 \sqrt{n_1}(\hat\Delta_Q-\Delta_Q)
 \Rightarrow\cZ_Q=\cT_Q(\BB_1,\BB_0)
 \quad\text{in }\ell^\infty(\cU),
\]
with the formula for $\cZ_Q$ stated in Theorem~\ref{thm:quantile}.
The process $\cZ_Q$ has continuous paths on $\cU$.

\paragraph{Bootstrap approximation for the quantile statistics.}
Apply \eqref{eq:q33:inverse-expansion} a second time, now to
$H_n=\hat F_z^*$. Its centered error about the same target CDF is
$E_n=\mathbb A_{z,n}+\mathbb A_{z,n}^*$, which satisfies the required
bounds by \eqref{eq:q33:two-processes}. Subtracting
\eqref{eq:q33:sampling-expansion} gives
\[
 \sqrt{n_1}(\hat Q_z^*-\hat Q_z)(u)
 =-\frac{\mathbb{A}_{z,n}^*\{Q_z(u)\}}
             {f_{z,T}\{Q_z(u)\}}+o_{P}(1)
\]
uniformly on $\cU$ under the joint data and resampling law.
Thus, with $\mathbb A_n^*=(\mathbb A_{1,n}^*,\mathbb A_{0,n}^*)$, we have 
\begin{equation}\label{eq:q33:bootstrap-expansion}
 S_Q^{*(1)}=\cT_Q\mathbb A_n^*+r_{Q,n}^*,\qquad
 \|r_{Q,n}^*\|_{\cU}=o_{P}(1).
\end{equation}
The remainder is conditionally negligible: for every
$\varepsilon,\eta>0$,
\begin{equation}\label{eq:q33:conditional-remainder}
 \PP\left\{\PP^*(\|r_{Q,n}^*\|_{\cU}>\varepsilon)
                         >\eta\right\}
 \le\eta^{-1}\mathbb P(\|r_{Q,n}^*\|_{\cU}>\varepsilon)
 \to 0.
\end{equation}

It remains to identify the conditional laws of the scalar statistics
used in the intervals. The two maps
\[
 g\mapsto\|\cT_Q g\|_{\cU},\qquad
 g\mapsto(\cT_Q g)(u),\quad u\in\cU\text{ fixed},
\]
are Lipschitz with constant $C_Q$. Applying the auxiliary i.i.d.
empirical-bootstrap theorem to these maps, and using the same
random-sample-size and actual--auxiliary comparisons as in Appendix~\ref{app:subsec_thm_wald_cdf},
gives their conditional CDF limits. Indeed, the actual--auxiliary
errors of either mapped statistic are at most $C_Q\|E_n^*\|_\infty$,
whose conditional exceedance probabilities vanish by
\eqref{eq:q33:cdf-bootstrap-comparison}. The conditional CDF sandwich
from Appendix~\ref{app:subsec_thm_wald_cdf} then transfers the auxiliary limits to the actual statistics. Finally, \eqref{eq:q33:conditional-remainder}
allows the same sandwich to remove $r_{Q,n}^*$.
Consequently,
\begin{align}
 \PP^*\{T_Q^{*(1)}\le t\}
   &\xrightarrow{ P}\PP(W_Q\le t), \quad 
      W_Q=\|\cZ_Q\|_{\cU},\notag\\
 \PP^*\{|S_Q^{*(1)}(u)|\le t\}
   &\xrightarrow{ P}\PP\{|\cZ_Q(u)|\le t\},
     \quad u\in\cU\text{ fixed},
 \label{eq:q33:scalar-bootstrap-laws}
\end{align}
at every continuity point of the respective limiting CDF.

\paragraph{Simultaneous and pointwise coverage.} 
We now prove the coverage results. 
Let $G_Q(t)=\mathbb P(W_Q\le t)$, and suppose it has a regular
$(1-\alpha)$-quantile $c_Q$. Write
\[
 G_{Q,n}^*(t)=\PP^*\{T_Q^{*(1)}\le t\},\qquad
 \hat G_{Q,n}^*(t)=\frac1{B_n}\sum_{b=1}^{B_n}
                       \ind\{T_Q^{*(b)}\le t\}.
\]
The same conditional Bernoulli calculation as in Appendix~\ref{app:subsec_thm_wald_cdf} gives 
 $\Var^*\{\hat G_{Q,n}^*(t)\}\le\frac1{4B_n}$. 
Since $B_n\to\infty$, \eqref{eq:q33:scalar-bootstrap-laws} and the
critical-value consistency argument from Appendix~\ref{app:subsec_thm_wald_cdf} imply
$c_Q^* \xrightarrow{P} c_Q$. The sampling limit gives
$T_{Q,n}:=\sqrt{n_1}\|\hat\Delta_Q-\Delta_Q\|_{\mathcal U}
\Rightarrow W_Q$. The same random-critical-value bracketing, using
continuity of $G_Q$ at $c_Q$ and $G_Q(c_Q)=1-\alpha$, therefore yields 
 $\mathbb P\{T_{Q,n}\le c_Q^*\}\to 1-\alpha$. 
On $\mathcal I_n$, this event is the coverage event, and since $\PP(\cI_n^c)=o(1)$, we prove~\eqref{dq:eq:quantile-coverage}.

Next fix $u\in\mathcal U$ with
$v_Q(u):=\operatorname{Var}\{\mathcal Z_Q(u)\}>0$ and put
$z_q=\Phi^{-1}(q)$. The second conclusion of
\eqref{eq:q33:scalar-bootstrap-laws} has the nondegenerate normal limit
$N\{0,v_Q(u)\}$. 
The quantiles of its absolute value at $q=\alpha$ are regular. The same empirical-quantile argument,
now applied at these two levels, gives the joint convergence  
 $|c_{Q}^*(u)|\xrightarrow{P}\sqrt{v_Q(u)}\,z_{1-\alpha}'$,
 where $z_q'$ is the $q$-th quantile of $|Z|$ for $Z\sim N(0,1)$. 
Meanwhile $S_{Q,n}(u):=\sqrt{n_1}\{\hat\Delta_Q(u)-\Delta_Q(u)\}
\Rightarrow N\{0,v_Q(u)\}$ by the sampling limit.
For the pointwise interval as defined in Section~\ref{subsec:wald_quantile}, on $\cI_n$ the coverage event is
\[
 \{\Delta_Q(u)\in\mathrm{CI}_{n,1-\alpha}^{Q,\mathrm{pt}}(u)\}
 =\{ |S_{Q,n}(u)|     \le c_{Q,1-\alpha}^*(u)\},
\]
and since $\PP(\cI_n^c)=o(1)$, we complete the proof of~\eqref{dq:eq:quantile-pointwise-coverage}.

%% file: 150_rand_inference.tex

\section{Proofs for Section~\ref{sec:rand_inference}}
\label{app:sec_rand_inference}

\subsection{Proof of Theorem~\ref{thm:rand_inf}}
\label{app:subsec_proof_rand_inf}

We use the sampling CLT and variance consistency in Theorem~\ref{thm:ATT}. We first
record a sign-sum bound that will also be used in the four-dimensional example.
All calculations are on $\mathcal I_n$ unless stated otherwise. Its complement
has probability tending to zero, so the conventions in Section~\ref{sec:rand_inference} do not change the conclusion.

\paragraph{A conditional normal approximation for the signs.}
We use the Berry--Esseen inequality for independent, not necessarily identically
distributed, summands \citep[Chapter~V]{petrov1972independent}: if $X_1,\ldots,X_m$
are independent and centered, $b^2=\sum_i\EE[ X_i^2]>0$, and
$\sum_i\EE |X_i|^3<\infty$, then
\begin{equation}\label{eq:rimean:BE-general}
\sup_{x\in\RR}
\big|\PP\big\{b^{-1}\textstyle{\sum_{i=1}^mX_i}\le x\big\}-\Phi(x)\big|
\le C_{\mathrm{BE}}\,b^{-3}\sum_{i=1}^m\EE[|X_i|^3],
\end{equation}
for a universal constant $C_{\mathrm{BE}}$.
Recall $\cG_n= \sigma(\{X_i,Z_i\}_{i=1}^{n})$ is the $\sigma$-field for the information used in matching. 
Define 
\begin{equation}\label{eq:rimean:raw-second-moment}
\hat v_n:=\frac1{n_1}\sum_{i=1}^{n_1}R_i^2
=\frac{n_1-1}{n_1}\hat V_{\mathrm{ATT}}
 +\hat\theta_{\mathrm{ATT}}^2.
\end{equation}
Conditional on the data, the variables $\omega_iR_i$ are independent and centered,
with variances $R_i^2$ and third absolute moments $|R_i|^3$. Thus, writing $\PP_\omega$ as the probability over the $\omega_i$'s conditional on the data, and 
\[
F_n^\omega(x)=\PP_\omega\left\{
 \frac{T_{\RI,n}^{\omega}}{\sqrt{\hat v_n}}\le x\right\},
\qquad
M_{kn}=\frac1{n_1}\sum_{i=1}^{n_1}|R_i|^k,\quad k=3,4 ,
\]
we obtain the finite-sample bound
\begin{equation}\label{eq:rimean:BE-sign}
\Delta_n^\omega:=\sup_x|F_n^\omega(x)-\Phi(x)|
\le \frac{C_{\mathrm{BE}}M_{3n}}{\sqrt{n_1}\hat v_n^{3/2}}.
\end{equation} 

Under Assumption~\ref{assump:outcome},  the conditional fourth moments of the observed pair
differences are uniformly bounded. Indeed, $\mu_0$ is bounded on $\Omega$, $\mu_1$ is bounded by assumption, and the conditional centered fourth
moments are bounded. The inequalities $|a+b|^4\le8(|a|^4+|b|^4)$ and
$|a-b|^4\le8(|a|^4+|b|^4)$ therefore give 
$\EE [M_{4n}\given \mathcal G_n]\le C$ on  $\cI_n$. 
Markov's inequality yields $M_{4n}=O_P(1)$, and
$M_{3n}\le M_{4n}^{3/4}=O_P(1)$ by H\"older's inequality.
Under $\theta_{\att}=0$, Theorem~\ref{thm:ATT} and
\eqref{eq:rimean:raw-second-moment} give
\begin{equation}\label{eq:rimean:positive-inputs}
T_{\RI,n}\Rightarrow N(0,V),\qquad
\hat v_n\to_P V>0,\qquad n_1\to\infty\ \text{in probability}.
\end{equation}
It follows from \eqref{eq:rimean:BE-sign} that
$\Delta_n^\omega=o_P(1)$. In particular, the event $\hat v_n=0$ has
probability tending to zero.

We consider the population-level p-value
\$
p_{\RI,n}^* = \PP_\omega(|T_{\RI,n}^\omega|\geq |T_{\RI,n}|),
\$
where the probability $\PP_\omega$ is over the independent randomization indices only, viewing data and $T_{\RI,n}$ as fixed. 
Put $U_n=T_{\RI,n}/\sqrt{\hat v_n}$ when $\hat v_n>0$ and
$U_n=0$ otherwise. For any $a>0$, we know 
\[
\PP_\omega\left\{
 |T_{\RI,n}^{\omega}|/\sqrt{\hat v_n}\ge a\right\}
=1-F_n^\omega(a-)+F_n^\omega(-a).
\]
Since $\Phi$ is continuous, the bound defining $\Delta_n^\omega$ holds
also for $F_n^\omega(a-)$. Consequently,
\begin{equation}\label{eq:rimean:p-normal}
\left|p_{\RI,n}^*-2\{1-\Phi(|U_n|)\}\right|
\le2\Delta_n^\omega
\end{equation}
on $\{\hat v_n>0\}$. The same conclusion is immediate at $U_n=0$, where
both tails equal one. Thus the inequality handles ties as well as random
observed cutoffs. By \eqref{eq:rimean:positive-inputs}, $U_n\Rightarrow N(0,1)$,
and hence 
$p_{\RI,n}^*\Rightarrow\operatorname{Unif}(0,1)$. 
For example, if $Z\sim N(0,1)$, then
$\PP[2\{1-\Phi(|Z|)\}\le\alpha]
=\PP(|Z|\ge z_{1-\alpha/2})=\alpha$. 
 
Conditional on the observed data, the count $\sum_{b=1}^{B_n}\ind\{|T_{\RI,n}^{\omega^{(b)}}| \ge |T_{\RI,n}|\}$ is binomial
with parameters $B_n,p_{\RI,n}^*$. Therefore, 
$\left|\EE_\omega p_{\RI,n}-p^*_{\RI,n}\right|
 \le\frac1{B_n+1}$ and
$\operatorname{Var}_\omega( p_{\RI,n})
 =\frac{B_np_{\RI,n}(1-p_{\RI,n})}{(B_n+1)^2}
\le\frac1{4B_n}$. 
For any fixed $\varepsilon>0$, once $1/(B_n+1)\le\varepsilon/2$,
conditional Chebyshev and averaging give
\begin{equation}\label{eq:rimean:MC-close}
\PP\{|  p_{\RI,n}-p^*_{\RI,n}|>\varepsilon\}
\le\frac1{B_n\varepsilon^2}\to 0.
\end{equation}
Here the probability includes the simulated signs. The uniform limiting law
of $p_{\RI,n}^*$ now gives the same rejection limit for
$p_{\RI,n}$ under the weak null. 
Since the weak null holds under the sharp null, the same conclusion holds under $H_{0,\text{sharp}}$ as well. This concludes the proof of Theorem~\ref{thm:rand_inf}.

\subsection{Proof of Proposition~\ref{prop:rand_boundary}}
\label{app:subsec_rand_dimension}

We keep the fixed covariate law for $X\in \RR^4$   and the function $g$ from the proof of Theorem~\ref{thm:dimension-obs}.
For constants $b_*,p_*>0$, that conclusion gives
\begin{equation}\label{eq:rimean:bias-input}
\liminf_{n\to\infty}\PP \left\{
 \sqrt{n_1}B_{n,\sigma_n^\star}(g)\ge b_*\right\}\ge p_*.
\end{equation} 
 
For completeness, the exact optimizer also satisfies
$L_n(\sigma_n^\star)=o_P(1)$ in this four-dimensional model. To see this,
fix an integer $K\ge1$ and partition $[0,1]^4$ into finitely many disjoint
cubes of side $1/K$. For each such cell $Q$, Assumption~\ref{assump:obs-sample} which the example satisfies  gives
\[
\PP (X\in Q,Z=0)-\PP (X\in Q,Z=1)
=\int_Q f(x)\{1-2e(x)\}\,dx>0.
\]
The law of large numbers, followed by a finite union bound, implies that every
cell has at least as many controls as treated units with probability tending
to one. On that event there is an injection using only within-cell pairs,
with average squared distance at most $4/K^2$. Since the matching objective is
$L_n/2$ for $A_0=I_4$, its optimizer satisfies the same upper bound.
For any fixed $\varepsilon>0$, choosing $K$ with $4/K^2<\varepsilon$ and
letting $n\to\infty$ proves the claimed length convergence.

Write $p(x)=1/4+g(x)/2$ so $Y(1)=Y(0)=\ind\{U\leq p(X)\}$. 
In \eqref{eq:rimean:binary-model}, we have $p(x)\in[1/4,3/4]$, and
$h(x):=p(x)\{1-p(x)\}\in[3/16,1/4]$. Both functions are smooth and
Lipschitz on $\Omega$. The construction gives
$\mu_0(x)=\mu_1(x)=p(x)$, $\sigma_0^2(x)=\sigma_1^2(x)=h(x)$,
bounded conditional fourth moments, and unconfoundedness. Thus it satisfies
the mean-outcome conditions of Assumption~\ref{assump:outcome}, while its dimension is four.
It also satisfies the sharp null by using the same uniform variable for both
potential outcomes of each unit.
Write
\[
d_i=p(T_i)-p(C_{\sigma_n^\star(i)}),\qquad
\eta_i=R_i-d_i,\qquad
a_n=\frac12\sqrt{n_1}B_{n,\sigma_n^\star}(g),\qquad
W_n=\frac1{\sqrt{n_1}}\sum_{i=1}^{n_1}\eta_i.
\]
Then we have the decomposition for the observed test statistic: 
$T_{\RI ,n}=a_n+W_n$,
where $a_n$ is $\cG_n$-measurable. Conditional on $\cG_n$,
the $\eta_i$'s are independent and centered due to the injection. They satisfy $|\eta_i|\le2$ and
\[
s_n^2:=\operatorname{Var}(W_n\mid\mathcal G_n)
=\frac1{n_1}\sum_{i=1}^{n_1}
 \{h(T_i)+h(C_{\sigma_n^\star(i)})\}\in[3/8,1/2].
\]
Let $\text{Lip}(h)<\infty$ be the Lipschitz constant of $h$ on $\Omega$, then by Cauchy-Schwarz inequality,
\[
\left|\frac1{n_1}\sum_{i=1}^{n_1}h(C_{\sigma_n^\star(i)}) -\frac1{n_1}\sum_{i=1}^{n_1} h(T_i)\right| 
\leq \bigg\{ \frac{1}{n_1}\sum_{i=1}^{n_1} \big\{ h(C_{\sigma_n^\star(i)}) - h(T_i)\big\}^2 \bigg\}^{1/2}
\le\operatorname{Lip}(h)\sqrt{L_n(\sigma_n^\star)}=o_P(1).
\]
Conditional on $n_1$, the treated covariates are i.i.d.~with the in-group distribution. Boundedness of $h$ and conditional Chebyshev, averaged over
$n_1\to\infty$, therefore give
\begin{equation}\label{eq:rimean:noise-variance}
s_n^2\to_P v_*:=2\EE_1 [h(X)]\in[3/8,1/2].
\end{equation}

Apply \eqref{eq:rimean:BE-general} conditionally on $\mathcal G_n$ to
$\eta_1,\ldots,\eta_{n_1}$. Their total third absolute moment is at most
$8n_1$, and their total variance is $n_1s_n^2$. Consequently,
\[
\sup_x\left|\PP (W_n\le x\mid\mathcal G_n)
                 -\Phi(x/s_n)\right|
\leq{C}/{\sqrt{n_1}}.
\]
To replace $s_n$ by $\sqrt{v_*}$ uniformly in $x$, use the deterministic
inequality
\begin{equation}\label{eq:rimean:normal-scale}
\sup_x|\Phi(x/a)-\Phi(x/b)|
\le C_\Phi|\log(a/b)|,\quad a,b>0,
\qquad C_\Phi=\sup_t|t\phi(t)|<\infty.
\end{equation}
It follows by differentiating $\Phi(xe^{-u})$ with respect to $u$ and
integrating between $\log a$ and $\log b$. Together with
\eqref{eq:rimean:noise-variance}, this proves
\begin{equation}\label{eq:rimean:noise-normal}
\epsilon_n:=\sup_x\left|
\PP(W_n\le x\given \cG_n)-\Phi(x/\sqrt{v_*})\right|
\to_P0,\qquad \EE[\epsilon_n]\to0.
\end{equation}
On $\mathcal I_n^c$, set $\epsilon_n=1$. The expectation converges because
$0\le\epsilon_n\le1$. The uniform bound also controls left limits of the
conditional CDF and applies at cutoffs depending on $a_n$.

Since $p(x)$ is Lipschitz, letting $\text{Lip}(p)$ be its Lipschitz constant, we have
\[
\frac1{n_1}\sum_i d_i^2
\le\operatorname{Lip}(p)^2L_n(\sigma_n^\star)=o_P(1).
\]
Moreover, conditional independence and $0\le R_i^2\le1$ give
\[
\EE (\hat v_n\mid\mathcal G_n)
=s_n^2+\frac1{n_1}\sum_i d_i^2,
\qquad
\operatorname{Var}(\hat v_n\mid\mathcal G_n)\le\frac1{n_1}.
\]
Conditional Chebyshev and averaging therefore show
\begin{equation}\label{eq:rimean:raw-fourdim}
\hat v_n\to_P v_*.
\end{equation}
For binary outcomes, $R_i\in\{-1,0,1\}$, so $M_{3n}=\hat v_n$.
Equation~\eqref{eq:rimean:BE-sign} and
\eqref{eq:rimean:raw-fourdim} thus give a vanishing normal-approximation
error for the conditional sign sum. Using \eqref{eq:rimean:normal-scale}
once more, we have 
\[
\sup_x\left|\PP _\omega(T_{\RI ,n}^{\omega}\le x)
                      -\Phi(x/\sqrt{v_*})\right|=o_P(1).
\]
Using the same arguments as in the proof of Theorem~\ref{thm:rand_inf} up to~\eqref{eq:rimean:p-normal} then
gives (as $B_n\to \infty$)
\begin{equation}\label{eq:rimean:fourdim-p}
p_{\RI ,n}
=\ell(|T_{\RI ,n}|)+o_P(1),
\qquad
\ell(t)=2\{1-\Phi(t/\sqrt{v_*})\},\quad t\ge0.
\end{equation}   
 
Now we convert the geometric lower bound into the type-I error violation. 
For a fixed threshold $t>0$ and a fixed $a\in \RR$, define
\[
\cR_t(a)
=1-\Phi\left(\frac{t-a}{\sqrt{v_*}}\right)
 +\Phi\left(\frac{-t-a}{\sqrt{v_*}}\right).
\]
This is the probability that the absolute value of a $N(a,v_*)$ variable
is at least $t$. It is even in $a$, $\cR_t(0)=\ell(t)$, and for
$a>0$,
\[
\cR_t'(a)
=\frac1{\sqrt{v_*}}\left\{
 \phi\left(\frac{t-a}{\sqrt{v_*}}\right)
 -\phi\left(\frac{t+a}{\sqrt{v_*}}\right)\right\}>0.
\]
The strict inequality follows from $|t-a|<t+a$.

By \eqref{eq:rimean:noise-normal}, conditional on $\mathcal G_n$ the
probability of $|a_n+W_n|\ge t$ differs from $\cR_t(a_n)$ by at
most $2\epsilon_n$. Averaging, and allowing for $\mathcal I_n^c$, gives 
$\PP (|T_{\RI ,n}|\ge t)
=\EE [\cR_t(a_n)]+o(1)$. 
Let $a_0=b_*/2$. Equation~\eqref{eq:rimean:bias-input} implies
$\liminf_n\PP (a_n\ge a_0)\ge p_*$. Monotonicity and evenness give,
for every real $a$,
\[
\cR_t(a)\ge\cR_t(0)
+\{\cR_t(a_0)-\cR_t(0)\}\ind\{a\ge a_0\}.
\]
Consequently,
\begin{equation}\label{eq:rimean:tail-lower}
\liminf_n\PP (|T_{\RI ,n}|\ge t)
\ge\ell(t)+p_*\{\cR_t(a_0)-\ell(t)\}.
\end{equation} 

Fix $\alpha\in(0,1)$ and let
$t_\alpha=\sqrt{v_*}\,z_{1-\alpha/2}>0$, so $\ell(t_\alpha)=\alpha$. For any fixed $t>t_\alpha$,
$\alpha-\ell(t)>0$, and
\[
\PP (p_{\RI ,n}\le\alpha) \ge\PP (|T_{\RI ,n}|\ge t)
 -\PP \{|p_{\RI ,n}-\ell(|T_{\RI ,n}|)|  >\alpha-\ell(t)\}.
\]
The last term tends to zero. Apply \eqref{eq:rimean:tail-lower} and then
let $t\downarrow t_\alpha$. Continuity in $t$ gives
\begin{equation}\label{eq:rimean:strict-overrejection}
\liminf_n\PP (p_{\RI ,n}\le\alpha)
\ge\alpha+p_*\{\cR_{t_\alpha}(a_0)-\alpha\}
>\alpha.
\end{equation}
The final inequality follows from $a_0>0$ and strict monotonicity. Taking
$\epsilon_\alpha=p_*\{\cR_{t_\alpha}(b_*/2)-\alpha\}>0$
proves Proposition~\ref{prop:rand_boundary}.

%% file: 161_additional_simulation.tex

\section{Additional simulation results}
\label{app:sec_simu}


\begin{table}[p]
\centering
\caption{Population ATT inference with $a=0$. Details are otherwise the same as Table~\ref{tab:att-a2}.}
\label{tab:att-a0}
\small\setlength{\tabcolsep}{6pt}\renewcommand{\arraystretch}{1.08}
\begin{tabular}{rrlrrrr}
\toprule
&&&\multicolumn{2}{c}{Coverage (\%)}&\multicolumn{2}{c}{Mean CI width}\\
\cmidrule(lr){4-5}\cmidrule(lr){6-7}
$n$ & $d$ & Matching & Wald & Bootstrap-t & Wald & Bootstrap-t \\
\midrule
500 & 1 & Global & 95.50 (0.46) & 95.30 (0.47) & 0.4020 (0.0007) & 0.4054 (0.0008)\\
 &  & Greedy & 95.15 (0.48) & 95.15 (0.48) & 0.4020 (0.0007) & 0.4052 (0.0008)\\
\addlinespace[2pt]
 & 2 & Global & 94.65 (0.50) & 94.85 (0.49) & 0.4034 (0.0007) & 0.4074 (0.0008)\\
 &  & Greedy & 94.80 (0.50) & 94.95 (0.49) & 0.4039 (0.0007) & 0.4079 (0.0008)\\
\addlinespace[2pt]
 & 3 & Global & 94.00 (0.53) & 94.15 (0.52) & 0.4089 (0.0007) & 0.4126 (0.0008)\\
 &  & Greedy & 93.25 (0.56) & 93.60 (0.55) & 0.4095 (0.0007) & 0.4132 (0.0008)\\
\addlinespace[2pt]
 & 4 & Global & 90.25 (0.66) & 90.35 (0.66) & 0.4157 (0.0007) & 0.4197 (0.0008)\\
 &  & Greedy & 90.80 (0.65) & 90.55 (0.65) & 0.4173 (0.0007) & 0.4208 (0.0008)\\
\addlinespace[2pt]
 & 6 & Global & 82.25 (0.85) & 83.10 (0.84) & 0.4333 (0.0007) & 0.4373 (0.0009)\\
 &  & Greedy & 83.15 (0.84) & 83.00 (0.84) & 0.4352 (0.0007) & 0.4388 (0.0009)\\
\midrule
2000 & 1 & Global & 94.00 (0.53) & 93.85 (0.54) & 0.2012 (0.0002) & 0.2014 (0.0003)\\
 &  & Greedy & 93.95 (0.53) & 93.80 (0.54) & 0.2012 (0.0002) & 0.2013 (0.0003)\\
\addlinespace[2pt]
 & 2 & Global & 94.75 (0.50) & 94.50 (0.51) & 0.2012 (0.0002) & 0.2013 (0.0003)\\
 &  & Greedy & 94.40 (0.51) & 94.20 (0.52) & 0.2015 (0.0002) & 0.2017 (0.0003)\\
\addlinespace[2pt]
 & 3 & Global & 94.50 (0.51) & 94.65 (0.50) & 0.2025 (0.0002) & 0.2024 (0.0003)\\
 &  & Greedy & 94.50 (0.51) & 94.30 (0.52) & 0.2028 (0.0002) & 0.2029 (0.0003)\\
\addlinespace[2pt]
 & 4 & Global & 90.15 (0.67) & 90.20 (0.66) & 0.2046 (0.0002) & 0.2050 (0.0003)\\
 &  & Greedy & 90.15 (0.67) & 90.35 (0.66) & 0.2051 (0.0002) & 0.2053 (0.0003)\\
\addlinespace[2pt]
 & 6 & Global & 72.40 (1.00) & 72.50 (1.00) & 0.2108 (0.0002) & 0.2108 (0.0003)\\
 &  & Greedy & 75.10 (0.97) & 75.05 (0.97) & 0.2114 (0.0002) & 0.2115 (0.0003)\\
\midrule
5000 & 1 & Global & 95.40 (0.47) & 95.20 (0.48) & 0.1272 (0.0001) & 0.1272 (0.0001)\\
 &  & Greedy & 95.00 (0.49) & 94.95 (0.49) & 0.1272 (0.0001) & 0.1272 (0.0001)\\
\addlinespace[2pt]
 & 2 & Global & 94.85 (0.49) & 94.50 (0.51) & 0.1274 (0.0001) & 0.1272 (0.0001)\\
 &  & Greedy & 94.70 (0.50) & 94.30 (0.52) & 0.1274 (0.0001) & 0.1272 (0.0001)\\
\addlinespace[2pt]
 & 3 & Global & 93.95 (0.53) & 93.65 (0.55) & 0.1277 (0.0001) & 0.1276 (0.0001)\\
 &  & Greedy & 93.80 (0.54) & 93.30 (0.56) & 0.1278 (0.0001) & 0.1277 (0.0001)\\
\addlinespace[2pt]
 & 4 & Global & 90.00 (0.67) & 89.80 (0.68) & 0.1286 (0.0001) & 0.1283 (0.0001)\\
 &  & Greedy & 90.75 (0.65) & 90.60 (0.65) & 0.1287 (0.0001) & 0.1287 (0.0001)\\
\addlinespace[2pt]
 & 6 & Global & 65.95 (1.06) & 65.70 (1.06) & 0.1315 (0.0001) & 0.1314 (0.0001)\\
 &  & Greedy & 67.30 (1.05) & 66.75 (1.05) & 0.1317 (0.0001) & 0.1316 (0.0001)\\
\bottomrule\end{tabular}
\end{table}

\begin{table}[!htbp]
\centering
\small
\setlength{\tabcolsep}{5pt}
\renewcommand{\arraystretch}{1.08}
\caption{Coverage and width of pointwise bootstrap confidence intervals for quantile treatment effects at $u\in\{0.25, 0.5, 0.75\}$ and nominal level $1-\alpha=95\%$, where $B_n=500$. The mean and SEs are reported over $R=2000$ independent runs.}
\label{tab:pointwise-qte}
\begin{tabular}{rrlcccccc}
\toprule
&&&\multicolumn{2}{c}{$u=0.25$}&\multicolumn{2}{c}{$u=0.50$}&\multicolumn{2}{c}{$u=0.75$}\\
\cmidrule(lr){4-5}\cmidrule(lr){6-7}\cmidrule(lr){8-9}
$n$&$d$&Matching&Coverage (\%)&CI width&Coverage (\%)&CI width&Coverage (\%)&CI width\\
\midrule
500 & 1 & Global & 96.30 (0.42) & 0.9567 & 95.90 (0.44) & 0.8878 & 96.95 (0.38) & 0.9737 \\
 &  & Greedy & 96.90 (0.39) & 0.9578 & 96.50 (0.41) & 0.8891 & 96.90 (0.39) & 0.9728 \\
\addlinespace[2pt]
 & 2 & Global & 97.00 (0.38) & 0.9262 & 96.35 (0.42) & 0.8386 & 96.20 (0.43) & 0.9369 \\
 &  & Greedy & 96.90 (0.39) & 0.9274 & 96.40 (0.42) & 0.8391 & 95.90 (0.44) & 0.9375 \\
\addlinespace[2pt]
 & 3 & Global & 96.55 (0.41) & 0.9243 & 95.10 (0.48) & 0.8338 & 95.05 (0.49) & 0.9351 \\
 &  & Greedy & 96.30 (0.42) & 0.9290 & 95.10 (0.48) & 0.8357 & 95.25 (0.48) & 0.9385 \\
\addlinespace[2pt]
 & 4 & Global & 95.20 (0.48) & 0.9275 & 93.60 (0.55) & 0.8384 & 93.20 (0.56) & 0.9450 \\
 &  & Greedy & 95.55 (0.46) & 0.9272 & 93.80 (0.54) & 0.8382 & 93.05 (0.57) & 0.9481 \\
\addlinespace[2pt]
 & 6 & Global & 90.95 (0.64) & 0.9447 & 87.40 (0.74) & 0.8474 & 86.85 (0.76) & 0.9673 \\
 &  & Greedy & 91.05 (0.64) & 0.9431 & 86.60 (0.76) & 0.8488 & 86.55 (0.76) & 0.9666 \\
\midrule
2000 & 1 & Global & 96.55 (0.41) & 0.4581 & 95.50 (0.46) & 0.4268 & 96.20 (0.43) & 0.4689 \\
 &  & Greedy & 96.15 (0.43) & 0.4577 & 95.35 (0.47) & 0.4269 & 96.30 (0.42) & 0.4684 \\
\addlinespace[2pt]
 & 2 & Global & 96.20 (0.43) & 0.4449 & 96.35 (0.42) & 0.4028 & 95.75 (0.45) & 0.4463 \\
 &  & Greedy & 96.00 (0.44) & 0.4449 & 96.15 (0.43) & 0.4029 & 95.80 (0.45) & 0.4463 \\
\addlinespace[2pt]
 & 3 & Global & 94.65 (0.50) & 0.4405 & 95.05 (0.49) & 0.3973 & 94.65 (0.50) & 0.4445 \\
 &  & Greedy & 95.30 (0.47) & 0.4405 & 94.50 (0.51) & 0.3986 & 95.25 (0.48) & 0.4454 \\
\addlinespace[2pt]
 & 4 & Global & 94.85 (0.49) & 0.4417 & 93.15 (0.56) & 0.3986 & 92.25 (0.60) & 0.4459 \\
 &  & Greedy & 94.30 (0.52) & 0.4416 & 93.55 (0.55) & 0.3993 & 93.15 (0.56) & 0.4475 \\
\addlinespace[2pt]
 & 6 & Global & 83.45 (0.83) & 0.4414 & 73.20 (0.99) & 0.4050 & 75.85 (0.96) & 0.4535 \\
 &  & Greedy & 83.95 (0.82) & 0.4420 & 74.40 (0.98) & 0.4047 & 76.65 (0.95) & 0.4537 \\
\midrule
5000 & 1 & Global & 94.95 (0.49) & 0.2826 & 96.05 (0.44) & 0.2643 & 95.90 (0.44) & 0.2884 \\
 &  & Greedy & 95.30 (0.47) & 0.2827 & 96.05 (0.44) & 0.2643 & 95.95 (0.44) & 0.2886 \\
\addlinespace[2pt]
 & 2 & Global & 95.80 (0.45) & 0.2763 & 94.45 (0.51) & 0.2496 & 94.65 (0.50) & 0.2785 \\
 &  & Greedy & 96.25 (0.42) & 0.2765 & 94.45 (0.51) & 0.2500 & 94.90 (0.49) & 0.2789 \\
\addlinespace[2pt]
 & 3 & Global & 94.65 (0.50) & 0.2748 & 94.25 (0.52) & 0.2470 & 93.75 (0.54) & 0.2755 \\
 &  & Greedy & 95.10 (0.48) & 0.2746 & 94.50 (0.51) & 0.2474 & 94.50 (0.51) & 0.2757 \\
\addlinespace[2pt]
 & 4 & Global & 93.15 (0.56) & 0.2724 & 90.15 (0.67) & 0.2461 & 91.60 (0.62) & 0.2763 \\
 &  & Greedy & 93.60 (0.55) & 0.2724 & 91.00 (0.64) & 0.2461 & 91.75 (0.62) & 0.2761 \\
\addlinespace[2pt]
 & 6 & Global & 78.60 (0.92) & 0.2741 & 64.70 (1.07) & 0.2485 & 65.45 (1.06) & 0.2789 \\
 &  & Greedy & 78.90 (0.91) & 0.2743 & 65.25 (1.06) & 0.2487 & 67.35 (1.05) & 0.2787 \\
\bottomrule
\end{tabular} 
\end{table}

\begin{table} 
\centering
\small
\setlength{\tabcolsep}{4.5pt}
\renewcommand{\arraystretch}{1.08}
\caption{Randomization inference under the unitwise sharp null: $Y_i(1)=Y_i(0)$. The rejection rates and its Monte Carlo SE (parentheses) are in percentage, averaged over $R=2000$ independent runs.}
\label{tab:ri-sharp-shared-alphas}
\begin{tabular}{@{}rr*{6}{r}@{}}
\toprule
 & & \multicolumn{3}{c}{Global} & \multicolumn{3}{c}{Greedy} \\
\cmidrule(lr){3-5}\cmidrule(lr){6-8}
$n$ & $d$ & $\alpha=0.025$ & $\alpha=0.05$ & $\alpha=0.10$ & $\alpha=0.025$ & $\alpha=0.05$ & $\alpha=0.10$ \\
\midrule
500 & 1 & 2.55 (0.35) & 5.10 (0.49) & 10.05 (0.67) & 2.40 (0.34) & 4.90 (0.48) & 9.90 (0.67) \\
500 & 2 & 3.00 (0.38) & 5.20 (0.50) & 10.25 (0.68) & 3.00 (0.38) & 5.45 (0.51) & 10.55 (0.69) \\
500 & 3 & 3.15 (0.39) & 6.30 (0.54) & 11.75 (0.72) & 3.15 (0.39) & 6.10 (0.54) & 11.95 (0.73) \\
500 & 4 & 6.00 (0.53) & 9.25 (0.65) & 16.05 (0.82) & 5.10 (0.49) & 9.40 (0.65) & 15.65 (0.81) \\
500 & 6 & 11.45 (0.71) & 17.75 (0.85) & 26.80 (0.99) & 11.05 (0.70) & 17.70 (0.85) & 26.20 (0.98) \\
\addlinespace
2000 & 1 & 2.10 (0.32) & 4.90 (0.48) & 9.60 (0.66) & 2.15 (0.32) & 4.75 (0.48) & 9.80 (0.66) \\
2000 & 2 & 3.30 (0.40) & 5.65 (0.52) & 10.55 (0.69) & 2.75 (0.37) & 5.40 (0.51) & 10.60 (0.69) \\
2000 & 3 & 3.05 (0.38) & 5.25 (0.50) & 9.60 (0.66) & 2.95 (0.38) & 5.00 (0.49) & 9.80 (0.66) \\
2000 & 4 & 4.55 (0.47) & 8.10 (0.61) & 14.45 (0.79) & 4.65 (0.47) & 7.45 (0.59) & 14.80 (0.79) \\
2000 & 6 & 16.80 (0.84) & 26.10 (0.98) & 38.35 (1.09) & 16.20 (0.82) & 24.75 (0.96) & 36.40 (1.08) \\
\addlinespace
5000 & 1 & 2.85 (0.37) & 6.15 (0.54) & 11.65 (0.72) & 2.95 (0.38) & 6.15 (0.54) & 11.90 (0.72) \\
5000 & 2 & 2.50 (0.35) & 5.45 (0.51) & 10.90 (0.70) & 2.95 (0.38) & 5.30 (0.50) & 11.55 (0.71) \\
5000 & 3 & 2.20 (0.33) & 4.70 (0.47) & 10.00 (0.67) & 2.30 (0.34) & 5.20 (0.50) & 10.15 (0.68) \\
5000 & 4 & 4.65 (0.47) & 8.00 (0.61) & 14.50 (0.79) & 5.05 (0.49) & 8.50 (0.62) & 14.40 (0.79) \\
5000 & 6 & 25.95 (0.98) & 36.10 (1.07) & 47.75 (1.12) & 25.20 (0.97) & 34.25 (1.06) & 46.45 (1.12) \\
\bottomrule
\end{tabular}
\end{table}